\documentclass[a4paper]{article}

\usepackage[top=1.2in, bottom=1.5in, left=1.1in, right=1.8in]{geometry}
\usepackage{amsfonts}
\usepackage{amsmath}
\usepackage{amsthm,amssymb}
\usepackage{CJK,graphicx}
\usepackage{amscd}
\usepackage{amssymb}
\usepackage{mathrsfs}
\usepackage[all,cmtip]{xy}
\usepackage{lmodern}
\usepackage[Symbol]{upgreek}
\usepackage{bm}
\usepackage{eucal}
\usepackage[nopar]{lipsum}
\usepackage{rotating}
\usepackage{tikz}
\usepackage{calc}
\usepackage{tikz-cd}
\usepackage{mathabx}
\usepackage{tikz-3dplot}
\usepackage{hyperref}
\usetikzlibrary{calc}
\usetikzlibrary{decorations.markings,arrows}
\usetikzlibrary{shapes.geometric}
\usepackage{lipsum}
\usepackage{makecell}
\usepackage{times}

\newsavebox\CBox
\newcommand\hcancel[2][0.5pt]{%
	\ifmmode\sbox\CBox{$#2$}\else\sbox\CBox{#2}\fi%
	\makebox[0pt][l]{\usebox\CBox}%
	\rule[0.5\ht\CBox-#1/2]{\wd\CBox}{#1}}

\newtheorem{theorem}{Theorem}
\newtheorem{corollary}[theorem]{Corollary}
\newtheorem{definition}[theorem]{Definition}
\newtheorem{lemma}[theorem]{Lemma}
\newtheorem{proposition}[theorem]{Proposition}

\newtheorem{conjecture}[theorem]{Conjecture}

\newtheorem{remark}[theorem]{Remark}

\numberwithin{equation}{section}

\begin{document}
	
\title{\textbf{Lagrangian capacity and chain level string topology}}\author{Shah Faisal \and Yin Li}
\newcommand{\Addresses}{{
		\bigskip
		\footnotesize
		Shah Faisal, \textsc{Department of Mathematics, Uppsala University, 753 10 Uppsala, Sweden}\par\nopagebreak\textit{E-mail address}: \texttt{shah.faisal@math.uu.se}
		
		\medskip
		Yin Li, \textsc{Department of Mathematics, Uppsala University, 753 10 Uppsala, Sweden}\par\nopagebreak
		\textit{E-mail address}: \texttt{yin.li@math.uu.se}
}}
\date{}\maketitle

\begin{abstract}
We derive upper bounds for the Lagrangian capacities of Liouville domains with finite Gutt--Hutchings capacities and show that the Lagrangian capacity of a convex or concave toric domain of arbitrary dimension equals its diagonal. In particular, this completely settles the conjecture of Cieliebak--Mohnke on the Lagrangian capacity of ellipsoids. Our proof is based on an $S^1$-equivariant variant of the techniques of Fukaya and Irie, and does not use holomorphic curves with local tangency constraints, which would inevitably cause transversality issues. Moreover, we show that any extremal Lagrangian torus in a $2n$-dimensional ellipsoid must lie on the boundary, verifying a conjecture of \cite{sfe}. Applications of our results and techniques include a proof of the fact that every monotone Lagrangian torus in the complex projective space has non-trivial intersections with the Clifford torus and the first computations of the Lagrangian capacities for many non-subcritical Weinstein domains in dimensions $4$ and $6$.
\end{abstract}
	
\section{Introduction}

\subsection{Background}

A symplectic manifold is a pair $(X,\omega)$, where $X$ is a smooth $2n$-dimensional manifold and $\omega$ is a closed, non-degenerate differential $2$-form. Non-degeneracy of $\omega$ means that the top-degree form $\omega^{n}$ is nowhere vanishing on $X$. In this article, we will focus on a particular class of symplectic manifolds, called \textit{Liouville domains}. By a Liouville domain we mean a compact oriented smooth manifold $X$ with nonempty boundary $\partial X$ equipped with a $1$-form $\lambda$ such that $d\lambda$ is a symplectic form on $X$. The restriction $\lambda|_{\partial X}$ is a contact form on $\partial X$, and the orientation induced by $\lambda$ on $\partial X$ agrees with the boundary orientation.

Elementary examples of Liouville domains arise from star-shaped domains $(X,\lambda_{\mathrm{std}})$ with smooth boundary in the standard symplectic vector space $(\mathbb{R}^{2n},\omega_{\mathrm{std}}:=d\lambda_{\mathrm{std}})$, where
\begin{equation}\label{eq:std}
\lambda_{\mathrm{std}}:=\frac{1}{2}\sum_{i=1}^{n}(x_idy_i-y_idx_i).
\end{equation}
Ellipsoids provide a particularly important class of smooth star-shaped domains. Given real numbers 
$0 < a_1 \le a_2 \le \cdots \le a_n < \infty$, 
the associated $2n$-dimensional ellipsoid is defined by
\[
E^{2n}(a_1,\dots,a_n):= \left\{ (z_1,\dots,z_n) \in \mathbb{C}^n \left\vert\sum_{i=1}^{n} \frac{\pi |z_i|^2}{a_i} \le 1 \right.\right\}.
\]
In particular,
\begin{equation}\label{eq:ball}
B^{2n}(a) := E^{2n}(a,\dots,a)
\end{equation}
is the standard symplectic ball of capacity $a>0$, or equivalently, with radius $\sqrt{a/\pi}$.

In this paper, we establish a relation between the \textit{Gutt--Hutchings capacities} \cite{ghs} and the \textit{Cieliebak--Mohnke capacity} \cite{cmp} for Liouville domains---the latter is also known as the \textit{Lagrangian capacity}. To define the latter, let $L \subset (X,\omega)$ be a closed Lagrangian submanifold. We define the symplectic area of $L$ by
\[
A_{\min}(L):=\inf_{\substack{A \in \pi_2(X,L) \\ \int_A \omega > 0}} 
\int_A \omega \in [0,\infty].
\]

\begin{definition}\label{definition:CM}
The Cieliebak--Mohnke capacity of a symplectic manifold $(X,\omega)$, denoted 
$C^{\mathrm{CM}}(X,\omega)$, is defined as
\[
C^{\mathrm{CM}}(X,\omega):= \sup_{\substack{L \subset (X,\omega) \\ \textrm{Lagrangian torus}}}A_{\min}(L)\;\in\; [0,\infty].
\]
\end{definition}

\begin{definition}\label{CMextremal}
A Lagrangian torus $L \subset (X,\omega)$ is called extremal if
\[
A_{\min}(L)=
	C^{\mathrm{CM}}(X,\omega).
	\]
\end{definition}

We recall some conjectures concerning the Cieliebak--Mohnke capacity, which we will study in this paper and subsequent works.

\begin{conjecture}[\cite{cmp}, Conjecture 1.5]\label{CMcapaciy-conjecture}
The Cieliebak--Mohnke capacity of ellipsoids are given by
\[
C^{\mathrm{CM}}\left(E^{2n}(a_1,\dots,a_n)\right)=\left(\frac{1}{a_1} + \frac{1}{a_2} + \cdots + \frac{1}{a_n}\right)^{-1}.
\]
\end{conjecture}

\begin{conjecture}[\cite{cmp}, Conjecture 1.9]\label{extremal-ball-conjecture}
Every extremal Lagrangian torus in the symplectic unit ball $\left(B^{2n}(1),\omega_{\mathrm{std}}\right)$ is contained entirely in the boundary $\partial B^{2n}(1)$.
\end{conjecture}

\begin{conjecture}[\cite{cmp}, Conjecture 1.8]\label{monotone-extremal-conjecture}
Let $\omega_{\mathrm{FS}}$ denote the Fubini--Study form on $\mathbb{CP}^n$. A Lagrangian torus in $(\mathbb{CP}^n,\omega_{\mathrm{FS}})$ is monotone if and only if it is extremal.
\end{conjecture}

This conjecture is of particular interest to us, as it suggests that extremal Lagrangian tori may provide a natural replacement for monotone Lagrangian tori, which only exist in monotone symplectic manifolds, but we will not prove it in this paper.

\begin{conjecture}[\cite{sfe}, Conjecture 9.11]\label{extremal-ellip-conjecture}
For any $0<a_1\le a_2\le \cdots \le a_n<\infty$, every extremal Lagrangian torus in $\left(E^{2n}(a_1,\dots,a_n),\omega_{\mathrm{std}}\right)$ is contained entirely in the boundary $\partial E^{2n}(a_1,\dots,a_n)$.
\end{conjecture}

The next conjecture predicts that Lagrangian embedding problems are governed by the same numerical obstruction as symplectic embedding problems.

\begin{conjecture}[\cite{chls}, Conjecture 2]\label{realize-CMcapacity}
For $a \ge 1$, define
\[
\operatorname{EC}(a):=\inf\left\{ r>0\left\vert E^4(1,a)\hookrightarrow\left(B^4(r)\cup B^2(r/2)\times\mathbb{C},\omega_{\mathrm{std}}\right)\right\}.\right.
\]
Then
\[
\operatorname{EC}(a)=2\, C^{\mathrm{CM}}\left(E^{4}(1,a),\omega_{\mathrm{std}}\right).
\]
\end{conjecture}

We summarize below the progresses so far on the conjectures stated above.

\begin{itemize}
	\item Cieliebak--Mohnke (2014) (\cite{cmp}, Corollary 1.3) proved Conjecture \ref{CMcapaciy-conjecture} for all balls and cylinders, i.e., for
	\[E^{2n}(a,a ,\dots, a)\text{ and } E^{2n}(a,\infty,\dots, \infty).\]
	Moreover, they established that every monotone Lagrangian torus in $(\mathbb{CP}^n,\omega_{\mathrm{FS}})$ is extremal (cf. \cite{cmp}, Corollary 1.7), which proves one direction of Conjecture \ref{monotone-extremal-conjecture}.
	\item Dimitroglou-Rizell (2015) \cite{gdu} gave a proof of Conjecture \ref{extremal-ball-conjecture} for the $4$-dimensional ball $B^{4}(1)$. His proof is essentially based on holomorphic curve techniques in dimension $4$, e.g. positivity of the intersection and automatic transversality. 
	\item Pereira in his PhD thesis (2022, cf. \cite{mpo}, Theorem 3.28) confirmed Conjecture \ref{CMcapaciy-conjecture} for all $4$-dimensional ellipsoids. In higher dimensions, he shows that the conjecture (cf. \cite{mpo}, Theorem 4.37) holds under the assumption that a suitable virtual perturbation scheme exists to define curve counts for linearized contact homology and moduli spaces with local tangency constraints, which are used to construct augmentations on the linearized contact homology algebra of ellipsoids (cf. Siegel \cite{ksh}). Under the same assumption for the perturbation scheme, a proof of Conjecture \ref{realize-CMcapacity} would also follow from \cite{mpo}.
	\item The first author in his PhD thesis (2025) proves Conjecture \ref{extremal-ball-conjecture} in all dimensions (cf. \cite{sfe}, Theorem 1.7). Using similar ideas, he also confirms Conjecture  \ref{extremal-ellip-conjecture} for a class of toric domains including all ellipsoids $E^4(a, b)$ and cylinders $B^{2k}(1)\times \mathbb{C}^m$ for arbitrary $k, m\in \mathbb{N}$ (cf. \cite{sfe}, Theorem 1.11 and Theorem 1.8). Moreover, he proves Conjecture \ref{monotone-extremal-conjecture} for $\mathbb{CP}^2$ (cf. \cite{sfe}, Theorem 1.18).
\end{itemize}

In this paper, we introduce a new symplectic capacity using all aspherical Lagrangian submanifolds, extending the Cieliebak--Mohnke capacity. Using the techniques from Fukaya--Irie \cite{kfa,kic,kie} and the second author \cite{yla}, which are based on chain level string topology, we establish an upper bound on the symplectic area of closed oriented aspherical Lagrangian submanifolds in a Liouville domain in terms of its Gutt--Hutchings capacities. We prove that this bound is sharp for the new capacity on convex and concave toric domains in $\mathbb{R}^{2n}$, thereby completely settle the Conjectures \ref{CMcapaciy-conjecture} and \ref{realize-CMcapacity}. Furthermore, under an additional assumption on the asymptotic behavior of the Gutt--Hutchings capacities, we obtain a boundary rigidity result for extremal aspherical Lagrangians in Liouville domains, which is strong enough to confirm Conjecture \ref{extremal-ellip-conjecture}.

A notable feature of our approach is that it works without any putative perturbative schemes, which is one of the main novelties of this paper.

\subsection{Summary of results}

Before stating the main results of this paper, we fix some notations and introduce a variant of the Cieliebak--Mohnke capacity using aspherical Lagrangian submanifolds.

Let $X$ be a Liouville domain with $c_1(X)=0$. Denote by $\mathit{SH}^\ast(X)$ the symplectic cohomology of $X$, and by $\mathit{SH}^\ast_{S^1}(X)$ its $S^1$-equivariant version, which are both $\mathbb{Z}$-graded vector spaces over some field $\mathbb{K}$. For the purpose of this paper, we shall take $\mathbb{K}=\mathbb{R}$ to be the field of real numbers.

Recall that the $S^1$-equivariant symplectic cohomology $\mathit{SH}^\ast_{S^1}(X)$ is the cohomology of the complex
\begin{equation}\label{eq:SC}
\left(\mathit{SC}^\ast_{S^1}(X):=\mathit{SC}^\ast(X)\otimes_\mathbb{R}\mathbb{R}(\!(u)\!)/u\mathbb{R}[\![u]\!],\partial^{S^1}:=\partial+u\delta_1+u^2\delta_2+\cdots\right),
\end{equation}
where $\mathit{SC}^\ast(X)$ is the cochain complex defining the (non-equivariant) symplectic cohomology $\mathit{SH}^\ast(X)$, $\partial$ is the usual Floer differential, $\delta_1$ is the cochain level BV operator, and $u$ is a formal variable of degree 2. The action filtration on $\mathit{SC}^\ast(X)$ induces a filtration $F^\bullet$ on $\mathit{SC}^\ast_{S^1}(X)$, and the \textit{$d$-th Gutt--Hutchings capacity} of $X$, introduced in \cite{ghs}, is defined to be
\begin{equation}
C_d^\mathrm{GH}(X):=\inf\left\{a\left\vert\partial^{S^1}(\tilde{y})=e_X\otimes u^{-d+1}\textrm{ for some }\tilde{y}\in F^{\leq a}\mathit{SC}_{S^1}^{1-2d}(X)\right.\right\}, \nonumber
\end{equation}
where $e_X\in\mathit{SC}^0(X)$ is the cochain level representative of the identity $1\in\mathit{SH}^0(X)$.

\begin{definition}\label{definition:ALC}
For a symplectic manifold $(X,\omega)$, we define 
\[C^{\mathrm{AL}}(X,\omega):=\sup_{\substack{L\subset (X,\omega)\\{\textrm{aspherical Lagrangian}}}}A_{\min}(L),\]
where the supremum on the right-hand side is taken over all closed oriented aspherical Lagrangian submanifolds $L\subset X$ which are $\mathit{Spin}$.
\end{definition}

When the choice of the symplectic form is clear, we shall simply write $C^\mathrm{CM}(X)$ and $C^\mathrm{AL}(X)$ for the Lagrangian capacities.

\begin{definition}\label{definition:Asp-extremal}
Let $L\subset(X,\omega)$ be a closed oriented aspherical Lagrangian submanifold that is $\mathit{Spin}$. It is called extremal if  
\[A_{\min}(L)=C^{\mathrm{AL}}(X,\omega).\]
\end{definition}

We note that $C^{\mathrm{AL}}(X)$ defines a symplectic capacity on star-shaped domains $X\subset\mathbb{C}^n$. Moreover, by definition
\begin{equation}
C^{\mathrm{CM}}(X) \leq C^{\mathrm{AL}}(X).
\end{equation}
Our first result relates the Lagrangian capacity $C^{\mathrm{AL}}(X)$ to Gutt--Hutchings capacities.

\begin{theorem}\label{theorem:GH-ALbound}
Let $(X,\lambda)$ be a  Liouville domain with $c_1(X)=0$. We have 
\begin{equation}\label{eq:ubound}
C^{\mathrm{AL}}(X)\leq\inf_{d\in\mathbb{N}}\frac{C^{\operatorname{GH}}_d(X)}{d}.
\end{equation}
\end{theorem}

\begin{remark}
In \cite{mpo}, a weaker version of the inequality (\ref{eq:ubound}), namely $C^\mathrm{CM}(X)\leq\inf_{d\in\mathbb{N}}\frac{C^{\operatorname{GH}}_d(X)}{d}$ is proved under the additional assumption that $\pi_1(X)=0$ (and the putative perturbation scheme mentioned before if $\dim(X)>4$). This is not needed in our case. In fact, as we will see in Section \ref{section:dilation}, the right-hand side of (\ref{eq:ubound}) exists for many Liouville domains with $\pi_1(X)\neq0$.

According to \cite{ghs}, the limit $\lim_{d\rightarrow\infty}\frac{C_d^\mathrm{GH}(X)}{d}$ exists when $X\subset\mathbb{C}^n$ is a convex or concave toric domain. Irie's work \cite{kia} provides strong evidence for the conjecture that $\lim_{d\rightarrow\infty}\frac{C_d^\mathrm{GH}(X)}{d}$ actually exists for all star-shaped toric domains $X\subset\mathbb{C}^2$.
\end{remark}

\begin{corollary}\label{CM=AL}
For any convex or concave toric domain $X\subset\mathbb{C}^n$, we have
\[C^{\mathrm{CM}}(X)=C^{\mathrm{AL}}(X)=\operatorname{diagonal}(X) .\]
In particular, Conjecture \ref{CMcapaciy-conjecture} holds, i.e., the Cieliebak--Mohnke capacity of an ellipsoid is given by
\[C^{\mathrm{CM}}\left(E^{2n}(a_1,a_2,\dots,a_n)\right)=\left(\frac{1}{a_1}+\frac{1}{a_2}+\cdots+\frac{1}{a_n}\right)^{-1}.\]
\end{corollary}

\begin{remark}\label{CMextremal=Asp-extremal}
By Corollary \ref{CM=AL}, a Lagrangian torus in an ellipsoid is extremal in the sense of Definition \ref{CMextremal} if and only if it is extremal in the sense of Definition \ref{definition:Asp-extremal}. 
\end{remark}

The following result is a partial generalization of \cite{sfe}, Corollary 1.12 in higher dimensions.

\begin{corollary}\label{Asp-extremal-intersection}
For any $0<a_1\leq a_2\leq a_3\leq \dots\leq a_n<\infty$, every closed oriented aspherical Lagrangian submanifold in the ellispoid $E^{2n}(a_1,a_2,\dots,a_n)$ that is extremal in the sense of Definition \ref{definition:Asp-extremal} intersects the extremal product torus
\[
T_{\operatorname{Clif}}:=S^1\left(\delta\right)^{n}
\subset \partial E^{2n}(a_1,a_2,\dots,a_n) \]
non-trivially, where $\delta:= \left(\frac{1}{a_1}+\frac{1}{a_2}+\cdots+\frac{1}{a_n}\right)^{-1}$.
\end{corollary}

Corollary \ref{Asp-extremal-intersection} also implies the following rigidity result for monotone Lagrangian tori. In particular, this proves \cite{kbp}, Conjecture 1.7, which focuses on the case $n=2$.

\begin{corollary}\label{monotone-intersection}
Every monotone Lagrangian torus in $(\mathbb{CP}^n,\omega_{\mathrm{FS}})$ has non-trivial intersections with the Clifford torus.
\end{corollary}

Our method of proving Theorem \ref{theorem:GH-ALbound} also leads to a new upper bound for Lagrangian width for Lagrangian submanifolds in Liouville domains. Let $L$ be a closed Lagrangian submanifold in a symplectic manifold $(X,\omega)$. We say that a symplectic embedding
\[
i:\left(B^{2n}(r),\omega_{\mathrm{std}}\right)\hookrightarrow (X,\omega)
\]
is \textit{relative to} $L$ if
\[
i^{-1}(L)=B^{2n}(r)\cap \mathbb{R}^n.
\]
We define the \textit{Lagrangian width} of $L\subset X$ by
\[
w(L;X):=\sup\left\{r\left\vert\exists\text{ symplectic embedding }B^{2n}(r)\hookrightarrow X \text{ relative to }L\right.\right\}.
\]
Let $(X,\lambda)$ be a Liouville domain. For a Lagrangian submanifold $L\subset\mathrm{int}(X)$ in the interior of $X$, we introduce the following quantities:
\[
C^{\operatorname{GH}}_d(L;X):=\inf_{L\subset W\subset X}C^{\operatorname{GH}}_d(W),\textrm{ } C^{\operatorname{GH}}_{\inf}(L;X):=\inf_{d\in\mathbb N}\frac{C^{\operatorname{GH}}_d(L;X)}{d},
\]
where the infimum is taken over all Liouville subdomains $W\subset X$
containing $L$. Note that the quantities $C^{\operatorname{GH}}_d(L;X)$ and $C^{\operatorname{GH}}_{\inf}(L;X)$ are invariant under Hamiltonian isotopies of $L$. We say that a closed Lagrangian submanifold $L\subset\mathrm{int}(X)$ admits an \textit{exact Lagrangian cap} if there exists a Liouville subdomain $W\subset X$ such that $L\setminus \operatorname{int}(W)$ is a non-empty exact Lagrangian and
$\lambda|_{L\setminus \operatorname{int}(W)}$ admits a primitive which vanishes along $L\cap \partial W$; cf. \cite{r,em}. Our argument in Section \ref{section:argument} gives the following new estimate on the Lagrangian widths of $L$, which has the advantage that in general no assumption on the displaceability of $L$ is needed. When $L\subset\mathrm{int}(X)$ is displaceable, it recovers the upper bound given by Borman--McLean (cf. \cite{bm}, Theorem 1.1).

\begin{corollary}\label{Lagrangian-width}
Let $X$ be a Liouville domain with $c_1(X)=0$. For every closed oriented aspherical Lagrangian submanifold $L\subset\mathrm{int}(X)$ which is Spin, we have
\[
w(L;X)\leq2\, C^{\operatorname{GH}}_{\inf}(L;X).
\]
In particular, $L$ does not admit an exact Lagrangian cap if it lies in a Liouville subdomain $W\subset X$ with $C^{\operatorname{GH}}_d(W)<\infty$ for some $d\in\mathbb N$.

If we further assume that the aspherical Lagrangian submanifold $L\subset\mathrm{int}(X)$ is displaceable, then
\[
w(L;X)\leq4e(L;X),
\]
where $e(L;X)$ is the displacement energy.
\end{corollary}

Using Murphy's $h$-principle for loose Legendrians \cite{m} and its extension to Lagrangian caps \cite{em}, Ekholm-Eliashberg-Murphy-Smith constructed in \cite{eems} Lagrangian embeddings of $S^1\times S^{n-1}$ into $\mathbb{C}^n$ for all $n\geq 3$, which have infinite Lagrangian widths by \cite{r}. Obviously, these Lagrangians are contained in smooth star-shaped domains whose Gutt--Hutchings capacities are finite. This shows that Corollary \ref{Lagrangian-width} fails in general for non-aspherical Lagrangians.

Our results also yield the following quantitative form of the Arnold chord conjecture, see \cite{dct}, Theorem 4, which establishes similar results for the so-called ``E3 Legendrians".

\begin{corollary}\label{chord-estimate}
Let $X\subset \mathbb{C}^n$ be a compact smooth star-shaped domain. Then every closed oriented aspherical Legendrian submanifold $\Lambda\subset \partial X$ which is Spin admits a Reeb chord of length at most
\[
\inf_{d\in\mathbb N}\frac{C^{\operatorname{GH}}_d(X)}{d}.
\]
\end{corollary}

Another application of Theorem \ref{theorem:GH-ALbound} is a proof of Conjecture \ref{realize-CMcapacity}.

\begin{corollary}\label{Embedd=CM}
We have $\mathrm{EC}(a)=2C^\mathrm{CM}\left(E^4(1,a)\right)$.
\end{corollary}

\cite{chls}, Problem 11 asks about the stability of symplectic capacities. For the Cieliebak--Mohnke capacity we can prove the following.

\begin{corollary}\label{stability-CMcapacity}
Let $X \subset \mathbb{C}^n$ be a convex or concave toric domain, we have 
\[C^{\mathrm{CM}}(X)=C^{\mathrm{CM}}(X\times \mathbb{C}).\]
\end{corollary}

Our proof of Theorem \ref{theorem:GH-ALbound} is based on a general result about the non-exact $S^1$-equivariant Viterbo functoriality for closed Lagrangian embedding $L\hookrightarrow(X,\lambda)$ in a Liouville domain $(X,\lambda)$ with finite $d$-th Gutt--Hutchings capacity $C_d^\mathrm{GH}(X)$ for some $d\in\mathbb{N}$. Note that if we assume in addition that $L$ is aspherical, then any such Lagrangian embedding is necessarily non-exact since the disc cotangent bundle $D^\ast L$ has infinite $d$-th Gutt--Hutchings capacity for any $d\in\mathbb{N}$. To state the result, we need to introduce some notations.

For $a\in H_1(L;\mathbb{Z})$, denote by $\mathcal{L}(a)L\subset\mathcal{L}L$ the subspace of the free loop space of $L$ which consists of loops in the class $a$. Consider the $S^1$-equivariant homology group
\begin{equation}\label{eq:StrH}
\mathbb{H}^{S^1}_\ast(a):=H_{\ast+n+\mu(a)-1}^{S^1}(\mathcal{L}(a)L;\mathbb{R}),
\end{equation}
where $\mu:H_1(L;\mathbb{Z})\rightarrow\mathbb{Z}$ is the Maslov index and the $S^1$-action is given by reparametrization of loops. The direct sum
\begin{equation}\label{eq:dirsum}
\mathbb{H}_\ast^{S^1}:=\bigoplus_{a\in H_1(L;\mathbb{Z})}\mathbb{H}_\ast^{S^1}(a)
\end{equation}
carries the action filtration
\begin{equation}
F^\Xi\mathbb{H}_\ast^{S^1}:=\bigoplus_{\lambda(a)>\Xi}\mathbb{H}^{S^1}_\ast(a), \nonumber
\end{equation}
which allows us to define the completion
\begin{equation}\label{eq:H-c}
\widehat{\mathbb{H}}_\ast^{S^1}:=\varprojlim_{\Xi\rightarrow\infty}\mathbb{H}_\ast^{S^1}/F^\Xi\mathbb{H}_\ast^{S^1}.
\end{equation}
Note that $\widehat{\mathbb{H}}_\ast^{S^1}$ is a module over $\mathbb{R}(\!(h)\!)/h\mathbb{R}[\![h]\!]$, where $h$ is a formal variable of degree $-2$ coming from the chain level $S^1$-action induced by reparametrization of loops. 

The following theorem generalizes \cite{yla}, Theorem 6, which deals with the case when $d=1$. Besides, it is also a quantitative version in the sense that the action of $y$ is carefully bounded from below.

\begin{theorem}\label{theorem:deform-Viterbo}
Let $X$ be a $2n$-dimensional Liouville domain with $c_1(X)=0$, and assume that the $d$-th Gutt--Hutchings capacity $C_d^\mathrm{GH}(X)$ is finite for some $d\in\mathbb{N}$. Let $L\subset\mathrm{int}(X)$ be a closed oriented Lagrangian submanifold which is $\mathit{Spin}$ and lying in the interior of $X$. Then there exists an $L_\infty$-structure $(\ell_m)_{m\geq1}$ on $\mathbb{H}_\ast^{S^1}$, together with homology classes $x\in\widehat{\mathbb{H}}_{-2}^{S^1}$, $y\in\widehat{\mathbb{H}}_{2d}^{S^1}$ such that
\begin{itemize}
	\item[(i)] $\ell_1$ vanishes.
	\item[(ii)] The $L_\infty$-structure $(\ell_m)_{m\geq2}$ respects the decomposition of $\mathbb{H}_\ast^{S^1}$ according to classes in $H_1(L;\mathbb{Z})$. In particular, it extends to the completion $\widehat{\mathbb{H}}_\ast^{S^1}$, and we continue to denote its extension by $(\ell_m)_{m\geq2}$.
	\item[(iii)] There is a constant $\eta>0$ such that $x\in F^\eta\widehat{\mathbb{H}}_{-2}^{S^1}$, and $y\in F^{-C_d^\mathrm{GH}(X)-\nu}\widehat{\mathbb{H}}_{2d}^{S^1}$ for any $\nu>0$.
	\item[(iv)] $x$ and $y$ satisfy the equations
	\begin{equation}
	\sum_{m=2}^\infty\frac{1}{m!}\ell_m(x,\cdots,x)=0, \nonumber
	\end{equation}
    \begin{equation}\label{eq:def1}
    \left(\sum_{m=2}^\infty\frac{1}{(m-1)!}\ell_m(y,x,\cdots,x)\right)_{a=0}=(-1)^{n+1}[\![L]\!]\otimes h^{-d+1},
    \end{equation}
    where the infinite sums on the left-hand side of both equations make sense because of (iii), the subscript $a=0$ means restricting to the $a=0\in H_1(L;\mathbb{Z})$ component, and $[\![L]\!]$ denotes the image of the fundamental class of $L$ under the composition
    \begin{equation}
    H_\ast(L;\mathbb{R})\rightarrow H_\ast(\mathcal{L}(0)L;\mathbb{R})\rightarrow H_\ast^{S^1}(\mathcal{L}(0)L;\mathbb{R}), \nonumber
    \end{equation}
    where the first map is induced by the inclusion of constant loops, and the second map is the erasing map in string topology.
\end{itemize}
\end{theorem}

The relation of the above theorem with the $S^1$-equivariant Viterbo functoriality is as follows. Given an exact Lagrangian embedding $L\hookrightarrow\mathrm{int}(X)$, the $S^1$-equivariant Viterbo functoriality gives a map
\[
\mathit{SH}^\ast_{S^1}(X)\rightarrow\mathit{SH}^\ast_{S^1}(D^\ast L)\xrightarrow{\cong}H_{n-\ast}^{S^1}(\mathcal{L}L;\mathbb{R})
\]
relating the $S^1$-equivariant symplectic cohomology of $X$ to the string homology of $L$. When the Lagrangian submanifold $L$ is non-exact, this map needs to be corrected by deforming $H_{n-\ast}^{S^1}(\mathcal{L}L;\mathbb{R})$ using the Maurer-Cartan element obtained by counting pseudoholomorphic discs in $X$ with boundary on $L$. In the above theorem, the class $x\in\widehat{\mathbb{H}}_{-2}^{S^1}$ plays the role of the Maurer-Cartan element. Moreover, the assumption $C_d^\mathrm{GH}(X)<\infty$ implies the existence of an $S^1$-equivariant cochain $\tilde{y}\in\mathit{SC}_{S^1}^{1-2d}(X)$ satisfying the equation $\partial^{S^1}(\tilde{y})=e_X\otimes u^{-d+1}$. Under the $S^1$-equivariant Viterbo functoriality, there should be a corresponding equation in the string homology $H_{n-\ast}^{S^1}(\mathcal{L}L;\mathbb{R})$, where the role of the primitive $\tilde{y}$ is played by $y\in\widehat{\mathbb{H}}_{2d}^{S^1}$ in Theorem \ref{theorem:deform-Viterbo}. However, when $L\subset\mathrm{int}(X)$ is non-exact, such an equation also needs to be deformed by the Maurer-Cartan element $x$, and (\ref{eq:def1}) is the deformed equation, which holds in the completed string homology $\widehat{H}_{n-\ast}^{S^1}(\mathcal{L}L;\mathbb{R})$.

It is clear that the assumptions of Theorem \ref{theorem:deform-Viterbo} are satisfied for star-shaped domains in $\mathbb{C}^n$. There are actually many examples of non-subcritical Weinstein domains where Theorem \ref{theorem:deform-Viterbo} is applicable, see our discussions in Section \ref{section:dilation}.

The proof of Theorem \ref{theorem:deform-Viterbo} relies on a class of moduli spaces of solutions to certain parametrized Floer equations considered by Cohen--Ganatra \cite{cg}. The study of these Floer solutions also enables us to get geometric constraints on the extremal Lagrangian submanifolds in a class of Liouville domains defined below.

\begin{definition}\label{definition:asymp-sta}
A Liouville domain $X$ with $c_1(X)=0$ is called spectrally convex if for a sequence $\{d_i\}_{i\in\mathbb{N}}$ of positive integers, the rate of convergence of
\[\lim_{d_i\to\infty}\frac{C^{\operatorname{GH}}_{d_i}(X)}{d_i}=\inf_{d\in\mathbb{N}}\frac{C^{\operatorname{GH}}_d(X)}{d}\]
is faster than $\frac{1}{d_i}$. Note that the sequence $\{d_i\}_{i\in\mathbb{N}}$ is allowed to be constant. If for some  $d'\in\mathbb{N}$ we have
\[ 
\frac{C^{\operatorname{GH}}_{d'}(X)}{d'}=\inf_{d\in\mathbb{N}}\frac{C^{\operatorname{GH}}_d(X)}{d},
\]
then we can take $d_i=d'$ for all $i\in\mathbb{N}$.
\end{definition}

The term ``spectrally convex" is inspired by the fact that all convex toric domains $X\subset\mathbb{C}^n$ satisfy the required decay condition of Gutt--Hutchings capacities, while there are examples of non-convex star-shaped domains for which the condition fails. With the above definition we have the following.

\begin{theorem}\label{theorem:boundaryrigidity}
Let $X$ be a spectrally convex Liouville domain with $c_1(X)=0$, then any closed oriented aspherical Lagrangian submanifold $L\subset X$ which is $\mathit{Spin}$ and satisfies
\[A_{\min}(L)=\inf_{d\in \mathbb{N}}\frac{C^{\operatorname{GH}}_d(X)}{d}\]
lies entirely in the boundary $\partial X$.
\end{theorem}

The spectral convexity assumption of Definition \ref{definition:asymp-sta} is satisfied for many Liouville domains, there are both subcritical and non-subcritical examples (cf. Section \ref{section:capacity}). In particular, we have the following, which confirms Conjecture \ref{extremal-ellip-conjecture} in view of Remark \ref{CMextremal=Asp-extremal}.

\begin{corollary}\label{extremalonboundary}
For any $0<a_1\leq a_2\leq a_3\leq \dots\leq a_n<\infty$, every closed oriented aspherical Lagrangian submanifold in the ellispoid $E^{2n}(a_1,a_2,\dots,a_n)$ that is extremal in the sense of Definition \ref{definition:Asp-extremal} lies entirely in the boundary $\partial E^{2n}(a_1,a_2,\dots,a_n)$.
\end{corollary}

We remark that there exist extremal aspherical Lagrangians which are not tori. To see this, take a monotone Lagrangian $K\subset\mathbb{C}^3$ diffeomorphic to $S^1\times\Sigma_g$ for some $g\geq2$. The abundance of such Lagrangians follows from \cite{akf}, Theorem 4.2. Set $r_1:=A_{\min}(K)$. Let $r_2>r_1$ be sufficiently large so that $K\subset\textrm{int}\left(B^6(r_2)\right)$. Now
\[
L:=\partial B^2(r_1)\times K\subset B^2(r_1)\times B^6(r_2)
\]
is an extremal Lagrangian in the convex toric domain that is not diffeomorphic to a torus.

As a byproduct of the results obtained in this paper, especially Theorem \ref{theorem:GH-ALbound} and Corollary \ref{CM=AL}, we give the first computations of the Lagrangian capacities $C^\mathrm{CM}(X)$ and $C^\mathrm{AL}(X)$ for many non-subcritical Weinstein domains $X$ in Section \ref{section:capacity}.

\begin{proposition}
Let $X$ be the unit disc cotangent bundle of $S^2$ or $\mathbb{RP}^2$, then
\[
C^\mathrm{CM}(X)=C^\mathrm{AL}(X)=2\pi.
\]
Let $S\subset\mathbb{R}^3$ be a Zoll sphere of revolution, such that the length of any simple closed geodesic is $l$, then
\[
C^\mathrm{CM}(D^\ast S)=C^\mathrm{AL}(D^\ast S)=l
\]
for the unit disc cotangent bundle of $S$. Let $X$ be the unit disc cotangent bundle of a $3$-dimensional lens space $L(p,q)$, then
\[
C^\mathrm{CM}(X)=C^\mathrm{AL}(X)=\pi.
\]
\end{proposition}

This paper studies a version of Lagrangian capacity which is more general than the original definition of Cieliebak--Mohnke \cite{cmp} in the sense that we use closed aspherical Lagrangians in the symplectic manifold $X$ instead of just the Lagrangian tori. One can include more general Lagrangians in $X$ and consider the corresponding capacity. For example, the recent work of S. Li \cite{slt} suggests that similar techniques as explored here can be used to study capacities defined using all $L\subset X$ with $\pi_2(L)=0$, which, however, would differ drastically from $C^\mathrm{AL}(X)$ since there exist closed Lagrangians in $\mathbb{C}^n$ with $\pi_2(L)=0$ and Maslov number $>2$. See also \cite{cmp}, Remark 1.6, where the authors suggest defining the Lagrangian capacity using all closed Lagrangians. In the case of a unit ball $B^{2n}(1)$ with $n\geq2$, the capacity defined using all closed Lagrangians is conjectured to be $\frac{1}{2}$. On the other hand, $C^\mathrm{CM}\left(B^{2n}(1)\right)=C^\mathrm{AL}\left(B^{2n}(1)\right)=\frac{1}{n}$.

The paper is organized as follows. In Section \ref{section:argument}, we prove Theorems \ref{theorem:GH-ALbound} by assuming the validity of Theorem \ref{theorem:deform-Viterbo}. In Section \ref{section:corollary}, after further assuming the correctness of Theorem \ref{theorem:boundaryrigidity}, we prove Corollaries \ref{CM=AL}, \ref{Lagrangian-width}, \ref{Embedd=CM}, \ref{stability-CMcapacity} and \ref{extremalonboundary}. The proof of Theorem \ref{theorem:deform-Viterbo} is given in Section \ref{section:proof}, after introducing the relevant moduli spaces in Section \ref{section:CG} and the required $S^1$-equivariant de Rham chain model of the free loop space in Section \ref{section:de Rham}. The proof of Theorem \ref{theorem:boundaryrigidity} is postponed to Section \ref{section:extremal}, since some familiarity with the moduli spaces considered in Section \ref{section:CG} would be helpful for understanding the argument. Finally, we discuss the examples of (non-subcritical) Liouville domains where Corollary \ref{Lagrangian-width} and Theorem \ref{theorem:deform-Viterbo} are applicable and explicitly compute the Lagrangian capacites of some of these Liouville domains in Section \ref{section:example}.

\section*{Acknowledgements}
The possibility of applying the techniques of \cite{yla} to Liouville domains with finite higher Gutt--Hutchings capacities was first asked by Mark McLean during a talk given by the second author at Stony Brook University in 2023. The first author is grateful to Georgios Dimitroglou-Rizell and Luis Diogo for many valuable conversations, and to Kei Irie for helpful discussions during the 2024 Paris conference \textit{From Hamiltonian Dynamics to Symplectic Topology and Beyond}, held in honor of Claude Viterbo.

\section{Bounding the Lagrangian capacity}\label{section:argument}

Assuming Theorem \ref{theorem:deform-Viterbo}, we prove Theorem \ref{theorem:GH-ALbound} in this section. We remark that due to the complexities in the constructions of the string homology classes $x,y\in\widehat{\mathbb{H}}_\ast^{S^1}$, there are some unavoidable technical ambiguities in the proof, which will become clear once Theorem \ref{theorem:deform-Viterbo} is proved in Section \ref{section:proof}. To prove Theorem \ref{theorem:GH-ALbound}, we will use the following fact about the free loop spaces of aspherical manifolds.

\begin{lemma}[\cite{jlf}, Corollary 5.3]\label{lemma:CW}
If $L$ is an aspherical manifold, then every connected component of $\mathcal{L}L$ has the homotopy type of a CW complex of dimension at most $n$.
\end{lemma}

\begin{proof}[Proof of Theorem \ref{theorem:GH-ALbound}]
Let $(X,\lambda)$ be a Liouville domain with $c_1(X)=0$. Suppose that for some $d\in\mathbb{N}$, the $d$-th Gutt--Hutchings capacity $C^{\operatorname{GH}}_d(X)$ of $X$ is finite. Let $J$ be any contact type almost complex structure on $X$ (cf. Section \ref{section:CG}), we aim to prove that any closed oriented aspherical Lagrangian submanifold $L\subset X$ which is $\mathit{Spin}$ bounds a $J$-holomorphic disk $u:(D,\partial D)\rightarrow(X,L)$ such that
\[
0<\int_{\partial D}u^\ast\lambda\leq\frac{C^{\operatorname{GH}}_d(X)}{d}.
\]
	
We first assume that $L\subset\mathrm{int}(X)$. The equation (\ref{eq:def1}) in Theorem \ref{theorem:deform-Viterbo} can be written as
\[
\sum_{m=1}^{\infty}\frac{1}{m!}\sum_{a=a_1+\cdots+a_m}\ell_{m+1}(y(-a),x(a_1),\cdots,x(a_m))=(-1)^{n+1}[\![L]\!]\otimes h^{-d+1},
\]
where $a,a_1,\dots,a_m\in H_1(L;\mathbb{Z})$, and $x(a_i)$ denotes the projection of $x$ onto
\[
H^{S^1}_{\ast+n+\mu(a_i)-1}(\mathcal{L}(a_i)L;\mathbb{R})
\]
under the decomposition of the string homology defined in (\ref{eq:dirsum}).
	
Under the assumption that $L$ is aspherical, there is a topological splitting on the free loop space $\mathcal{L}L$, form which it follows that the class $[\![L]\!]\otimes h^{-d+1}$ is nontrivial in the string homology $H_\ast^{S^1}(\mathcal{L}L;\mathbb{R})$. Therefore, for some $m\geq 1$, we must have
\begin{equation}\label{eq:lm}
\ell_{m+1}(y(-a),x(a_1), \dots,x(a_m))\neq 0.
\end{equation}
The gradings of the inputs in the above equation are given by
\[
|y(-a)|=2d+n-1-\mu(a)\text{ and }|x(a_i)|=n-3+\mu(a_i).
\]
By Lemma \ref{lemma:CW}, the vector space $H_\ast^{S^1}(\mathcal{L}(b)L;\mathbb{R})$ is supported in degrees $0\leq *\leq n-1$ for $b\neq0\in H_1(L;\mathbb{Z})$, with the only possible non-trivial class in $H_n^{S^1}(\mathcal{L}L;\mathbb{R})$ being a multiple of $[\![L]\!]$. But for (\ref{eq:lm}) to be non-zero, none of the entries of $\ell_{m+1}$ can be a multiple of $[\![L]\!]$. It follows that
\[
2d\leq \mu(a)\leq 2d+n-1,
\text{ }
3-n\leq \mu(a_i)\leq 2
\]
for all $i=1,\dots,m$. Consequently,
\[
2m\geq \sum_{i=1}^m \mu(a_i)=\mu(a)\geq 2d,
\]
which yields $m\geq d$.
	
By the construction of $x$ in the proof of Theorem \ref{theorem:deform-Viterbo} (cf. Section \ref{section:proof}) and that $x(a_i)\neq 0$ for each $i=1,\dots,m$, there exist non-constant $J$-holomorphic discs $u_1, \dots, u_m:(D,\partial D)\rightarrow(X,L)$\footnote{These discs exist for any convex $d\lambda$-compatible almost complex structure. However, their boundaries on $L$ do not necessarily represent distinct classes in $H_1(L;\mathbb{Z})$. For example, when $L$ is the Chekanov--Schlenk torus \cite{csn} in $\mathbb{C}^2$, all of these disks represent the same class in $H_1(L;\mathbb{Z})$.} such that
\[
\left[u_i(\partial D)\right]=a_i\in H_1(L;\mathbb{Z}),
\]
for all $i=1,\dots,m$. Moreover, it follows from (\ref{eq:lm}) and Theorem \ref{theorem:deform-Viterbo}, (ii) and (iii) that
\[
\sum_{i=1}^d \int_{\partial D}u_i^\ast\lambda\leq\sum_{i=1}^m \int_{\partial D}u_i^\ast\lambda\leq\lambda(a)\leq C^{\operatorname{GH}}_d(X).
\]
In particular, there must be some disc $u_j:(D,\partial D)\rightarrow(X,L)$ with $1\leq j\leq d$, such that
\begin{equation}\label{eq:disk-area}
0<\int_{\partial D}u_j^\ast\lambda\leq \frac{1}{d}\sum_{i=1}^d \int_{\partial D}u_i^\ast\lambda\leq \frac{C^{\operatorname{GH}}_d(X)}{d}.
\end{equation}

To deal with the case when $L\cap\partial X\neq\emptyset$, we extend $X$ slightly by attaching the collar
\[
\left([0,\eta]\times \partial X,d(e^r\lambda)\right)
\]
along $\partial X$ for some small $\eta>0$, where $r\in[0,\eta]$ is the radial coordinate, and define
\[
X_{+\eta}:=X\cup_{\partial X}\left([0,\eta]\times \partial X\right).
\]
Note that there is an obvious inclusion $X\subset X_{+\eta}$ as a Liouville subdomain. By \cite{grc}, Lemma 4.5 we have
\begin{equation}\label{eq:collar}
C^{\operatorname{GH}}_d(X_{+\eta})=e^\eta C^{\operatorname{GH}}_d(X)=(1+o(\eta))\,C^{\operatorname{GH}}_d(X).
\end{equation}
Now, we can regard $L$ as a Lagrangian submanifold in $\mathrm{int}(X_{+\eta})$. Running the same argument as above implies that there is a $J$-holomorphic disc $u:(D,\partial D)\rightarrow(X_{+\eta},L)$ satisfying
\[
0<\int_{\partial D}u^\ast\lambda\leq\frac{C^{\operatorname{GH}}_d(X_{+\eta})}{d}=\frac{\left(1+o(\eta)\right)C_d^\mathrm{GH}(X)}{d}.
\]
Letting $\eta\rightarrow0$ we get the same inequality as (\ref{eq:disk-area}).

Thus we have proved that
\[
C^{\mathrm{AL}}(X)
\leq
\inf_{l\in\mathbb{N}}\frac{C^{\operatorname{GH}}_d(X)}{d}. \qedhere
\]
\end{proof}

\section{Proofs of the Corollaries}\label{section:corollary}

In this section, we prove the corollaries of Theorems \ref{theorem:GH-ALbound} and \ref{theorem:boundaryrigidity}. The proof of Theorem \ref{theorem:boundaryrigidity} will be given in Section \ref{section:extremal}, here we assume its correctness. We start by recalling some basic notions.

Given a domain $\Omega \subset \mathbb{R}^n_{\geq 0}$, the associated toric domain
$X_{\Omega} \subset \mathbb{C}^n$ is defined by
\[
X_{\Omega} := \mu^{-1}(\Omega),
\]
where $\mu: \mathbb{C}^n \to\mathbb{R}^n_{\geq 0}$ is the moment map $\mu(z_1,\dots,z_n)=\pi\big(|z_1|^2,\dots,|z_n|^2\big)$ associated to the standard Hamiltonian torus action on $\mathbb{C}^n$. $X_\Omega$ is equipped with the restriction of the standard symplectic form $\omega_\mathrm{std}$ on $\mathbb{C}^n$. Define the diagonal of $X_{\Omega}$ by
\[
\delta_{\Omega}:=\sup\{a: (a,\dots,a) \in \Omega\}.
\]
We say that $X_\Omega$ is a \textit{convex} toric domain if the set $\Omega\subset\mathbb{R}^n_{\geq 0}$ is convex, and it is a \textit{concave} toric domain if $\mathbb{R}^n_{\geq 0}\setminus\Omega$ is convex.

For $\delta>0$, we denote by $N^{2n}(\delta) \subset \mathbb{C}^n$ the non-disjoint union of cylinders of diagonal $\delta$, which is given by
\[
N^{2n}(\delta) := \mu^{-1}(\Omega'),
\]
where
\[
\Omega':=\left.\left\{(x_1,\dots,x_n)\in \mathbb{R}^n_{\geq 0}\right\vert x_i \leq \delta \text{ for some } i \in \{1,\dots,n\}\right\}.
\]

We start with the proof of Corollary \ref{CM=AL}, which computes the Lagrangian capacities of convex and concave toric domains and in particular confirms \cite{cmp}, Conjecture 1.5.

\begin{proof}[Proof of Corollary \ref{CM=AL}]
For $X_{\Omega}$ convex or concave, we have
\[
X_{\Omega} \subset N^{2n}(\delta_{\Omega}).
\]
By the monotonicity of Gutt--Hutchings capacities and \cite{ghs}, Lemma 1.19, we have
\[
C^{\operatorname{GH}}_d(X_{\Omega})\leq C^{\operatorname{GH}}_d\big(N^{2n}(\delta_{\Omega})\big)=\delta_{\Omega}(d+n-1)	\]
for all $d\in\mathbb{N}$. Hence, by Theorem \ref{theorem:GH-ALbound}, we have
\[
C^{\mathrm{AL}}(X_{\Omega})\leq\inf_{d\in\mathbb{N}}\frac{C^{\operatorname{GH}}_d(X_{\Omega})}{d}\leq\delta_{\Omega}.
\]
	
On the other hand, the Clifford torus
\[
T_\mathrm{Cl}(\delta_\Omega):=S^1(\delta_{\Omega}) \times\cdots\times S^1(\delta_{\Omega})\subset (\mathbb{C}^n, \omega_{\mathrm{std}})
\]
is a Lagrangian torus with $A_{\min}\left(T_\mathrm{Cl}(\delta_\Omega)\right)=\delta_{\Omega}$ contained in $\partial X_{\Omega}$. Therefore,
\[
C^{\mathrm{AL}}(X_{\Omega})\geq C^{\mathrm{CM}}(X_{\Omega})\geq\delta_{\Omega}.
\]
Combining the above inequalities, we conclude that
\[
C^{\mathrm{CM}}(X_{\Omega})=C^{\mathrm{AL}}(X_{\Omega})=\delta_{\Omega}.\qedhere
\]
\end{proof}

\begin{remark}
The same argument can be used to prove the that $C^{\mathrm{CM}}(X_{\Omega})=C^{\mathrm{AL}}(X_{\Omega})=\eta_{\Omega}$ for toric domains $X\subset\mathbb{C}^n$ satisfying $(\eta_\Omega,\cdots,\eta_\Omega)\in\partial\Omega$, where $\eta_\Omega:=\inf\left\{a\left\vert X_\Omega\subset N^{2n}(a)\right.\right\}$. See \cite{gprc}, Theorem 18.
\end{remark}

\begin{proof}[Proof of Corollary \ref{Asp-extremal-intersection}]
Let $L$ be an extremal closed oriented aspherical Lagrangian submanifold in the ellipsoid $E^{2n}(a_1,a_2,\dots,a_n)\subset\mathbb{C}^n$. Suppose, for contradiction, that $L$ does not intersect $T_{\operatorname{Clif}}$. Let $\mu:\mathbb{C}^n\to(\mathbb{R}_{\geq 0})^n$ be the standard moment map.
For every point $
(x_1,\dots,x_n)\in \mu(E(a_1,\dots,a_n))$,
we have
\[
1\ge\frac{x_1}{a_1}+\cdots+\frac{x_n}{a_n}\ge\left(\frac{1}{a_1}+\cdots+\frac{1}{a_n}\right)\min_i x_i=\frac{\min_i x_i}{\delta}.
\]
Hence
\[
\min_i x_i\le \delta.
\]
Moreover, equality holds if and only if
\[
x_1=\cdots=x_n=\delta.
\]
Since $L$ does not intersect $T_{\operatorname{Clif}}$, the set $\mu(L)$ does not contain the point $(\delta,\dots,\delta)$. Therefore,
\[
\min_i x_i<\delta
\]
for every $(x_1,\dots,x_n)\in \mu(L)$. Since $L$ is compact, there exists $\epsilon>0$ such that 
\[
\min_i x_i\le \delta-\epsilon
\]
for every $(x_1,\dots,x_n)\in\mu(L)$. It follows that
\[
L\subset \operatorname{int}N^{2n}(\delta-\epsilon),
\]
where $N^{2n}(\delta-\epsilon)$ is the non-disjoint union of cylinders of diagonal $\delta-\epsilon$. Next, choose a smooth star-shaped domain $X$ such that
\[
L\subset X\subset N^{2n}(\delta-\epsilon).
\]
If $L$ bounds a disc of area smaller than $\delta$ in $X$, then the same disc has to appear in $E^{2n}(a_1,\cdots,a_n)$, therefore $A_{\min}(L)\geq\delta$ for $L\subset X$. By Theorem \ref{theorem:GH-ALbound},
\[
\delta\leq A_{\min}(L)\le\inf_{d\in\mathbb{N}}\frac{C_d^{\operatorname{GH}}(X)}{d}.
\]
By monotonicity of the Gutt--Hutchings capacities,
\[
\inf_{d\in\mathbb{N}}\frac{C_d^{\operatorname{GH}}(X)}{d}\le\inf_{d\in\mathbb{N}}\frac{C_d^{\operatorname{GH}}\bigl(N^{2n}(\delta-\epsilon)\bigr)}{d}.
\]
By \cite{ghs}, Lemma 1.19 we have
\[
\inf_{d\in\mathbb{N}}\frac{C_d^{\operatorname{GH}}\bigl(N^{2n}(\delta-\epsilon)\bigr)}{d}=\delta-\epsilon.
\]
Consequently,
\[
\delta\le\delta-\epsilon,
\]
which is a contradiction. Thus $L\cap T_{\operatorname{Clif}}\neq\emptyset$.
\end{proof}

\begin{proof}[Proof of Corollary \ref{monotone-intersection}]
We follow the strategy outlined in \cite{sfe}, Remark 1.13. Let $L\subset\mathbb{CP}^n$ be a monotone Lagrangian torus. By \cite{sfe}, Corollary $1.10$, its lift $\widetilde{L}$ to the boundary of the closed ball $B^{2n+2}(1)$ is an extremal Lagrangian torus in the ball. Applying Corollary \ref{Asp-extremal-intersection} to $\widetilde{L}$ in $B^{2n+2}(1):=E^{2(n+1)}(1,\dots, 1)$, we conclude that $\widetilde{L}$ intersects the diagonal Clifford torus
\[
S^1\left(\frac{1}{n+1}\right)^{n+1}.
\]
Under the Hopf projection
\[
S^{2n+1}\longrightarrow \mathbb{CP}^n,
\]
this diagonal torus projects to the monotone Clifford torus in $\mathbb{CP}^n$. Since $\widetilde{L}$ projects to $L$, it follows that $L$ intersects the monotone Clifford torus in $\mathbb{CP}^n$.
\end{proof}

Next, we prove Corollary \ref{extremalonboundary}, which confirms \cite{sfe}, Conjecture 9.11. According to Theorem \ref{theorem:boundaryrigidity}, it is enough to verify that ellipsoids are spectrally convex as Liouville domains in the sense of Definition \ref{definition:asymp-sta}.

\begin{proof}[Proof of Corollary \ref{extremalonboundary}]
We note that	\[\inf_{d\in\mathbb{N}}\frac{C^{\operatorname{GH}}_d\left(E^{2n}(a_1,a_2,\dots,a_n)\right)}{d}=\left(\frac{1}{a_1}+\dots+\frac{1}{a_n}\right)^{-1}.\]
By Theorem \ref{theorem:boundaryrigidity}, it suffices to prove that for some $l\in \mathbb{N}$ we have
\[C^{\operatorname{GH}}_{d}\left(E^{2n}(a_1,a_2,\dots,a_n)\right)=d\bigg(\frac{1}{a_1}+\dots+\frac{1}{a_n}\bigg)^{-1}.\]
By density, we can assume that $a_1, a_2,\dots,a_n$ are all rationals. By \cite{ghs}, Example 1.8, the $d$-th Gutt--Hutchings capacity of the ellipsoid $E^{2n}(a_1,a_2,\dots,a_n)$ is the $d$-th term in the sequence of positive integer multiples of $a_i$ arranged in non-decreasing order with repetitions. For $x> 0$, the number of elements of the multiset $\{m a_i|m\in \mathbb{N}\}$ that are less than or equal to $x$ is given by
\[
n(x):=\sum_{i=1}^n \left\lfloor \frac{x}{a_i}\right\rfloor\in\mathbb{N},
\]
so
\[
C^{\operatorname{GH}}_d\left(E^{2n}(a_1,a_2,\dots,a_n)\right)=\min\left\{x| n(x)\ge d\right\}.
\]
Since all $a_i$'s are rational, we can choose $q\in\mathbb{N}$ and $p_i\in\mathbb{N}$ such that	
\[
a_i=\frac{p_i}{q}\text{ for } i=1,\dots,n.
\]	
Let $L=\mathrm{lcm}(p_1,\dots,p_n)$ and note that
\[
x:=\frac{L}{q}
\]
is an element of the multiset $\{m a_i|m\in \mathbb{N}\}$. We claim that for $n(x)\in \mathbb{N}$, we have
\[
C^{\operatorname{GH}}_{n(x)}\left(E^{2n}(a_1,a_2,\dots,a_n)\right)=n(x)\left(\frac{1}{a_1}+\dots+\frac{1}{a_n}\right)^{-1}.
\]
To see this, note that for each $i$ we have
\[
\frac{x}{a_i}=\frac{L/q}{n_i/q}=\frac{L}{n_i}\in\mathbb{Z}.\]
Hence
\[
n(x)=\sum_{i=1}^n \left\lfloor \frac{x}{a_i}\right\rfloor=\sum_{i=1}^n \frac{x}{a_i}=x\sum_{i=1}^n \frac1{a_i},
\]
or equivalently
\begin{equation}\label{GHin}
x=n(x)\left(\frac{1}{a_1}+\dots+\frac{1}{a_n}\right)^{-1}.
\end{equation}
Next, we prove that
\[C^{\operatorname{GH}}_{n(x)}\left(E^{2n}(a_1,a_2,\dots,a_n)\right)=x.\]
Note that
\[
x=\left(\frac{x}{a_i}\right)a_i,
\]
so $x$ is itself a multiple of each $a_i$ and hence $x$ appears in the multiset $\{m a_i|m\in \mathbb{N}\}$. If $y<x$, then for each $i$ we have
\[
\frac{y}{a_i}<\frac{x}{a_i}\quad\Rightarrow\quad\left\lfloor\frac{y}{a_i}\right\rfloor\le\frac{x}{a_i}-1.
\]
Taking the sum gives
\[
n(y)\le\sum_{i=1}^n\left(\frac{x}{a_i}-1\right)=\sum_{i=1}^n\frac{x}{a_i}-n=n(x)-n<n(x).
\]
Hence, no $y<x$ satisfies $n(y)\ge k$. Therefore
\[
C^{\operatorname{GH}}_{n(x)}\left(E^{2n}(a_1,a_2,\dots,a_n)\right)=\min\{y|\ n(y)\ge n(x)\}=x.
\]
Combining this with (\ref{GHin}), we obtain
\[
C^{\operatorname{GH}}_{n(x)}(E^{2n}(a_1,a_2,\dots,a_n))=n(x)\left(\frac{1}{a_1}+\dots+\frac{1}{a_n}\right)^{-1}.
\]
This completes our proof.
\end{proof}

We then prove Corollary \ref{Lagrangian-width}, which gives an upper bound of Lagrangian width for aspherical Lagrangian submanifolds in terms of Gutt--Hutchings capacities.

\begin{proof}[Proof of Corollary \ref{Lagrangian-width}]
Let $L\subset\mathrm{int}(X)$ be a closed oriented aspherical Lagrangian submanifold which is $\textit{Spin}$. Assume that there exists a Liouville subdomain
\[
(W,\lambda|_W)\subset (X,\lambda)
\]
such that $L\subset W$ and $C^{\operatorname{GH}}_d(W)<\infty$ for some $d\in \mathbb{N}$. It follows from the proof of Theorem \ref{theorem:GH-ALbound} in Section \ref{section:argument} that, for every convex $d\lambda$-compatible almost complex structure $J$ on $W$ and every point $p\in L$, there exists a non-constant $J$-holomorphic disc
\[
u:(D,\partial D)\to (W,L)
\]
of Maslov index $2$, with $p\in u(\partial D)$, satisfying
\[
\int_{\partial D}u^*\lambda\leq\frac{C^{\operatorname{GH}}_d(W)}{d}.
\]
The upper bound $w(L;X)\leq2C_\mathrm{inf}^\mathrm{GH}(L;X)$ then follows from \cite{bm}, Lemma 1.4 and Theorem 1.2.

Now assume that the aspherical Lagrangian submanifold $L\subset\mathrm{int}(X)$ is displaceable. In this case, we can run the same argument as in Section \ref{section:argument} with a moduli space used by Irie \cite{kic} to realize Fukaya's original proposal \cite{kfa} in place of the Cohen--Ganatra moduli space used to define the class $y\in\widehat{\mathbb{H}}_{2d}^{S^1}$. For the reader's convenience, we recall its definition. Take a compactly supported Hamiltonian function $H_t:[0,1]\times X\rightarrow\mathbb{R}$ that displaces $L$ from itself, i.e. $L\cap\phi_H^1(L)=\emptyset$, where $\phi_H^1$ is the time-$1$ flow of the Hamiltonian vector field $X_{H_t}$. Let $\chi:\mathbb{R}\rightarrow[0,1]$ be a function such that $\chi=0$ on $(-\infty,0]$ and $\chi=1$ on $[1,\infty)$. For each $r\geq0$, define $\chi_r(s)=\chi(r+s)\chi(r-s)$. Choose an identification $D\setminus\{\pm1\}\cong\mathbb{R}\times[0,1]$, and denote by $s$ and $t$ the coordinates on $\mathbb{R}$ and $[0,1]$, respectively. Define
\begin{equation}
\mathcal{N}^r_{k+1}(L,\bar{a}):=\left\{(u,z_0=1,\cdots,z_1,\cdots,z_k)\right\} \nonumber
\end{equation}
to be the space of smooth maps $u:(D,\partial D)\rightarrow(X,L)$ in the relative homotopy class $\bar{a}\in\pi_2(X,L)$ satisfying the perturbed Cauchy-Riemann equation
\begin{equation}
\left(du-X_{\chi_r(s)H_t}(u)\otimes dt\right)^{0,1}=0, \nonumber
\end{equation}
where the $(0,1)$-part is taken with respect to some convex almost complex structure on $X$, and $z_0=1,\cdots,z_k\in\partial D$ are distinct marked points on the boundary aligned counterclockwisely. Define
\begin{equation}
\mathcal{N}_{k+1}^{\geq0}(L,\bar{a}):=\bigcup_{r\geq0}\mathcal{N}^r_{k+1}(L,\bar{a}). \nonumber
\end{equation}
The class $y\in\widehat{\mathbb{H}}_{2d}^{S^1}$ is then replaced with a class $y'\in\widehat{\mathbb{H}}_2^{S^1}$ (or its non-equivariant version in $\widehat{\mathbb{H}}_2$) defined by pushing forward the virtual fundamental chains of the admissible K-spaces $\overline{\mathcal{N}}_{k+1}^{\geq0}(L,\bar{a})$ to the chain models $\mathcal{L}_{k+1}$ of the free loop space defined in Section \ref{section:de Rham} below. See also \cite{yla}, Remark 52, where the relations between the Cohen-Ganatra moduli spaces and $\overline{\mathcal{N}}_{k+1}^{\geq0}(L,\bar{a})$ is discussed. Now $y'(-a)\neq0$ for some $a\in H_1(L;\mathbb{Z})$ implies that its action is bounded below by $-2\vert\!\vert H\vert\!\vert$, where
\begin{equation}
\vert\!\vert H\vert\!\vert=\int_0^1(\max H_t-\min H_t)dt \nonumber
\end{equation}
is the Hofer norm, see \cite{kic}, Lemma 7.17, (iii). Thus the same argument as in Section \ref{section:argument} implies the existence of a Maslov $2$ holomorphic disc $u:(D,\partial D)\rightarrow (X,L)$ passing through $p\in L$ whose area is bounded above by two times the displacement energy
\[
e(L;X):=\inf\left\{\vert\!\vert H\vert\!\vert:L\cap\phi_H^1(L)=\emptyset\right\}. \qedhere
\]
Applying  \cite{bm}, Lemma 1.4 finishes the proof.
\end{proof}

We now prove Corollary \ref{chord-estimate}, which establishes a quantitative version of the Arnold chord conjecture for aspherical Legendrian submanifolds in the boundary of star-shaped domains $X\subset\mathbb{C}^n$.

\begin{proof}[Proof of Corollary \ref{chord-estimate}]
The proof is based on the well-known argument of Mohnke from \cite{mk}, so we only give the essential details. Assume, towards a contradiction, that there exists a Legendrian submanifold
$\Lambda\subset \partial X$ with no Reeb chord of length at most $T$, where
\begin{equation}\label{Reeb-chord-est}
T>\inf_{d\in\mathbb N}\frac{C^{\operatorname{GH}}_d(X)}{d}.
\end{equation}
Let $\varepsilon>0$. Starting from $\Lambda$, we flow along the Reeb vector field on $\partial X$ for time $T$. We then move the resulting Legendrian radially inward to the hypersurface $\varepsilon\partial X$. On $\varepsilon \partial X$, we apply the negative Reeb flow for time $T$, and finally return radially to $\partial X$. Since $\Lambda$ has no Reeb chord of length at most $T$, this procedure gives, after smoothing the corners, a Lagrangian embedding of $S^1\times \Lambda$ into $X$. Its minimal symplectic area is
\[
A_{\min}(S^1\times \Lambda)=(1-\varepsilon)T.
\]
If $\Lambda$ is closed, oriented, aspherical, and Spin, then Theorem \ref{theorem:GH-ALbound} applies to this Lagrangian and gives
\[
(1-\varepsilon)T\leq \inf_{d\in\mathbb N}\frac{C^{\operatorname{GH}}_d(X)}{d}
\]
for all $\varepsilon>0$. However, by \eqref{Reeb-chord-est}, this inequality is violated for sufficiently small $\varepsilon>0$. This is the desired contradiction.
\end{proof}

Next, we prove Corollary \ref{Embedd=CM}, which relates the Cieliebak--Mohnke capacity to the embedding capacity of a $4$-dimensional ellipsoid, which confirms \cite{chls}, Conjecture 2.

\begin{proof}[Proof of Corollary \ref{Embedd=CM}]

We want to prove that for every \(a>0\), we have
\[\operatorname{EC}(a)=2\, C^{\mathrm{CM}}\big(E^{4}(1,a),\omega_{\mathrm{std}}\big).
\]
To do so, it suffices to show that for every $r>0$,
\begin{equation}\label{CM-ballcylinder}
C^{\mathrm{CM}}\!\left(B^4(r)\cup B^2(r/2)\times\mathbb{C}\right)=\frac{r}{2}.
\end{equation}
Indeed, suppose there exists a symplectic embedding
\[
E^4(1,a)\hookrightarrow B^4(r)\cup B^2(r/2)\times\mathbb{C}
\]
for some $a,r>0$. By monotonicity of the Cieliebak--Mohnke capacity, we obtain
\[
C^{\mathrm{CM}}\!\left(B^4(r)\cup B^2(r/2)\times\mathbb{C},\omega_{\mathrm{std}}\right)
\ge
C^{\mathrm{CM}}\big(E^{4}(1,a),\omega_{\mathrm{std}}\big).
\]
Assuming (\ref{CM-ballcylinder}), it follows that
\[
\frac{r}{2}
\ge
C^{\mathrm{CM}}\big(E^{4}(1,a),\omega_{\mathrm{std}}\big),
\]
and hence
\[
r\ge
2\,C^{\mathrm{CM}}\big(E^{4}(1,a),\omega_{\mathrm{std}}\big).
\]
Therefore,
\begin{equation}\label{eq:ineq1}
\operatorname{EC}(a)\ge
2\,C^{\mathrm{CM}}\big(E^{4}(1,a),\omega_{\mathrm{std}}\big).
\end{equation}
On the other hand, for
\[
r=2\,C^{\mathrm{CM}}\big(E^{4}(1,a),\omega_{\mathrm{std}}\big)
=\frac{2a}{a+1},
\]
there exists a symplectic embedding given by the inclusion
\[
E^4(1,a)\hookrightarrow B^4(r)\cup B^2(r/2)\times\mathbb{C}.
\]
Consequently,
\begin{equation}\label{eq:ineq2}
\operatorname{EC}(a)\le
2\,C^{\mathrm{CM}}\big(E^{4}(1,a),\omega_{\mathrm{std}}\big).
\end{equation}
Combining (\ref{eq:ineq1}) and (\ref{eq:ineq2}) yields
\[
\operatorname{EC}(a)
=
2\,C^{\mathrm{CM}}\big(E^{4}(1,a),\omega_{\mathrm{std}}\big).
\]
It remains to verify \eqref{CM-ballcylinder}. To this end, observe that
\[
B^4(r)\cup B^2(r/2)\times\mathbb{C} \subset \mathbb{C}^2
\]
is a toric domain whose diagonal equals $r/2$. Thus (\ref{CM-ballcylinder}) follows from Corollary \ref{CM=AL}. This completes the proof.
\end{proof}

Finally, we prove Corollary \ref{stability-CMcapacity}, which establishes the stability of the Cieliebak--Mohnke capacity.

\begin{proof}[Proof of Corollary \ref{stability-CMcapacity}]
We aim to prove that for any compact star-shaped domain $X \subset\mathbb{C}^n$ and any $m\in \mathbb{N}$,
\[
C^{\mathrm{CM}}\left(X\right)=C^{\mathrm{CM}}\left(X\times \mathbb{C}^m\right)\leq\inf_{d\in\mathbb{N}}\frac{C^{\operatorname{GH}}_d(X)}{d}.
\]
Let \( L \subset X \) be a Lagrangian torus with minimal symplectic area
\(\delta = A_{\min}(L) > 0\). Consider its stabilization
\[
L_{\mathrm{stab}}:=L \times \overset{\text{$m$-times}}{S^1(\delta) \times \cdots \times S^1(\delta)}
\subset X \times \mathbb{C}^m.
\]
By construction, we have
\[
A_{\min}(L_{\mathrm{stab}}) = A_{\min}(L).
\]
It follows that
\[
C^{\mathrm{CM}}(X)\leq C^{\mathrm{CM}}(X\times\mathbb{C}^m).
\]
\cite{grc}, Theorem 1.2 and  \cite{ghs}, Equation (1-5) imply that for any $d \in \mathbb{N}$ and sufficiently large $r > 0 $, we have
\[
C^{\operatorname{GH}}_d\left(X \times B^{2m}(r)\right)
= C^{\operatorname{GH}}_d(X).
\]
Applying Theorem \ref{theorem:GH-ALbound} and \cite{mpo1}, Lemma 6.12  we obtain
\[
C^{\mathrm{CM}}(X\times \mathbb{C}^m)
\leq \lim_{r \to \infty} \inf_{d \in \mathbb{N}}
\frac{C^{\operatorname{GH}}_d(X \times B^{2m}(r))}{d}
= \inf_{d \in \mathbb{N}} \frac{C^{\operatorname{GH}}_d(X)}{d}.
\]
Putting these together yields
\[
C^{\mathrm{CM}}(X)
\leq C^{\mathrm{CM}}(X\times \mathbb{C}^m)
\leq \inf_{d\in\mathbb{N}}\frac{C^{\operatorname{GH}}_d(X)}{d}.
\]
Moreover, equality holds when $X $ is a convex or concave toric domain, by the same argument as in the proof of Corollary \ref{CM=AL}. This completes the proof.
\end{proof}

\section{Cohen--Ganatra moduli spaces}\label{section:CG}

We recall the Cohen--Ganatra moduli spaces from \cite{yla}, Section 4.2, whose definition is inspired by closely related moduli spaces studied earlier by Cohen--Ganatra \cite{cg}. Here we will only provide the essential definitions and properties, and refer to \cite{yla} for details. See also \cite{sg1,jz}.

Let ${}_l\mathcal{R}_{k+1}^1$ be the moduli space of domains
\begin{equation}\label{eq:domain}
(S;z_0,\cdots,z_k,p_1,\cdots,p_l;\ell)
\end{equation}
modulo the group of automorphisms, where $S=D\setminus\{\zeta\}$ is a closed unit disc with an interior puncture $\zeta$, which will serve as an input. At $\zeta$ there is an asymptotic marker, namely a half-line $\ell\in T_\zeta D$. $z_0,\cdots,z_k\in\partial D$ are distinct marked points on the boundary, labeled in counterclockwise order. Moreover, there is a set of auxiliary marked points $p_1,\cdots,p_l\in S$ lying in the interior of $D$. After choosing the representative of an element of ${}_l\mathcal{R}_{k+1}^1$ with the puncture $\zeta$ fixed at the origin, and the boundary marked point $z_0$ fixed at $1$, the points $p_0,\cdots,p_l$ are required to be strictly radially ordered with norms in $(0,\frac{1}{2})$, i.e.
\begin{equation}\label{eq:radial}
0<|p_l|<\cdots<|p_1|<\frac{1}{2}.
\end{equation}
We require that the asymptotic marker $\ell$ at the origin points toward the point $p_l$, see Figure \ref{fig:domain}.

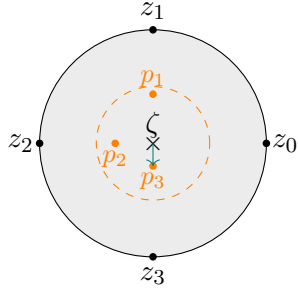
\begin{figure}
	\centering
	\begin{tikzpicture}
	\filldraw[draw=black,color={black!15},opacity=0.5] (0,0) circle (1.5);
	\draw (0,0) circle [radius=1.5];
	\draw [orange] [dashed] (0,0) circle [radius=0.75];
	\draw (0,0) node {$\times$};
	\node at (0,0.25) {$\zeta$};
	\draw [orange] (0,-0.3) node[circle,fill,inner sep=1pt] {};
	\draw [orange] (-0.5,0) node[circle,fill,inner sep=1pt] {};
	\draw [orange] (0,0.65) node[circle,fill,inner sep=1pt] {};
	\draw [teal] [->] (0,0) to (0,-0.3);
	\node [orange] at (0,-0.5) {\small $p_3$};
	\node [orange] at (-0.5,-0.2) {\small $p_2$};
	\node [orange] at (0,0.85) {\small $p_1$};
	\draw (1.5,0) node[circle,fill,inner sep=1pt] {};
	\draw (0,1.5) node[circle,fill,inner sep=1pt] {};
	\draw (-1.5,0) node[circle,fill,inner sep=1pt] {};
	\draw (0,-1.5) node[circle,fill,inner sep=1pt] {};
	\node at (1.75,0) {$z_0$};
	\node at (0,1.75) {$z_1$};
	\node at (-1.75,0) {$z_2$};
	\node at (0,-1.75) {$z_3$};
	\end{tikzpicture}
	\caption{An element in the moduli space ${}_3\mathcal{R}_4^1$}
	\label{fig:domain}
\end{figure}

To describe the compactifcation ${}_l\overline{\mathcal{R}}_{k+1}^1$, which is a manifold with corners, it would be convenient to introduce the following auxiliary moduli spaces.
\begin{itemize}
	\item ${}_l^{j,j+1}\mathcal{R}_{k+1}^1$ is the moduli space of the domains (\ref{eq:domain}), except that the condition (\ref{eq:radial}) is replaced with
	\begin{equation}
	|p_l|<\cdots<|p_{j+1}|=|p_j|<\cdots<|p_1|<\frac{1}{2}, \nonumber
	\end{equation}
	for some $1\leq j<l$. 
	\item ${}_{l-1}\mathcal{R}_{k+1}^{S^1}$ is the moduli space of the same domains, but with (\ref{eq:radial}) replaced with
	\begin{equation}
	|p_l|<\cdots<|p_1|=\frac{1}{2}. \nonumber
	\end{equation}
    By forgetting the marked point $p_1$, there is an abstract identification ${}_{l-1}\mathcal{R}_{k+1}^{S^1}\cong{}_{l-1}\mathcal{R}_{k+1}^1\times S^1$, under which the compactification ${}_{l-1}\overline{\mathcal{R}}_{k+1}^{S^1}$ is abstractly modeled by ${}_{l-1}\overline{\mathcal{R}}_{k+1}^1\times S^1$. In particular, the codimension $1$ boundary stratum ${}_{l-1}^{j,j+1}\overline{\mathcal{R}}_{k+1}^1\subset{}_{l-1}\overline{\mathcal{R}}_{k+1}^1$ for some $2\leq j\leq l-1$ corresponds to a codimension $1$ boundary stratum ${}_{l-1}^{j,j+1}\overline{\mathcal{R}}_{k+1}^{S^1}\subset\partial{}_{l-1}\overline{\mathcal{R}}_{k+1}^{S^1}$, where the $S^1$ factor describes the situation that $|p_1|=|p_2|=\frac{1}{2}$. Denote by
    \begin{equation}\label{eq:forget}
    \pi_j^{S^1}:{}_{l-1}^{j,j+1}\overline{\mathcal{R}}_{k+1}^{S^1}\rightarrow{}_{l-2}\overline{\mathcal{R}}_{k+1}^{S^1}
    \end{equation}
    the map which forgets $p_j$. 
    
    Note also that there is a free $\mathbb{Z}_{k+1}$-action on the moduli space ${}_{l-1}\overline{\mathcal{R}}_{k+1}^{S^1}$ generated by the map
    \begin{equation}\label{eq:permu}
    \kappa:{}_{l-1}\overline{\mathcal{R}}_{k+1}^{S^1}\rightarrow{}_{l-1}\overline{\mathcal{R}}_{k+1}^{S^1},
    \end{equation}
    which cyclically permutes the labels of the boundary marked points $z_0,\cdots,z_k$. It can be shown that this $\mathbb{Z}_{k+1}$-action is properly discontinuous.
    
    For an element $S$ of ${}_{l-1}\overline{\mathcal{R}}_{k+1}^{S^1}$, we say that $p_1$ \textit{points at a boundary point} $z_i$, for some $0\leq i\leq k$, if for a representative of $S$ with $\zeta$ fixed at the origin, the ray from $\zeta$ to $p_1$ points at $z_i$. Denote by ${}_{l-1}\overline{\mathcal{R}}_{k+1}^{S_i^1}\subset{}_{l-1}\overline{\mathcal{R}}_{k+1}^{S^1}$ the codimension $1$ locus where $p_1$ points at $z_i$. There is a bijection
    \begin{equation}
    \tau_i:{}_{l-1}\overline{\mathcal{R}}_{k+1}^{S_i^1}\rightarrow{}_{l-1}\overline{\mathcal{R}}_{k+1}^1 \nonumber
    \end{equation}
    defined as follows. When $l\geq2$, $\tau_i$ forgets the point $p_1$ on the circle $|z|=\frac{1}{2}$, and relabels the remaining auxiliary marked points $p_2,\cdots,p_l$ as $p_1,\cdots,p_{l-1}$. When $l=1$, $\tau_i$ is defined by cyclically permuting the boundary marked points until the original $z_i$ is now labeled $z_k$, and then forgetting $p_1$. Similarly, we say that $p_1$ \textit{points between $z_i$ and} $z_{i+1\textrm{ mod }k}$ if for such a representative, the ray from $\zeta$ to $p_1$ intersects the arc in $\partial S$ from $z_i$ to $z_{i+1\textrm{ mod }k}$. The (codimension $0$) locus in ${}_{l-1}\mathcal{R}_{k+1}^{S^1}$ where $p_1$ points between $z_i$ and $z_{i+1\textrm{ mod }k}$ is denoted by ${}_{l-1}\mathcal{R}_{k+1}^{S_{i,i+1}^1}$. 
\end{itemize}
Denote by ${}_l^{j,j+1}\overline{\mathcal{R}}_{k+1}^1$ and ${}_{l-1}\overline{\mathcal{R}}_{k+1}^{S^1}$ the  compactifcations of the moduli spaces introduced above.

It is clear that we have a decomposition of ${}_{l-1}\overline{\mathcal{R}}_{k+1}^{S^1}$ into its sectors ${}_{l-1}\overline{\mathcal{R}}_{k+1}^{S_{i,i+1}^1}$. To understand the situation better, we introduce a new moduli space
\begin{equation}
{}_{l-1}\mathcal{R}_{k+1,\tau_i}^1 \nonumber
\end{equation}
for each $i=0,\cdots,k$. This is the abstract moduli space of punctured discs with boundary marked points $z_0,\cdots,z_i,z_f,z_{i+1},\cdots,z_k$ aligned in counterclockwise order, where $z_f$ is marked as auxiliary, one interior puncture $\zeta$ at the origin, marked as an input, equipped with an asymptotic marker $\ell$ pointing in the direction of $p_l$ (or $z_f$ if $l=0$), and $l$ marked points $p_1,\cdots,p_l$ which are strictly radially ordered in the interior of the disc of radius $\frac{1}{2}$. Note that the compactification ${}_{l-1}\overline{\mathcal{R}}_{k+1,\tau_i}^1$ of ${}_{l-1}\mathcal{R}_{k+1,\tau_i}^1$ is abstractly isomorphic to ${}_{l-1}\overline{\mathcal{R}}_{k+2}^1$, except that $z_f$ is marked as auxiliary. Moreover, at any stratum of ${}_{l-1}\overline{\mathcal{R}}_{k+1,\tau_i}^1$:
\begin{itemize}
	\item we regard the main component (the one with the puncture $\zeta$) with $k'+2$ boundary marked points as an element of ${}_{l-1}\overline{\mathcal{R}}_{k'+1,\tau_j}^1$ for some $0\leq k'\leq k$ and $0\leq j\leq k'$ if it has $z_f$ as a boundary marked point, and an element of ${}_{l-1}\overline{\mathcal{R}}_{k'+2}^1$ if it does not;
	\item if a non-main disc component (the one without $\zeta$) has the boundary marked point $z_f$, we view it as an element of $\mathcal{R}_{k'+1,f_i}$ for some $0\leq k'\leq k$. This is the moduli space of discs with $k'+1$ boundary marked points, with the $i$-th point marked as forgotten. See \cite{sg1}, Appendix A.2 for its construction.
\end{itemize}
The moduli space ${}_{l-1}\overline{\mathcal{R}}_{k+1,\tau_i}$ is related to the sector ${}_{l-1}\overline{\mathcal{R}}_{k+1}^{S_{i,i+1}^1}\subset{}_{l-1}\overline{\mathcal{R}}_{k+1}^{S^1}$ via the \textit{auxiliary-rescaling map}
\begin{equation}\label{eq:aux-res}
\phi_f^i:{}_{l-1}\overline{\mathcal{R}}_{k+1,\tau_i}^1\rightarrow{}_{l-1}\overline{\mathcal{R}}_{k+1}^{S_{i,i+1}^1},
\end{equation}
which, for a representative of an element of ${}_{l-1}\overline{\mathcal{R}}_{k+1,\tau_i}$, adds a point $p_0$ on the line segment connecting $\zeta$ and $z_f$ with $|p_0|=\frac{1}{2}$ and deletes $z_f$. Then, we relabel the marked points $p_0,\cdots,p_{l-1}$ as $p_1,\cdots,p_{l}$, where now $|p_1|=\frac{1}{2}$. See Figure \ref{fig:auxre}. With the appropriate choices of orientations (cf. \cite{yla}, Appendix B), the map $\phi_f^i$ is an oriented diffeomorphism.

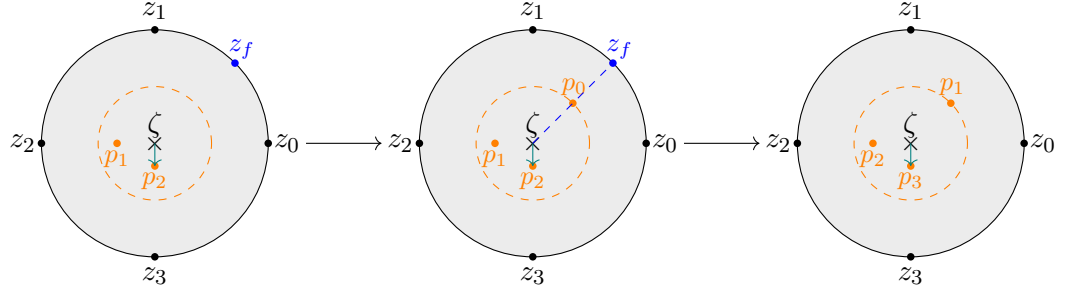
\begin{figure}
	\centering
	\begin{tikzpicture}
	\filldraw[draw=black,color={black!15},opacity=0.5] (0,1.9) circle (1.5);
	\draw (0,1.9) circle [radius=1.5];
	\draw [orange] [dashed] (0,1.9) circle [radius=0.75];
	\draw (0,1.9) node {$\times$};
	\node at (0,2.15) {$\zeta$};
	\draw [orange] (0,1.6) node[circle,fill,inner sep=1pt] {};
	\draw [orange] (-0.5,1.9) node[circle,fill,inner sep=1pt] {};
	\draw [teal] [->] (0,1.9) to (0,1.6);
	\node [orange] at (0,1.4) {\small $p_2$};
	\node [orange] at (-0.5,1.7) {\small $p_1$};
	\draw [blue] (1.06,2.96) node[circle,fill,inner sep=1pt] {};
	\node [blue] at (1.15,3.15) {\small $z_f$};
		
	\draw (1.5,1.9) node[circle,fill,inner sep=1pt] {};
	\draw (0,3.4) node[circle,fill,inner sep=1pt] {};
	\draw (-1.5,1.9) node[circle,fill,inner sep=1pt] {};
	\draw (0,0.4) node[circle,fill,inner sep=1pt] {};
	\node at (1.75,1.9) {$z_0$};
	\node at (0,3.65) {$z_1$};
	\node at (-1.75,1.9) {$z_2$};
	\node at (0,0.15) {$z_3$};
		
	\filldraw[draw=black,color={black!15},opacity=0.5] (5,1.9) circle (1.5);
	\draw (5,1.9) circle [radius=1.5];
	\draw [orange] [dashed] (5,1.9) circle [radius=0.75];
	\draw (5,1.9) node {$\times$};
	\node at (5,2.15) {$\zeta$};
	\draw [orange] (5,1.6) node[circle,fill,inner sep=1pt] {};
	\draw [orange] (4.5,1.9) node[circle,fill,inner sep=1pt] {};
	\draw [orange] (5.53,2.43) node[circle,fill,inner sep=1pt] {};
	\draw [teal] [->] (5,1.9) to (5,1.6);
	\node [orange] at (5,1.4) {\small $p_2$};
	\node [orange] at (4.5,1.7) {\small $p_1$};
	\node [orange] at (5.55,2.65) {\small $p_0$};
	\draw [blue] (6.06,2.96) node[circle,fill,inner sep=1pt] {};
	\node [blue] at (6.15,3.15) {\small $z_f$};
		
	\draw (6.5,1.9) node[circle,fill,inner sep=1pt] {};
	\draw (5,3.4) node[circle,fill,inner sep=1pt] {};
	\draw (3.5,1.9) node[circle,fill,inner sep=1pt] {};
	\draw (5,0.4) node[circle,fill,inner sep=1pt] {};
	\node at (6.75,1.9) {$z_0$};
	\node at (5,3.65) {$z_1$};
	\node at (3.25,1.9) {$z_2$};
	\node at (5,0.15) {$z_3$};
	\draw [blue] [dashed] (5,1.9) to (6.06,2.96);
		
	\filldraw[draw=black,color={black!15},opacity=0.5] (10,1.9) circle (1.5);
	\draw (10,1.9) circle [radius=1.5];
	\draw [orange] [dashed] (10,1.9) circle [radius=0.75];
	\draw (10,1.9) node {$\times$};
	\node at (10,2.15) {$\zeta$};
	\draw [orange] (10,1.6) node[circle,fill,inner sep=1pt] {};
	\draw [orange] (9.5,1.9) node[circle,fill,inner sep=1pt] {};
	\draw [orange] (10.53,2.43) node[circle,fill,inner sep=1pt] {};
	\draw [teal] [->] (10,1.9) to (10,1.6);
	\node [orange] at (10,1.4) {\small $p_3$};
	\node [orange] at (9.5,1.7) {\small $p_2$};
	\node [orange] at (10.55,2.65) {\small $p_1$};
		
	\draw (11.5,1.9) node[circle,fill,inner sep=1pt] {};
	\draw (10,3.4) node[circle,fill,inner sep=1pt] {};
	\draw (8.5,1.9) node[circle,fill,inner sep=1pt] {};
	\draw (10,0.4) node[circle,fill,inner sep=1pt] {};
	\node at (11.75,1.9) {$z_0$};
	\node at (10,3.65) {$z_1$};
	\node at (8.25,1.9) {$z_2$};
	\node at (10,0.15) {$z_3$};
		
	\draw [->] (2,1.9) to (3,1.9);
	\draw [->] (7,1.9) to (8,1.9);
	\end{tikzpicture}
	\caption{The definition of the auxiliary-rescaling map $\phi_f^0:{}_2\overline{\mathcal{R}}_{4,\tau_0}^1\rightarrow{}_2\overline{\mathcal{R}}_4^{S_{0,1}^1}$}
	\label{fig:auxre}
\end{figure}

The codimension $1$ boundary of ${}_l\overline{\mathcal{R}}_{k+1}^1$ is covered by the natural inclusions of the following strata (cf. \cite{yla}, Proposition 40):
\begin{equation}\label{eq:dom-bdy1}
{}_j\overline{\mathcal{M}}\times{}_{l-j}\overline{\mathcal{R}}_{k+1}^1,\textrm{ }1\leq j\leq l,
\end{equation}
\begin{equation}\label{eq:dom-bdy2}
{}_l^{j,j+1}\overline{\mathcal{R}}_{k+1}^1,\textrm{ }1\leq j\leq l-1,
\end{equation}
\begin{equation}\label{eq:dom-bdy3}
{}_{l-1}\overline{\mathcal{R}}_{k+1}^{S^1},
\end{equation}
\begin{equation}\label{eq:dom-bdy4}
{}_l\overline{\mathcal{R}}_{k_1+1}^1\textrm{ }{{}_i\times}_0\textrm{ }\overline{\mathcal{R}}_{k_2+1},\textrm{ }k_1\geq0,k_2\geq2,k_1+k_2=k+1,1\leq i\leq k_1,
\end{equation}
\begin{equation}\label{eq:dom-bdy5}
\overline{\mathcal{R}}_{k_1+1}\textrm{ }{{}_i\times}_0\textrm{ }{}_l\overline{\mathcal{R}}_{k_2+1}^1,\textrm{ }k_1\geq2,k_2\geq0,k_1+k_2=k+1, 1\leq i\leq k_1.
\end{equation}
In the above, ${}_j\mathcal{M}$ is the moduli space of $j$-point angle decorated cylinders introduced by Ganatra in \cite{sg1}, Section 4.3, and ${}_j\overline{\mathcal{M}}$ is its compactification. $\mathcal{R}_{k+1}$ is the moduli space of closed unit discs $D$ with $k+1$ distinct marked points $z_0,\cdots,z_k\in\partial D$ aligned in counterclockwise order, modulo the automorphism group $\mathit{PSL}(2,\mathbb{R})$ of $D$. In (\ref{eq:dom-bdy4}), the notation ${{}_i\times_0}$ means that the disc breaking happens at $z_i\in\partial S$, with the nodal point playing the role of $z_0$ on the disc bubble in $\overline{\mathcal{R}}_{k_2+1}$, and similarly for (\ref{eq:dom-bdy5}). The strata in (\ref{eq:dom-bdy1}) come from the real blow-ups, when the modulus of the marked points $p_{l-j+1},\cdots,p_l$ tend to $0$, see Figure \ref{fig:real-blp}.

\begin{figure}
	\centering
	\begin{tikzpicture}
	\filldraw[draw=black,color={black!15},opacity=0.5] (0,0) circle (1.5);
	\draw (0,0) circle [radius=1.5];
	\draw [orange] [dashed] (0,0) circle [radius=0.75];
	\draw (0,0) node {$\times$};
	\node at (0,0.25) {$\zeta$};
	\draw [orange] (0,-0.3) node[circle,fill,inner sep=1pt] {};
	\draw [orange] (-0.5,0) node[circle,fill,inner sep=1pt] {};
	\draw [orange] (0,0.65) node[circle,fill,inner sep=1pt] {};
	\draw [teal] [->] (0,0) to (0,-0.3);
	\node [orange] at (0,-0.5) {\small $p_3$};
	\node [orange] at (-0.5,-0.2) {\small $p_2$};
	\node [orange] at (0,0.85) {\small $p_1$};
		
	\filldraw[draw=black,color={black!15},opacity=0.5] (6,0) circle (1.5);
	\draw (6,0) circle [radius=1.5];
	\draw [orange] [dashed] (6,0) circle [radius=0.75];
	\draw (6,0) node {$\times$};
	\node at (6,-0.25) {input};
	\draw [orange] (6.4,0.5) node[circle,fill,inner sep=1pt] {};
	\node [orange] at (6.6,0.75) {\small $p_3$};
	\draw [teal] [->] (6,0) to (6.3,0.375);
		
	\filldraw[draw=black,color={black!15},opacity=0.5] (6,3.75) circle (1.5);
	\draw (6,3.75) circle [radius=1.5];
	\draw (6,3.3) node {$\times$};
	\draw (6,4.2) node {$\times$};
	\draw [blue] [dashed] (6,3.3) to (6,0);
	\draw [orange] (6.2,2.55) node[circle,fill,inner sep=1pt] {};
	\draw [orange] (5.5,2.85) node[circle,fill,inner sep=1pt] {};
	\node [orange] at (6.25,2.75) {\small $p_2$};
	\node [orange] at (5.5,3.05) {\small $p_1$};
	\draw [dashed] (6,3.75) ellipse (1.5 and 0.45);
	\node at (6,4.45) {$\zeta$};
	\node at (6,3.55) {output};
		
	\draw [->] (2,0) to (4,0);
	\node at (3,0.25) {$|p_2|\rightarrow0$};
		
	\draw (1.5,0) node[circle,fill,inner sep=1pt] {};
	\draw (0,1.5) node[circle,fill,inner sep=1pt] {};
	\draw (-1.5,0) node[circle,fill,inner sep=1pt] {};
	\draw (0,-1.5) node[circle,fill,inner sep=1pt] {};
	\node at (1.75,0) {$z_0$};
	\node at (0,1.75) {$z_1$};
	\node at (-1.75,0) {$z_2$};
	\node at (0,-1.75) {$z_3$};
		
	\draw (7.5,0) node[circle,fill,inner sep=1pt] {};
	\draw (6,1.5) node[circle,fill,inner sep=1pt] {};
	\draw (4.5,0) node[circle,fill,inner sep=1pt] {};
	\draw (6,-1.5) node[circle,fill,inner sep=1pt] {};
	\node at (7.75,0) {$z_0$};
	\node at (6,1.75) {$z_1$};
	\node at (4.25,0) {$z_2$};
	\node at (6,-1.75) {$z_3$};
	\end{tikzpicture}
	\caption{An element belonging to the boundary stratum ${}_2\overline{\mathcal{M}}\times{}_1\overline{\mathcal{R}}_{4}^1$ of ${}_3\overline{\mathcal{R}}_{4}^1$.}
	\label{fig:real-blp}
\end{figure}
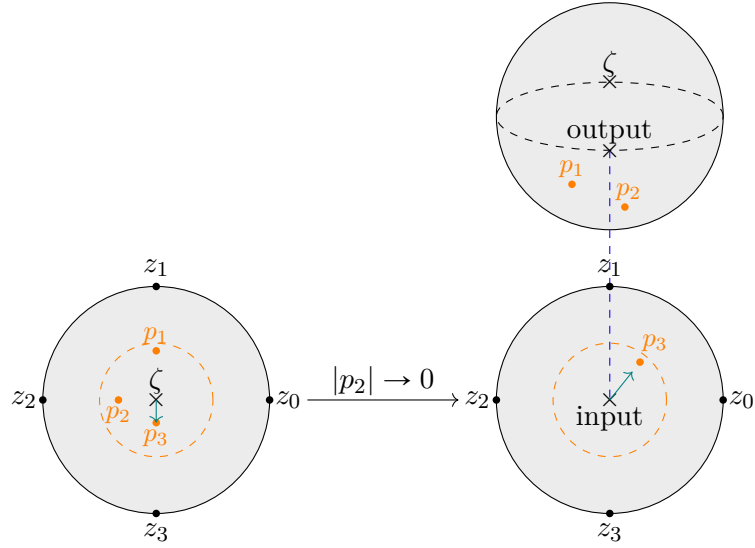

Let $(X,\lambda)$ be a Liouville domain. A function $H_t:S^1\times X\rightarrow\mathbb{R}$ is an \textit{admissible Hamiltonian} if $H_t=H+F_t$, where $H:X\rightarrow\mathbb{R}$ is an autonomous Hamiltonian which equals $r^2$ on the collar neighborhood $(1-\delta,1]\times\partial X$ for some sufficiently small $\delta>0$ and $C^2$-small away from the collar $(1-\delta',1]\times\partial X$ for some $\delta'>\delta$, and a time-dependent perturbation $F_t:S^1\times X\rightarrow\mathbb{R}$. Denote by $\mathcal{H}(X)$ the set of admissible Hamiltonians $H_t$ so that all the $1$-periodic orbits of the Hamiltonian vector field $X_{H_t}$ are non-degenerate. An almost complex structure $J_t:S^1\times\mathit{TX}\rightarrow\mathit{TX}$ is called \textit{contact type} if $dr\circ J_t=-\lambda$ in the collar $(1-\delta,1]\times\partial X$. Denote by $\mathcal{J}(X)$ the set of $d\lambda$-compatible almost complex structures on $X$ which are of contact type. Before defining the Cohen--Ganatra moduli spaces, we need to choose the corresponding Floer data first.

\begin{definition}\label{definition:data}
A Floer datum for an element $(S;z_0,\cdots,z_k,p_1,\cdots,p_l;\ell)\in{}_l\mathcal{R}_{k+1}^1$ consists of the following:
	\begin{itemize}
	\item A positive cylindrical end
	\begin{equation}
	\varepsilon^+:[0,\infty)\times S^1\rightarrow S,\textrm{ }(s,t)\mapsto\left(s+(p_l)_s+\eta,t\right) \nonumber
	\end{equation}
	for some $\eta>0$, where $(p_l)_s$ is the $s\in[0,\infty)$ coordinate of the point $p_l\in[0,\infty)\times S^1$.
	\item A sub-closed 1-form $\gamma_S\in\Omega^1(S)$ such that $\gamma_S\equiv0$ near $\partial S$ and $(\varepsilon^+)^\ast\gamma_S=dt$.
	\item A domain-dependent Hamiltonian function $H_S:S\times X\rightarrow\mathbb{R}$ satisfying
	\begin{equation}\label{eq:H}
	(\varepsilon^+)^\ast H_S=H_t
	\end{equation}
	for some $H_t\in\mathcal{H}(X)$, and
	\begin{equation}
	H_S\equiv0\textrm{ near }\partial S. \nonumber
	\end{equation}
	\item A domain-dependent almost complex structure $J_S:S\times\mathit{TX}\rightarrow\mathit{TX}$ such that
	\begin{equation}
	(\varepsilon^+)^\ast J_S=J_t \nonumber
	\end{equation}
	for some $J_t\in\mathcal{J}(X)$, and
	\begin{equation}\label{eq:J}
	J_S\equiv J\textrm{ near }\partial S,
	\end{equation}
	for some fixed convex almost complex structure $J$.
	\end{itemize}
\end{definition}

Similarly, one can define Floer data for the moduli spaces ${}_l^{j,j+1}\mathcal{R}_{k+1}^1$, ${}_{l-1}\mathcal{R}_{k+1}^{S^1}$ and ${}_{l-1}\mathcal{R}_{k+1,\tau_i}^1$. We will inductively choose the Floer data on ${}_l\overline{\mathcal{R}}_{k+1}^1$ so that certain consistency conditions are satisfied. To do so, we first choose the Floer data on ${}_{l-1}\overline{\mathcal{R}}_{k+1}^{S^1}$ so that
\begin{itemize}
	\item The Floer datum on ${}_{l-1}\overline{\mathcal{R}}_{k+1}^{S^1}$ is $\mathbb{Z}_{k+1}$-equivariant under the cyclic permutation map $\kappa$ (cf. (\ref{eq:permu})).
	\item On the boundary stratum ${}_{l-1}^{j,j+1}\overline{\mathcal{R}}_{k+1}^{S^1}\subset\partial{}_{l-1}\overline{\mathcal{R}}_{k+1}^{S^1}$, the Floer datum is conformally equivalent to the one pulled back from ${}_{l-2}\overline{\mathcal{R}}_{k+1}^{S^1}$ via the forgetful map $\pi_j^{S^1}$ (cf. (\ref{eq:forget})).
\end{itemize}

\begin{definition}
A Cohen--Ganatra Floer datum is a an inductive sequence of choices, for every $k,l\in\mathbb{Z}_{\geq0}$, of Floer data for representatives of elements of ${}_l\overline{\mathcal{R}}_{k+1}^1$ in the sense of Definition \ref{definition:data}, which vary smoothly over the moduli spaces $\left\{{}_l\overline{\mathcal{R}}_{k+1}^1\right\}_{k\geq0,l\geq1}$, such that the following conditions are satisfied.
\begin{itemize}
	\item[(i)] The choice of Floer datum on any boundary stratum of ${}_l\overline{\mathcal{R}}_{k+1}^1$ should agree with the inductively chosen Floer datum along the boundary stratum which has already been fixed.
	\item[(ii)] For $l\geq1$, near the boundary strata in (\ref{eq:dom-bdy2}), the Floer data are conformally equivalent to the ones obtained by puling back from ${}_{l-1}\overline{\mathcal{R}}_{k+1}^1$ via the maps $\pi_j:{}_l^{j,j+1}\overline{\mathcal{R}}_{k+1}^1\rightarrow{}_{l-1}\overline{\mathcal{R}}_{k+1}^1$ which forget the marked point $p_j$.
	\item[(iii)] On the codimension $1$ loci ${}_{l-1}\overline{\mathcal{R}}_{k+1}^{S_i^1}\subset{}_{l-1}\overline{\mathcal{R}}_{k+1}^{S^1}$ where $p_1$ points at $z_i$, the Floer datum should agree with the pullback by $\tau_i$ of the existing Floer datum on ${}_{l-1}\overline{\mathcal{R}}_{k+1}^1$.
\end{itemize}
\end{definition}

By \cite{yla}, Proposition 44, Cohen--Ganatra Floer datum exists. Fix such a datum, a $1$-periodic orbit $x$ of $X_{H_t}$ and a Lagrangian submanifold $L\subset\mathrm{int}(X)$ in the interior of $X$, define the \textit{Cohen--Ganatra moduli space}
\begin{equation}
{}_l\mathcal{R}_{k+1}^1(x,L) \nonumber
\end{equation}
to be the space of pairs
\begin{equation}
\left((S,z_0,\cdots,z_k,p_1,\cdots,p_l,\ell),u\right), \nonumber
\end{equation}
where $(S,z_0,\cdots,z_k,p_1,\cdots,p_l,\ell)\in{}_l\mathcal{R}_{k+1}^1$ and $u:S\rightarrow X$ is a smooth map satisfying
\begin{equation}\label{eq:CR}
	\left\{\begin{array}{l}
	(du-X_{H_S}\otimes\gamma_S)^{0,1}=0, \\
	u(\partial S)\subset L, \\
	\lim_{s\rightarrow\infty}(\varepsilon^+)^\ast u(s,\cdot)=x,
	\end{array}\right.
\end{equation}
where the $(0,1)$-part in the Floer equation is taken with respect to the domain-dependent almost complex structure $J_S$. We can decompose the moduli space ${}_l\mathcal{R}_{k+1}^1(x,L)$ according to the homotopy classes of the maps $u$, which gives
\begin{equation}
{}_l\mathcal{R}_{k+1}^1(x,L)=\bigsqcup_{\ring{\bar{a}}\in\pi_2(X,x,L)}{}_l\mathcal{R}_{k+1}^1(x,L,\ring{\bar{a}}), \nonumber
\end{equation}
where $\pi_2(X,x,L)$ denotes the set of homotopy classes of maps $u:S\rightarrow X$ with boundary on $L$ and asymptotic to $x$ at its positive puncture. We recall the following basic fact (cf. \cite{yla}, Lemma 45, (ii)).

\begin{lemma}\label{lemma:action}
${}_l\mathcal{R}_{k+1}^1(x,L,\ring{\bar{a}})\neq\emptyset$ implies that
\begin{equation}
\lambda(a)+\left\vert A_{H_t}(x)\right\vert\geq0, \nonumber
\end{equation}
where
\begin{equation}
A_{H_t}(x)=-\int_{S^1}x^\ast\lambda+\int_{S^1}H_t(x(t))dt \nonumber
\end{equation}
is the action of the orbit $x$ with respect to the Hamiltonian $H_t:S^1\times X\rightarrow\mathbb{R}$ in (\ref{eq:H}).
\end{lemma}

In a similar way, one can define the moduli spaces
\begin{equation}
{}_{l-1}\mathcal{R}_{k+1}^{S^1}(x,L,\ring{\bar{a}}), \nonumber
\end{equation}
\begin{equation}\label{eq:sector}
{}_{l-1}\mathcal{R}_{k+1,\tau_i}^1(x,L,\ring{\bar{a}}),\textrm{ }0\leq i\leq k,
\end{equation}
\begin{equation}
{}_l^{j,j+1}\mathcal{R}_{k+1}^1(x,L,\ring{\bar{a}}),\textrm{ }1\leq j\leq l-1, \nonumber
\end{equation}
using the auxiliary moduli spaces of domains ${}_{l-1}\mathcal{R}_{k+1}^{S^1}$, ${}_{l-1}\mathcal{R}_{k+1,\tau_i}^1$ and ${}_l^{j,j+1}\mathcal{R}_{k+1}^1$ introduced earlier.

In general, due to the non-exactness of $L$, the transversality of these moduli spaces cannot be achieved with standard perturbations of the domain-dependent almost complex structure $J_S$. Instead, we need the virtual perturbation scheme developed by Fukaya--Oh--Ohta--Ono \cite{fona,fooo4}.  More precisely, one can construct Kuranishi structures on the moduli spaces ${}_l\mathcal{R}_{k+1}^1(x,L,\ring{\bar{a}})$, ${}_{l-1}\mathcal{R}_{k+1}^{S^1}(x,L,\ring{\bar{a}})$, ${}_{l-1}\mathcal{R}_{k+1,\tau_i}^1(x,L,\ring{\bar{a}})$ and ${}_l^{j,j+1}\mathcal{R}_{k+1}^1(x,L,\ring{\bar{a}})$ so that they become oriented K-spaces. 

Recall that for $\bar{a}\in\pi_2(X,L)$ we have a moduli space $\mathcal{R}_{k+1}(L,\bar{a})$ parametrizing pairs $((D,z_0,\cdots,z_k),u)$, where $(D,z_0,\cdots,z_k)\in\mathcal{R}_{k+1}$ and $u:(D,\partial D)\rightarrow(X,L)$ is a $J$-holomorphic map in the class $\bar{a}$. Note that the Gromov compactification $\overline{\mathcal{R}}_{k+1}(L,\bar{a})$ is governed combinatorially by a decorated rooted ribbon tree $(T,B)$ (cf. \cite{kic}, Definition 7.18), where $T$ is a connected tree with the set of interior vertices $V_\mathrm{int}(T)$ and the set of interior edges $E_\mathrm{int}(T)$, and $B:V_\mathrm{int}(T)\rightarrow\pi_2(X,L)$ is a map which associates every interior vertex of $T$ to a relative homotopy class in $\pi_2(X,L)$. Denote by $\mathcal{T}(k+1,\bar{a})$ the set of decorated rooted ribbon trees with $k+1$ exterior vertices and $\sum_{v\in V_\mathrm{int}(T)}B(v)=\bar{a}$. For every $(T,B)\in\mathcal{T}(k+1,\beta)$, there is an \textit{interior evaluation map}
\begin{equation}\label{eq:ev-int}
\mathrm{ev}_\mathrm{int}:\prod_{v\in V_{\mathrm{int}}(T)}\mathcal{R}_{k_v+1}(L,B(v))\rightarrow\prod_{e\in E_{\mathit{int}}(T)}L^2
\end{equation}
defined by evaluating at the endpoints of each edge $e\in E_\mathrm{int}(T)$.

As in the case of $\overline{\mathcal{R}}_{k+1}(L,\bar{a})$, the strata in the Gromov compactifications ${}_l\overline{\mathcal{R}}_{k+1}^1(x,L,\ring{\bar{a}})$, ${}_l^{j,j+1}\overline{\mathcal{R}}_{k+1}^1(x,L,\ring{\bar{a}})$, ${}_{l-1}\overline{\mathcal{R}}_{k+1}^{S^1}(x,L,\ring{\bar{a}})$ and ${}_{l-1}\overline{\mathcal{R}}_{k+1,\tau_i}^1(x,L,\ring{\bar{a}})$ corresponding to disc-breaking are modeled on \textit{decorated rooted ribbon trees with a single puncture}, see \cite{yla}, Definition 46. Roughly speaking, this is a triple $(T,\ring{B},v_0)$, where $T$ is a tree with a distinguished interior vertex $v_0\in V_\mathrm{int}(T)$, such that the map $\ring{B}$ associates to $v_0$ a class in $\pi_2(X,x,L)$, and to other interior vertices a class in $\pi_2(X,L)$ with non-negative symplectic area. Denote by $\mathcal{T}(k+1,\ring{\bar{a}})$ the set of decorated rooted ribbon trees with a single puncture $(T,\ring{B},v_0)$ with $k+1$ exterior vertices and $\sum_{v\in V_\mathrm{int}(T)}\ring{B}(v)=\ring{\bar{a}}$. Similar to (\ref{eq:ev-int}), one can define interior evaluation maps on the strata in the compactified moduli space ${}_l\overline{\mathcal{R}}_{k+1}^1(x,L,\ring{\bar{a}})$.

Let $\varepsilon>0$ be chosen so that $2\varepsilon$ is less than the minimal symplectic area of $J$-holomorphic discs $u:(D,\partial D)\rightarrow(X,L)$, where the almost complex structure $J$ is fixed to be the one in (\ref{eq:J}). Choose $U\in\mathbb{N}$ so that
\begin{equation}
\varepsilon(U-1)\geq\left\vert A_{H_t}(x)\right\vert. \nonumber
\end{equation}
We summarize the properties of the compactified moduli spaces ${}_l\overline{\mathcal{R}}_{k+1}^1(x,L,\ring{\bar{a}})$, ${}_{l-1}\overline{\mathcal{R}}_{k+1}^{S^1}(x,L,\ring{\bar{a}})$, ${}_{l-1}\overline{\mathcal{R}}_{k+1,\tau_i}^1(x,L,\ring{\bar{a}})$ and ${}_l^{j,j+1}\overline{\mathcal{R}}_{k+1}^1(x,L,\ring{\bar{a}})$ in the following theorem (cf. \cite{yla}, Theorem 48 and Remark 49).

\begin{theorem}\label{theorem:moduli}
For every $k,m,l\in\mathbb{Z}_{\geq0}$, $\bar{a}\in\pi_2(X,L)$, $\ring{\bar{a}}\in\pi_2(X,x,L)$, $a=\partial\bar{a}=\partial\ring{\bar{a}}\in H_1(L;\mathbb{Z})$ and $P=\{m\}$ or $[m,m+1]$, we have the following data.
\begin{itemize}
	\item[(i)] Compact, oriented, admissible K-spaces 
	\begin{equation}\label{eq:m0}
	\overline{\mathcal{R}}_{k+1}(L,\bar{a};P),\textrm{ where }\lambda(a)<(m-k+1)\varepsilon,
	\end{equation}
	\begin{equation}\label{eq:m2}
	{}_l\overline{\mathcal{R}}_{k+1}^1(x,L,\ring{\bar{a}};P),\textrm{ where }\lambda(a)<(m-k-U)\varepsilon,
	\end{equation}
	\begin{equation}\label{eq:m3}
	{}_l^{j,j+1}\overline{\mathcal{R}}_{k+1}^1(x,L,\ring{\bar{a}};P),\textrm{ }1\leq j\leq l-1,\textrm{ where }\lambda(a)<(m-k-U)\varepsilon,
	\end{equation}
	\begin{equation}\label{eq:m4}
	{}_{l-1}\overline{\mathcal{R}}_{k+1}^{S^1}(x,L,\ring{\bar{a}};P),\textrm{ where }\lambda(a)<(m-k-U)\varepsilon,
	\end{equation}
	\begin{equation}\label{eq:m5}
	{}_{l-1}\overline{\mathcal{R}}_{k+1,\tau_i}^1(x,L,\ring{\bar{a}};P),\textrm{ }0\leq i\leq k,\textrm{ where }\lambda(a)<(m-k-U)\varepsilon,
	\end{equation}
	whose underlying topological spaces are $P\times\overline{\mathcal{R}}_{k+1}(L,\bar{a})$, $P\times{}_l\overline{\mathcal{R}}_{k+1}^1(x,L,\ring{\bar{a}})$, $P\times{}_l^{j,j+1}\overline{\mathcal{R}}_{k+1}^1(x,L,\ring{\bar{a}})$, $P\times{}_{l-1}\overline{\mathcal{R}}_{k+1}^{S^1}(x,L,\ring{\bar{a}})$ and $P\times{}_{l-1}\overline{\mathcal{R}}_{k+1,\tau_i}^1(x,L,\ring{\bar{a}})$, respectively. The dimensions of these K-spaces are
	\begin{equation}
	\dim\overline{\mathcal{R}}_{k+1}(L,\bar{a};P)=\mu(a)+n+k-2+\dim P, \nonumber
	\end{equation}
	\begin{equation}\label{eq:dim}
	\dim{}_l\overline{\mathcal{R}}_{k+1}^1(x,L,\ring{\bar{a}};P)=\mu(a)+k+2l+\mathit{CZ}(x)+\dim P,
	\end{equation}
	\begin{equation}
	\begin{split}
	\dim{}_l^{j,j+1}\overline{\mathcal{R}}_{k+1}^1(x,L,\ring{\bar{a}};P)&=\dim{}_{l-1}\overline{\mathcal{R}}_{k+1}^{S^1}(x,L,\ring{\bar{a}};P)=\dim{}_{l-1}\overline{\mathcal{R}}_{k+1,\tau_i}^1(x,L,\ring{\bar{a}};P) \\
	&=\mu(a)+k+2l-1+\mathit{CZ}(x)+\dim P, \nonumber
	\end{split}
	\end{equation}
	where $\mathit{CZ}(x)$ denotes the Conley-Zehnder index of $x$.
	\item[(ii)] Corner-stratified strongly smooth evaluation maps
	\begin{equation}\label{eq:eva0}
	\mathrm{ev}^{\mathcal{R},P}:\overline{\mathcal{R}}_{k+1}(L,\beta;P)\rightarrow P\times L^{k+1},
	\end{equation}
	\begin{equation}\label{eq:eva2}
	{}_l\mathrm{ev}^{\mathcal{R},P}:{}_l\overline{\mathcal{R}}_{k+1}^1(x,L,\ring{\bar{a}};P)\rightarrow P\times L^{k+1},
	\end{equation}
	\begin{equation}\label{eq:eva3}
	{}_l^{j,j+1}\mathrm{ev}^{\mathcal{R},P}:{}_l^{j,j+1}\overline{\mathcal{R}}_{k+1}^1(x,L,\ring{\bar{a}};P)\rightarrow P\times L^{k+1},
	\end{equation}
	\begin{equation}\label{eq:eva4}
	{}_{l-1}\mathrm{ev}^{S^1,P}:{}_{l-1}\overline{\mathcal{R}}_{k+1}^{S^1}(x,L,\ring{\bar{a}};P)\rightarrow P\times L^{k+1},
	\end{equation}
	\begin{equation}\label{eq:eva5}
	{}_{l-1}\mathrm{ev}_i^{\mathcal{R},P}:{}_{l-1}\overline{\mathcal{R}}_{k+1,\tau_i}^1(x,L,\ring{\bar{a}};P)\rightarrow P\times L^{k+1},
	\end{equation}
    whose underlying set-theoretic maps are $\mathrm{id}_P\times\mathrm{ev}^{\mathcal{R}}$, $\mathrm{id}_P\times{}_l\mathrm{ev}^\mathcal{R}$, $\mathrm{id}_P\times{}_l^{j,j+1}\mathrm{ev}^\mathcal{R}$, $\mathrm{id}_P\times{}_{l-1}\mathrm{ev}^{S^1}$ and $\mathrm{id}_P\times{}_{l-1}\mathrm{ev}_i^\mathcal{R}$, respectively, where the evaluation maps in the second factor are obvious evaluation maps at the boundary marked points. Note that for the moduli space ${}_{l-1}\overline{\mathcal{R}}_{k+1,\tau_i}^1(x,L,\ring{\bar{a}};P)$, the boundary point $z_f$ being auxiliary means that it is forgotten under the evaluation maps.
    \item[(iii)] Orientation-preserving isomorphisms of admissible K-spaces:
    \begin{equation}\label{eq:bd2}
    \begin{split}
    \partial{}_l\overline{\mathcal{R}}_{k+1}^1\left(x,L,\ring{\bar{a}};\{m\}\right)&\cong\bigsqcup_{\substack{k_1+k_2=k+1\\1\leq i\leq k_1\\ \ring{\bar{a}}_1+\bar{a}_2=\ring{\bar{a}}}}(-1)^{\varepsilon_1}{}_l\overline{\mathcal{R}}_{k_1+1}^1\left(x,L,\ring{\bar{a}}_1;\{m\}\right)\textrm{ }{{}_i\times_0}\textrm{ }\overline{\mathcal{R}}_{k_2+1}\left(L,\bar{a}_2;\{m\}\right) \\
    &\sqcup\bigsqcup_{\substack{k_1+k_2=k+1\\1\leq i\leq k_1\\ \bar{a}_1+\ring{\bar{a}}_2=\ring{\bar{a}}}}(-1)^{\varepsilon_2}\overline{\mathcal{R}}_{k_1+1}\left(L,\bar{a}_1;\{m\}\right)\textrm{ }{{}_i\times_0}\textrm{ }{}_l\overline{\mathcal{R}}^1_{k_2+1}\left(x,L,\ring{\bar{a}}_2;\{m\}\right) \\
    &\sqcup\bigsqcup_{0\leq j\leq l}(-1)^{\varepsilon_{3,j}}{}_j\overline{\mathcal{M}}\left(x,y_j;\{m\}\right)\times{}_{l-j}\overline{\mathcal{R}}_{k+1}^1\left(y_j,L,\ring{\bar{a}};\{m\}\right) \\
    &\sqcup\bigsqcup_{1\leq j\leq l-1}(-1)^{\varepsilon_4}{}_l^{j,j+1}\overline{\mathcal{R}}_{k+1}^1\left(x,L,\ring{\bar{a}};\{m\}\right) \\
    &\sqcup(-1)^{\varepsilon_5}{}_{l-1}\overline{\mathcal{R}}_{k+1}^{S^1}\left(x,L,\ring{\bar{a}};\{m\}\right),
    \end{split}
    \end{equation}
    \begin{equation}\label{eq:bd5}
    \begin{split}
    &\partial{}_l\overline{\mathcal{R}}_{k+1}^1\left(x,L,\ring{\bar{a}};[m,m+1]\right) \\
    &\cong(-1)^{\varepsilon_{6}}{}_l\overline{\mathcal{R}}_{k+1}^1\left(x,L,\ring{\bar{a}};\{m\}\right)\sqcup(-1)^{\varepsilon_{7}}{}_l\overline{\mathcal{R}}_{k+1}^1\left(x,L,\ring{\bar{a}};\{m+1\}\right) \\
    &\sqcup\bigsqcup_{\substack{k_1+k_2=k+1\\1\leq i\leq k_1\\ \ring{\bar{a}}_1+\bar{a}_2=\ring{\bar{a}}}}(-1)^{\varepsilon_{8}}{}_l\overline{\mathcal{R}}_{k_1+1}^1\left(x,L,\ring{\bar{a}}_1;[m,m+1]\right)\textrm{ }{{}_i\times_0}\textrm{ }\overline{\mathcal{R}}_{k_2+1}\left(L,\bar{a}_2;[m,m+1]\right) \\
    &\sqcup\bigsqcup_{\substack{k_1+k_2=k+1\\1\leq i\leq k_1\\ \bar{a}_1+\ring{\bar{a}}_2=\ring{\bar{a}}}}(-1)^{\varepsilon_{9}}\overline{\mathcal{R}}_{k_1+1}\left(L,\bar{a}_1;[m,m+1]\right)\textrm{ }{{}_i\times_0}\textrm{ }{}_l\overline{\mathcal{R}}_{k_2+1}\left(x,L,\ring{\bar{a}}_2;[m,m+1]\right)\\
    &\sqcup\bigsqcup_{0\leq j\leq l}(-1)^{\varepsilon_{10,j}}{}_j\overline{\mathcal{M}}\left(x,y_j;[m,m+1]\right)\times{}_{l-j}\overline{\mathcal{R}}_{k+1}^1\left(y_j,L,\ring{\bar{a}};[m,m+1]\right) \\
    &\sqcup\bigsqcup_{1\leq j\leq l-1}(-1)^{\varepsilon_{11}}{}_l^{j,j+1}\overline{\mathcal{R}}_{k+1}^1\left(x,L,\ring{\bar{a}};[m,m+1]\right)\sqcup(-1)^{\varepsilon_{12}}{}_{l-1}\overline{\mathcal{R}}_{k+1}^{S^1}\left(x,L,\ring{\bar{a}};[m,m+1]\right),
    \end{split}
    \end{equation}
    where
    \begin{equation}
    \begin{split}
    &\varepsilon_1=(k_1-i)(k_2-1)+n+k,\textrm{ }\varepsilon_2=(k_1-i)(k_2-1)+n+1, \\
    &\varepsilon_{3,j}=n+|y_j|,\textrm{ }0\leq j\leq l,\textrm{ }\varepsilon_4=0,\textrm{ }\varepsilon_5=0, \\
    &\varepsilon_{6}=1,\textrm{ }\varepsilon_{7}=0, \\
    &\varepsilon_{8}=(k_1-i)(k_2-1)+n+k+1,\textrm{ }\varepsilon_{9}=(k_1-i)(k_2-1)+n, \\
    &\varepsilon_{10,j}=n+|y_j|+1,\textrm{ }0\leq j\leq l,\textrm{ }\varepsilon_{11}=1,\textrm{ }\varepsilon_{12}=1. \nonumber
    \end{split}
    \end{equation}
    In the above, the notation ${{}_i\times_0}$ is an abbreviation for the fiber product ${{}_{\mathrm{ev}_i}\times_{\mathrm{ev}_0}}$, with
    \begin{equation}
    \mathrm{ev}_i:{}_l\overline{\mathcal{R}}_{k+1}^1(x,L,\ring{\bar{a}};P)\xrightarrow{{}_l\mathrm{ev}^{\mathcal{R},P}}P\times L^{k+1}\xrightarrow{\mathrm{id}_P\times\mathrm{pr}_i}P\times L \nonumber
    \end{equation}
    defined as a composition for the moduli space ${}_l\overline{\mathcal{R}}_{k+1}^1(x,L,\ring{\bar{a}};P)$ and similarly for the other moduli spaces. The compatibility of the Kuranishi structures at the boundaries of the admissible K-spaces ${}_l^{j,j+1}\overline{\mathcal{R}}_{k+1}^1(x,L,\ring{\beta};P)$, ${}_{l-1}\overline{\mathcal{R}}_{k+1}^{S^1}(x,L,\ring{\beta};P)$ and ${}_{l-1}\overline{\mathcal{R}}_{k+1,\tau_i}^1(x,L,\ring{\beta};P)$ are similar to that of ${}_l\overline{\mathcal{R}}_{k+1}^1(x,L,\ring{\beta};P)$ described above.
    \item[(iv)] Let $\widehat{S}_r\mathbb{X}$ be the normalized codimension $r$ corner of an admissible K-space $\mathbb{X}$, there are isomorphisms of admissible K-spaces
    \begin{equation}\label{eq:corner2}
    \begin{split}
    &\widehat{S}_r\left({}_l\overline{\mathcal{R}}_{k+1}^1(x,L,\ring{\bar{a}};P)\right)\cong\bigsqcup_{\substack{(T,\ring{B},v_0)\in\mathcal{T}(k+1,\ring{\bar{a}})\\ \#E_{\mathrm{int}}(T)+d+r_1+r_2=r \\ v_0\in V_{\mathrm{int}}(T)}}\bigsqcup_{\substack{j_1+\cdots+j_{r_2}=l_2\\l_1+l_2=l}}\left(\prod_{e\in E_{\mathrm{int}}(T)}\widehat{S}_dP\times L\right) \\
    &\textrm{ }{{}_\Delta\times\mathrm{ev}_\mathrm{int}}\textrm{ }\left(\prod_{v\in V_{\mathrm{int}}(T)\setminus\{v_0\}}\overline{\mathcal{R}}_{k_v+1}\left(L,\ring{B}(v);\widehat{S}_dP\right)\right.\\
    &\left.\times\left(\prod_{i=1}^{r_2}{}_{j_i}\overline{\mathcal{M}}(y_{j_i},x;\widehat{S}_dP)\times\widehat{S}_{r_1}({}_{l_1}\overline{\mathcal{R}}_{k_{v_0}+1}^1)\left(y_{j_i},L,\ring{B}(v_0);\widehat{S}_dP\right)\right)\right), \\
    \end{split}
    \end{equation}
    where $y_{j_i}$ are 1-periodic orbits of $X_{H_t}$ and 
    \[
    \Delta:\prod_{e\in E_\mathrm{int}(T)}\widehat{S}_dP\times L\rightarrow\prod_{e\in E_\mathrm{int}(T)}(\widehat{S}_dP\times L)^2
    \]
    is the diagonal map. The identifications of the codimension $r$ corners of the moduli spaces ${}_l^{j,j+1}\overline{\mathcal{R}}_{k+1}^1(x,L,\ring{\bar{a}};P)$, ${}_{l-1}\overline{\mathcal{R}}_{k+1}^{S^1}(x,L,\ring{\bar{a}};P)$ and ${}_{l-1}\overline{\mathcal{R}}_{k+1,\tau_i}^1(x,L,\ring{\bar{a}};P)$ are similar. 
    
    For every $r,r'\in\mathbb{N}$, the canonical covering map $\pi_{r,r'}:\widehat{S}_{r'}(\widehat{S}_r\mathbb{X})\rightarrow\widehat{S}_{r+r'}(\mathbb{X})$ coincides with the map defined from the fiber product presentation.
    \item[(v)] The Kuranishi structures on the admissible K-spaces (\ref{eq:m2}), (\ref{eq:m3}) and (\ref{eq:m4}) are invariant under the $\mathbb{Z}_{k+1}$-action induced by the cyclic permutations of the boundary marked points $z_0,\cdots,z_k$. The same $\mathbb{Z}_{k+1}$-action cyclically permutes the admissible K-spaces $\left\{\overline{\mathcal{R}}_{k+1,\tau_i}^1(x,L,\ring{\bar{a}};P)\right\}_{i=0}^k$.
\end{itemize}
\end{theorem}

As is observed in \cite{yla}, the construction of the Maurer-Cartan element (the chain $x\in\widehat{\mathbb{H}}_{-2}^{S^1}$ in Theorem \ref{theorem:deform-Viterbo}) in the $S^1$-equivariant case relies on the consideration of a moduli space of holomorphic discs that is a slight variant of $\mathcal{R}_{k+1}(L,\bar{a})$.

For $k\geq3$, define
\begin{equation}
\mathcal{R}_{k+1,\vartheta}:=\left\{(D,z_0,\cdots,z_{k-1},z_k=z_0e^{i\vartheta_k})\right\}/\mathit{PSL}(2,\mathbb{R}), \nonumber
\end{equation}
where $z_0,\cdots,z_k\in\partial D$ are aligned in counterclockwise order and $\vartheta_k\in(0,2\pi)$ is a fixed constant. Note that there is an obvious embedding $\mathcal{R}_{k+1,\vartheta}\subset\mathcal{R}_{k+1}$. For $1\leq i\leq k$, there is a map
\begin{equation}\label{eq:f1}
\pi_{\vartheta,i}:\mathcal{R}_{k+1,\vartheta}\rightarrow\mathcal{R}_k
\end{equation}
defined by applying the cyclic permutation $i$ times to the boundary marked points $z_0,\cdots,z_k$, so that $z_{i+j\textrm{ mod }k+1}$ becomes $z_j$ for $0\leq j\leq k$, and then forgetting the point labeled $z_{k+1-i}$ after the permutation. The map $(-1)^{k,i}\pi_{\vartheta,i}$ is orientation-preserving (cf. \cite{yla}, Appendix B), therefore identifies $\mathcal{R}_{k+1,\vartheta}$ as an open sector of $\mathcal{R}_k$. By \cite{yla}, Lemma 38, we actually have $\bigsqcup_{1\leq i\leq k}(-1)^{ki}\pi_{\vartheta,i}(\mathcal{R}_{k+1,\vartheta})$ covers all but codimension $1$ strata of $\mathcal{R}_k$.

For $\bar{a}\in\pi_2(X,L)$, define
\begin{equation}
\mathcal{R}_{k+1,\vartheta}(L,\bar{a}):=\left\{(u,z_0,\cdots,z_k=z_0e^{i\vartheta_k}),[u]=\bar{a}\right\}/\mathit{PSL}(2,\mathbb{R}), \nonumber
\end{equation}
where $u:(D,\partial D)\rightarrow(X,L)$ is a $J$-holomorphic map in the homotopy class $\bar{a}$ and $z_k=z_0e^{i\vartheta_k}$. Note that when $\bar{a}=0$, we have
\begin{equation}
\mathcal{R}_{2,\vartheta}(L,0)=\mathcal{R}_{3,\vartheta}(L,0)=\emptyset. \nonumber
\end{equation}
The compactification $\overline{\mathcal{R}}_{k+1,\vartheta}(L,\bar{a})$ is an admissible K-space and is modeled on decorated rooted ribbon trees $(T,B)\in\mathcal{T}(k+1,\beta)$. It follows from the embedding (\ref{eq:f1}) that
\begin{equation}
\dim\overline{\mathcal{R}}_{k+1,\vartheta}(L,\bar{a})=\dim\overline{\mathcal{R}}_k(L,\bar{a})=\mu(\bar{a})+n+k-3. \nonumber
\end{equation}
For $P=\{m\}$ or $[m,m+1]$ with $m\in\mathbb{Z}_{\geq0}$, and $a\in H_1(L;\mathbb{Z})$ satisfying $\lambda(a)<(m-k+1)\varepsilon$, there are corner-stratified strongly smooth maps
\begin{equation}
\mathrm{ev}_\vartheta^{\mathcal{R},P}:\overline{\mathcal{R}}_{k+1,\vartheta}(L,\bar{a};P)\rightarrow P\times L^{k+1} \nonumber
\end{equation}
and orientation-preserving isomorphisms of admissible K-spaces
\begin{equation}\label{eq:bd1}
\partial\overline{\mathcal{R}}_{k+1,\vartheta}(L,\bar{a};\{m\})\cong\bigsqcup_{\substack{k_1+k_2=k+1\\1\leq i\leq k_1+1\\ \bar{a}_1+\bar{a}_2=\bar{a}}}(-1)^{\varepsilon_{13}}\overline{\mathcal{R}}_{k_1+1,\vartheta}(L,\bar{a}_1;\{m\})\textrm{ }{{}_i\times_0}\textrm{ }\overline{\mathcal{R}}_{k_2+1}(L,\bar{a}_2;\{m\}),
\end{equation}
\begin{equation}\label{eq:bd4}
	\begin{split}
	&\partial\overline{\mathcal{R}}_{k+1,\vartheta}(L,\bar{a};[m,m+1])\cong-\overline{\mathcal{R}}_{k+1,\vartheta}(L,\bar{a};\{m\})\sqcup\overline{\mathcal{R}}_{k+1,\vartheta}(L,\bar{a};\{m+1\}) \\
	&\sqcup\bigsqcup_{\substack{k_1+k_2=k+1\\1\leq i\leq k_1+1\\ \bar{a}_1+\bar{a}_2=\bar{a}}}(-1)^{\varepsilon_{14}}\overline{\mathcal{R}}_{k_1+1,\vartheta}(L,\bar{a}_1;[m,m+1])\textrm{ }{{}_i\times_0}\textrm{ }\overline{\mathcal{R}}_{k_2+1}(L,\bar{a}_2;[m,m+1]),
	\end{split}
\end{equation}
where
\begin{equation}
\varepsilon_{13}+1=\varepsilon_{14}=(k_1-i)(k_2-1)+n+k_1, \nonumber
\end{equation}
and
\begin{equation}\label{eq:corner1}
	\begin{split}
	&\widehat{S}_r\overline{\mathcal{R}}_{k+2,\vartheta}(L,\beta;P)\cong\bigsqcup_{\substack{(T,B)\in\mathcal{T}(k+2,\beta)\\ \#E_{\mathrm{int}}(T)+d=r \\ v_0\in V_{0,\mathrm{int}}(T)}}\left(\prod_{e\in E_{\mathrm{int}}(T)}\widehat{S}_dP\times L\right)\textrm{ }{{}_\Delta\times\mathrm{ev}_\mathrm{int}}\textrm{ } \\
	&\left(\prod_{v\in V_{0,\mathrm{int}}(T)\setminus\{v_0\}}\overline{\mathcal{R}}_{k_v+1}\left(L,B(v);\widehat{S}_dP\right)\times\overline{\mathcal{R}}_{k_{v_0}+1,\vartheta}\left(L,B(v_0);\widehat{S}_dP\right)\right),
	\end{split}
\end{equation}
where the fiber product on the right-hand side is taken over $\prod_{e\in E_{\mathrm{int}}(T)}\left(\widehat{S}_dP\times L\right)^2$.

\section{The de Rham chain model}\label{section:de Rham}

Let $L$ be a closed, orientable manifold of dimension $n$. We shall recall in this section a de Rham chain model of the $S^1$-equivariant homology of $\mathcal{L}L$ introduced in \cite{yla}.

We start by recalling a convenient model for the space of Moore loops with marked points in $L$, which is due to Wang \cite{ywt}. Denote by $\Pi_1L$ the fundamental groupoid of $L$, which assigns to each ordered pair of points $(p,q)\in L^2$ the collection of equivalence classes of smooth paths from $p$ to $q$. Denote by
\begin{equation}
s:\Pi_1L\rightarrow L\textrm{ and }t:\Pi_1L\rightarrow L \nonumber
\end{equation}
the source and the target maps, which assigns the points $p$ and $q$ to each equivalence class of ordered pair $(p,q)$, respectively. There is an obvious map
\begin{equation}
\left\{(c_0,c_1)\in(\Pi_1L)^2|t(c_0)=s(c_1)\right\}\rightarrow\Pi_1L,\textrm{ }(c_0,c_1)\mapsto c_0\ast c_1 \nonumber
\end{equation}
induced by the concatenation of two paths. For every $k\in\mathbb{Z}_{\geq0}$, denote by $\mathcal{P}_{k+1}L\subset(\Pi_1L)^{k+1}$ the subspace consisting of $(c_0,\cdots,c_k)$ such that $t(c_i)=s(c_{i+1})$ for $0\leq i\leq k-1$. Define
\begin{equation}\label{eq:flsk}
\mathcal{L}_{k+1}L:=\left\{(c_0,\cdots,c_k)\in\mathcal{P}_{k+1}L|t(c_k)=s(c_0)\right\}.
\end{equation}
Note that compared to the space of Moore loops with $k+1$ marked points in $L$ (cf. \cite{kic}, Section 7 and \cite{kia}, Section 4.1), the space $\mathcal{L}_{k+1}L$ defined above has the advantage that it is a smooth oriented manifold of dimension $(k+1)n$. The spaces $\{\mathcal{L}_{k+1}L\}_{k\geq0}$ form finite-dimensional approximations of the free loop space $\mathcal{L}L$.

From now on, we will simply write $\mathcal{L}_{k+1}$ for $\mathcal{L}_{k+1}L$. These spaces are equipped with smooth evaluation maps
\begin{equation}
\mathrm{ev}_j^\mathcal{L}:\mathcal{L}_{k+1}\rightarrow L,\textrm{ }(c_0,\cdots,c_k)\mapsto s(c_j),\textrm{ }0\leq j\leq k \nonumber
\end{equation}
and concatenation maps
\begin{equation}\label{eq:con}
	\begin{split}
	\mathrm{con}_j:&\mathcal{L}_{k+1}\textrm{ }{{}_{\mathrm{ev}_j^\mathcal{L}}\times_{\mathrm{ev}_0^\mathcal{L}}}\textrm{ }\mathcal{L}_{k'+1}\rightarrow\mathcal{L}_{k+k'}, \\
	&\left((c_0,\cdots,c_k),(c_0',\cdots,c_k')\right)\mapsto \\
	&\left\{\begin{array}{ll}
	(c_0,\cdots,c_{j-2},c_{j-1}\ast c_0',c_1',\cdots,c_{k'-1}',c_{k'}'\ast c_j,\cdots,c_k) & \textrm{if }k'\geq1, \\
	(c_0,\cdots,c_{j-2},c_{j-1}\ast c_0'\ast c_j,c_{j+1},\cdots,c_k) & \textrm{if }k'=0.
	\end{array}\right.
	\end{split}
\end{equation}
It is clear from the definition that the concatenation maps are compatible with the decomposition $\mathcal{L}_{k+1}=\bigsqcup_{a\in H_1(L;\mathbb{Z})}\mathcal{L}_{k+1}(a)$ of $\mathcal{L}_{k+1}$ into different homotopy classes, where $\mathcal{L}_{k+1}(a)\subset\mathcal{L}_{k+1}$ is the subspace consisting of loops $\gamma=c_0\ast\cdots\ast c_k\in\Pi_1L$ with $[\gamma]=a$. Thus we obtain a map
\begin{equation}\label{eq:conca}
\mathrm{con}_j:\mathcal{L}_{k+1}(a)\textrm{ }{{}_{\mathrm{ev}_j^\mathcal{L}}\times_{\mathrm{ev}_0^\mathcal{L}}}\textrm{ }\mathcal{L}_{k'+1}(a')\rightarrow\mathcal{L}_{k+k'}(a+a')
\end{equation}
for $a,a'\in H_1(L;\mathbb{Z})$. 

Let $V$ be a smooth manifold and consider the map $\varphi:V\rightarrow\mathcal{L}_{k+1}$. We say that $\varphi$ is a smooth map if the map $\varphi$ is $C^\infty$ and the composition $\mathrm{ev}_0^\mathcal{L}\circ\varphi:V\rightarrow L$ is a submersion. For $m\in\mathbb{N}$, let $\mathfrak{V}_m$ be the collection of oriented submanifolds in $\mathbb{R}^m$, and define $\mathfrak{V}:=\bigsqcup_{m\geq1}\mathfrak{V}_m$. Let $\mathcal{P}(\mathcal{L}_{k+1}(a))$ denote the set of plots, i.e. pairs $(V,\varphi)$, where $V\in\mathfrak{V}$ and $\varphi:V\rightarrow\mathcal{L}_{k+1}(a)$ is a smooth map.

For each $N$, consider the vector space
\begin{equation}\label{eq:sp}
\bigoplus_{(V,\varphi)\in\mathcal{P}(\mathcal{L}_{k+1}(a))}\Omega_c^{\dim(V)-N}(V),
\end{equation}
where $\Omega_c^\ast(V)$ is the space of compactly supported differential forms on $V$. Denote by $Z_N$ the subspace of (\ref{eq:sp}) defined by
\begin{equation}
\begin{split}
&\left\{(V,\varphi,\pi_!\omega)-(V',\varphi\circ\pi,\omega)\left\vert(V,\varphi)\in\mathcal{P}(\mathcal{L}_{k+1}(a)),V'\in\mathfrak{V}, \right. \right. \\
&\left.\omega\in \Omega_c^{\dim(V')-N}(V'),\pi:V'\rightarrow V\textrm{ is a submersion}\right\}. \nonumber
\end{split}
\end{equation}
As a graded vector space, the $N$-th degree \textit{de Rham chain complex} of $\mathcal{L}_{k+1}(a)$ is the quotient
\begin{equation}
C_N^\mathrm{dR}(\mathcal{L}_{k+1}(a)):=\left(\bigoplus_{(V,\varphi)\in\mathcal{P}(\mathcal{L}_{k+1}(a))}\Omega_c^{\dim(V)-N}(V)\right)/Z_N. \nonumber
\end{equation}
The differential $\partial:C_\ast^\mathrm{dR}(\mathcal{L}_{k+1}(a))\rightarrow C_{\ast-1}^\mathrm{dR}(\mathcal{L}_{k+1}(a))$ is defined as
\begin{equation}
\partial(V,\varphi,\omega):=(-1)^{|\omega|+1}(V,\varphi,d\omega). \nonumber
\end{equation}
It is straightforward to check that $\partial$ is well-defined and $\partial^2=0$. The homology of $\left(C_\ast^\mathrm{dR}(\mathcal{L}_{k+1}(a)),\partial\right)$ is denoted by $H_\ast^\mathrm{dR}\left(\mathcal{L}_{k+1}(a)\right)$. Unlike in the case when $\mathcal{L}_{k+1}$ is the space of Moore loops considered in \cite{kia,kic}, $H_\ast^\mathrm{dR}(\mathcal{L}_{k+1}(a))$ is no longer isomorphic to $H_\ast(\mathcal{L}(a)L;\mathbb{R})$, but the usage of Wang's finite-dimensional model of $\mathcal{L}_{k+1}$  is justified by the following.

If one forms the total complex
\begin{equation}\label{eq:total}
\left(C_\ast^\mathrm{dR}(a):=\prod_{k=0}^\infty C_\ast^\mathrm{dR}\left(\mathcal{L}_{k+1}(a)\right),\partial^\mathrm{tot}\right),
\end{equation}
with the differential $\partial^\mathrm{tot}$ defined by
\begin{equation}\label{eq:totd}
\partial^\mathrm{tot}(x)(a,k)=\left\{\begin{array}{ll}
\partial x(a,0) & \textrm{if }k=0, \\
\partial x(a,k)+(-1)^{|x|}\sum_{i=0}^k(-1)^i\delta_{k,i}(x(a,k-1)) & \textrm{if }k\geq1,
\end{array}\right.
\end{equation}
where the $\delta_{k,i}$'s are the cosimplicial structure maps defined by (\ref{eq:cosimp1}) below. Then there is an isomorphism
\begin{equation}\label{eq:sing}
H_\ast\left(C_\ast^\mathrm{dR}(a),\partial^\mathrm{tot}\right)\cong H_\ast^\mathrm{sing}(\mathcal{L}(a)L;\mathbb{R}),
\end{equation}
where the singular homology on the right-hand side is defined using the $C^\infty$-topology on $\mathcal{L}(a)L$. See \cite{kic,ywt}.

For $k\in\mathbb{N}$, $k'\in\mathbb{Z}_{\geq0}$, and $1\leq j\leq k$, define the fiber product
\begin{equation}\label{eq:fp}
\circ_j:C_{n+d}^\mathrm{dR}(\mathcal{L}_{k+1}(a))\otimes C_{n+d'}^\mathrm{dR}\left(\mathcal{L}_{k'+1}(a')\right)\rightarrow C_{n+d+d'}^\mathrm{dR}\left(\mathcal{L}_{k+k'}(a+a')\right)
\end{equation}
by
\begin{equation}
x\circ_jy:=(-1)^{(\dim(V)-|\omega|-n)|\omega'|}\left(V{{}_{\varphi_j}\times_{\varphi_0'}}V',\mathrm{con}_j\circ(\varphi_j\times\varphi_0'),\omega\times\omega'|_{V{{}_{\varphi_j}\times_{\varphi_0'}}V'}\right), \nonumber
\end{equation}
where $\varphi_j=\mathrm{ev}_j^\mathcal{L}\circ\varphi$ and $\varphi_0'=\mathrm{ev}_0^\mathcal{L}\circ\varphi'$. One can check that this is a chain map, and after passing to homology of the total complex (\ref{eq:total}), it corresponds to the Chas-Sullivan loop product under the isomorphism (\ref{eq:sing}).

We also need the relative version of the de Rham complex $C_\ast^\mathrm{dR}(\mathcal{L}_{k+1}(a))$, whose construction uses de Rham chains on $[-1,1]\times\mathcal{L}_{k+1}(a)$ relative to $\{-1,1\}\times\mathcal{L}_{k+1}(a)$. This is defined by replacing the set of plots $\mathcal{P}(\mathcal{L}_{k+1}(a))$ above with $\overline{\mathcal{P}}(\mathcal{L}_{k+1}(a))$, which consists of tuples $(V,\varphi,\tau_+,\tau_-)$, where
\begin{itemize}
	\item $V\in\mathfrak{V}$ and $\varphi:V\rightarrow\mathbb{R}\times\mathcal{L}_{k+1}(a)$. Write $\varphi=(\varphi_\mathbb{R},\varphi_\mathcal{L})$, define $V_I:=(\varphi_\mathbb{R})^{-1}(I)$ for every interval $I\subset\mathbb{R}$.
	\item $\varphi_\mathbb{R}$ and $\varphi_\mathcal{L}$ are $C^\infty$, and the map $V\rightarrow\mathbb{R}\times L$ defined by $v\mapsto\left(\varphi_\mathbb{R}(v),\mathrm{ev}_0\circ\varphi_\mathcal{L}(v)\right)$ is a submersion.
	\item $\tau_+:V_{\geq1}\rightarrow\mathbb{R}_{\geq1}\times V_1$ is a diffeomorphism such that
	\begin{equation}
	\varphi|_{V_{\geq1}}=(i_{\geq1}\times\varphi_\mathcal{L}|_{V_1})\circ\tau_+, \nonumber
	\end{equation}
	where $i_{\geq1}:\mathbb{R}_{\geq1}\hookrightarrow\mathbb{R}$ is the obvious inclusion.
	\item $\tau_-:V_{\leq-1}\rightarrow\mathbb{R}_{\leq-1}\times V_{-1}$ is a diffeomorphism such that
	\begin{equation}
	\varphi|_{V_{\leq-1}}=(i_{\leq-1}\times\varphi_\mathcal{L}|_{V_{-1}})\circ\tau_{-1}, \nonumber
	\end{equation}
	where $i_{\leq-1}:\mathbb{R}_{\leq-1}\hookrightarrow\mathbb{R}$ is the obvious inclusion.
\end{itemize}
Note that $V_{\geq1}$ and $V_{\leq-1}$ can be empty.

For any $(V,\varphi,\tau_+,\tau_-)\in\overline{\mathcal{P}}(\mathcal{L}_{k+1}(a))$ and $N\in\mathbb{Z}$, let $\Omega^N(V,\varphi,\tau_+,\tau_-)$ be the vector space of differential $N$-forms $\omega\in\Omega^N(V)$ such that
\begin{itemize}
	\item $\omega|_{V_{[-1,1]}}$ is compactly supported,
	\item $\omega|_{V_{\geq1}}=(\tau_+)^\ast(1\times\omega|_{V_1})$,
	\item $\omega|_{V_{\leq-1}}=(\tau_-)^\ast(1\times\omega|_{V_{-1}})$.
\end{itemize}
Define the space of $N$-th degree \textit{relative de Rham chains} to be
\begin{equation}
\overline{C}_N^\mathrm{dR}(\mathcal{L}_{k+1}(a)):=\left(\bigoplus_{(V,\varphi,\tau_+,\tau_-)\in\overline{\mathcal{P}}(\mathcal{L}_{k+1}(a))}\Omega^{\dim(V)-N-1}(V,\varphi,\tau_+,\tau_-)\right)/\overline{Z}_N, \nonumber
\end{equation}
where $\overline{Z}_N\subset\overline{C}_N^\mathrm{dR}(\mathcal{L}_{k+1}(a))$ is generated by
\begin{equation}
(V,\varphi,\tau_+,\tau_-,\omega)-(V',\varphi',\tau_+',\tau_-',\omega'), \nonumber
\end{equation}
if there exists a submersion $\pi:V'\rightarrow V$ satisfying $\varphi'=\varphi\circ\pi$, $\omega=\pi_!\omega'$, and
\begin{equation}
\tau_+\circ\pi|_{V_{\geq1}'}=(\mathrm{id}_{\mathbb{R}_{\geq1}}\times\pi|_{V_1'})\circ\tau_+', \nonumber
\end{equation}
\begin{equation}
\tau_-\circ\pi|_{V_{\leq-1}'}=(\mathrm{id}_{\mathbb{R}_{\leq-1}}\times\pi|_{V_{-1}'})\circ\tau_-', \nonumber
\end{equation}
where $\mathrm{id}_{\mathbb{R}_I}$ denotes the identity map on $\mathbb{R}_I$. The differential
\begin{equation}
\overline{\partial}:\overline{C}_\ast^\mathrm{dR}(\mathcal{L}_{k+1}(a))\rightarrow\overline{C}_{\ast-1}^\mathrm{dR}(\mathcal{L}_{k+1}(a)) \nonumber
\end{equation}
and the fiber product
\begin{equation}\label{eq:fpr}
\circ_j:\overline{C}_{n+d}^\mathrm{dR}(\mathcal{L}_{k+1}(a))\otimes\overline{C}^\mathrm{dR}_{n+d'}(\mathcal{L}_{k'+1}(a'))\rightarrow\overline{C}^\mathrm{dR}_{n+d+d'}(\mathcal{L}_{k+k'}(a+a'))
\end{equation}
are defined similarly as in the case of ordinary de Rham chains. We omit the details and refer the reader to \cite{yla}, Section 3.1 for details.

The de Rham chain complexes $C_\ast^\mathrm{dR}(\mathcal{L}_{k+1}(a))$ and $\overline{C}_\ast^\mathrm{dR}(\mathcal{L}_{k+1}(a))$ are closely related. Consider the projections $e_\pm:\overline{C}_\ast^\mathrm{dR}(\mathcal{L}_{k+1}(a))\rightarrow C_\ast^\mathrm{dR}(\mathcal{L}_{k+1}(a))$ given by
\begin{equation}\label{eq:e+}
e_+(V,\varphi,\tau_+,\tau_-,\omega):=(-1)^{\dim(V)-1}(V_1,\varphi|_{U_1},\omega|_{U_1})
\end{equation}
and
\begin{equation}\label{eq:e-}
e_-(V,\varphi,\tau_+,\tau_-,\omega):=(-1)^{\dim(V)-1}(V_{-1},\varphi|_{V_{-1}},\omega|_{V_{-1}}),
\end{equation}
where $V_1$ and $V_{-1}$ are oriented so that $\tau_+$ and $\tau_-$ are orientation preserving, respectively. One can show that $e_\pm$ are quasi-isomorphisms, with the same quasi-inverse given by the obvious inclusion
\begin{equation}\label{eq:i}
i(V,\varphi,\omega):=(-1)^{\dim(V)}(\mathbb{R}\times V,\mathrm{id}_\mathbb{R}\times\varphi,\tau_+,\tau_-,1\times\omega),
\end{equation}
so we have
\begin{equation}
H_\ast\left(\overline{C}_\ast^\mathrm{dR}(a)\right)\cong H_\ast\left(C_\ast^\mathrm{dR}(a)\right)\cong H_\ast(\mathcal{L}(a)L;\mathbb{R}), \nonumber
\end{equation}
where $\overline{C}_\ast^\mathrm{dR}(a):=\prod_{k=0}^\infty\overline{C}_\ast^\mathrm{dR}\left(\mathcal{L}_{k+1}(a)\right)$ is the relative total de Rham complex.

For the purpose of establishing the equation (\ref{eq:def1}) in Theorem \ref{theorem:deform-Viterbo}, we need a chain model of the string homology that carries the structures of a strict $S^1$-complex and an odd dg Lie algebra. Write
\begin{equation}
C_\ast(a,k):=C^\mathrm{dR}_{\ast+n+\mu(a)+k-1}(\mathcal{L}_{k+1}(a)), \nonumber
\end{equation}
\begin{equation}
C_\ast(k):=C_{\ast+n}^\mathrm{dR}(\mathcal{L}_{k+1})=\bigoplus_{a\in H_1(L;\mathbb{Z})}C_{\ast+n+\mu(a)+k-1}^\mathrm{dR}(\mathcal{L}_{k+1}(a)). \nonumber
\end{equation}
The chain complexes $\left\{C_\ast(k)\right\}_{k\geq0}$ form a non-symmetric dg operad $\mathcal{O}_L$, with multiplication
\begin{equation}
\mu_L:=(L',i_2\circ\phi,1)\in C_{-1}(0,2), \nonumber
\end{equation}
and unit
\begin{equation} 
e_L:=(L',i_0\circ\phi,1)\in C_1(0,0), \nonumber
\end{equation}
where the map $i_0:L\rightarrow\mathcal{L}_1(0)$ is given by $p\mapsto[p]=(p,p)\in\Pi_1L$, $i_2:L\rightarrow\mathcal{L}_3(0)$ is defined by taking three copies of $i_0$ and $\phi:L'\rightarrow L$ is an orientation-preserving diffeomorphism. We have
\begin{equation}
\mu_L\circ_1e_L=\mu_L\circ_2e_L=\mathrm{id}_{\mathcal{O}_L}, \nonumber
\end{equation}
where
\begin{equation}
\mathrm{id}_{\mathcal{O}_L}:=(L',i_1\circ\phi,1)\in C_0(0,1) \nonumber
\end{equation}
is the identity, with $i_1:L\rightarrow\mathcal{L}_2(0)$ given by taking two copies of $i_0$.

From now on, we assume that $L\subset X$ is a closed Lagrangian submanifold in the Liouville domain $X$. Consider the total complex
\begin{equation}\label{eq:tot}
\left(C_\ast:=\bigoplus_{a\in H_1(L;\mathbb{Z})}\prod_{k=0}^\infty C_\ast(a,k),\partial^\mathrm{tot}\right),
\end{equation}
where the differential $\partial^\mathrm{tot}$ is defined as in (\ref{eq:totd}).

$C_\ast$ has the structure of an associative dg algebra, with the product $\bullet:C_i\otimes C_j\rightarrow C_{i+j-1}$ given by
\begin{equation}\label{eq:bullet}
(x\bullet y)(a,k):=\sum_{\substack{k_1+k_2=k\\a_1+a_2=a}}(-1)^{k_1(|y|+1)}(\mu_L\circ_1 x(a_1,k_1))\circ_{k_1+1}y(a_2,k_2).
\end{equation}
Using the Liouville form $\lambda$, one can define an action filtration
\begin{equation}\label{eq:Xi}
F^\Xi C_\ast:=\bigoplus_{\lambda(a)>\Xi}\prod_{k=0}^\infty C_\ast(a,k).
\end{equation}
We denote by $\widehat{C}_\ast$ the corresponding completion of $C_\ast$.

The dg operad $\mathcal{O}_L$ also carries a cyclic structure
\begin{equation}\label{eq:cyc}
\tau_k:C_\ast(k)\rightarrow C_\ast(k),
\end{equation}
which is induced by the cyclic permutation
\begin{equation}
\mathcal{L}_{k+1}\rightarrow\mathcal{L}_{k+1},\textrm{ }(c_0,\cdots,c_k)\mapsto(c_1,\cdots,c_k,c_0). \nonumber
\end{equation}
Note that $\tau_k^{k+1}=\mathrm{id}_C$, where $\mathrm{id}_C$ is the identity map on $C_\ast(k)$. This gives rise to a chain level BV operator
\begin{equation}\label{eq:BV-c}
\begin{split}
\delta_\mathrm{cyc}&:C_\ast(a,k+1)\rightarrow C_{\ast+1}(a,k), \\
(\delta_\mathrm{cyc} x)(a,k)&:=\sum_{j=1}^{k+1}(-1)^{|x|+k(j-1)}\tau_{k+1}^j\left(x(a,k+1)\right)\circ_{k+2-j}e_L.
\end{split}
\end{equation}
After passing to homology and composing with the aforementioned isomorphism $H_\ast(C_\ast)\cong H_\ast(\mathcal{L}L;\mathbb{R})$ (cf. (\ref{eq:sing})), it coincides with the BV operator $\Delta:H_\ast(\mathcal{L}L;\mathbb{R})\rightarrow H_{\ast+1}(\mathcal{L}L;\mathbb{R})$ defined by loop rotations. As before, $\delta_\mathrm{cyc}$ is compatible with the decomposition (\ref{eq:tot}) of the total complex $C_\ast$, therefore it extends to an operator on the completion $\widehat{C}_\ast$. 

$\mathcal{O}_L$ has the structure of a cosimplicial chain complex, with the operations
\begin{equation}
\delta_{k,i}:C_\ast(k-1)\rightarrow C_\ast(k),\textrm{ }\sigma_{k,i}:C_\ast(k+1)\rightarrow C_\ast(k)
\end{equation}
for $0\leq i\leq k$ defined by
\begin{equation}\label{eq:cosimp}
\delta_{k,i}(x):=\left\{\begin{array}{ll}
\mu_L\circ_2x & i=0, \\ x\circ_i\mu_L & 1\leq i\leq k-1, \\ \mu_L\circ_1x & i=k,
\end{array}\right.
\end{equation}
\begin{equation}
\sigma_{k,i}(x):=x\circ_{i+1} e_L. \nonumber
\end{equation}
More concretely, $\delta_{k,i}$ is induced by the map
\begin{equation}\label{eq:cosimp1}
\begin{split}
\mathcal{L}_k&\rightarrow\mathcal{L}_{k+1} \\
(c_0,\cdots,c_{k-1})&\mapsto\left\{\begin{array}{ll}
\left(c_0,\cdots,c_{i-1},[s(c_i)],c_i,\cdots,c_{k-1}\right) & \textrm{if }0\leq i\leq k-1, \\
\left(c_0,\cdots,c_{k-1},[t(c_{k-1})]\right) & \textrm{if }i=k,
\end{array}\right. \nonumber
\end{split}
\end{equation}
and $\sigma_{k,i}$ is induced by the map
\begin{equation}
\begin{split}
\mathcal{L}_{k+1}&\rightarrow\mathcal{L}_k \\
(c_0,\cdots,c_k)&\mapsto(c_0,\cdots,c_{i-1},c_i\ast c_{i+1},c_{i+2},\cdots,c_k). \nonumber
\end{split}
\end{equation}
Recall the cosimplicial identities
\begin{equation}\label{eq:cod}
\delta_{k+1,j}\circ\delta_{k,i}=\delta_{k+1,i}\circ\delta_{k,j-1}\textrm{ for }0\leq i<j\leq k+1,
\end{equation}
\begin{equation}\label{eq:cos}
\sigma_{k-1,j}\circ\sigma_{k,i}=\sigma_{k-1,i}\circ\sigma_{k,j+1}\textrm{ for }0\leq i\leq j\leq k-1,
\end{equation}
\begin{equation}\label{eq:sd}
\sigma_{k,j}\circ\delta_{k+1,i}=\left\{\begin{array}{ll}
\delta_{k,i}\circ\sigma_{k-1,j-1} & 0\leq i<j\leq k, \\ \mathrm{id}_C & i=j,j+1, \\ \delta_{k,i-1}\circ\sigma_{k-1,j} & k+1\geq i>j+1\geq1,
\end{array}\right.
\end{equation}
and the cocyclic identities
\begin{equation}\label{eq:td}
\tau_k\circ\delta_{k,i}=\left\{\begin{array}{ll}
\delta_{k,k} & i=0, \\ \delta_{k,i-1}\circ\tau_{k-1} & 1\leq i\leq k,
\end{array}\right.
\end{equation}
\begin{equation}\label{eq:cocy2}
\tau_k\circ\sigma_{k,i}=\left\{\begin{array}{ll}
\sigma_{k,k}\circ\tau_{k+1}^2 & i=0, \\ \sigma_{k,i-1}\circ\tau_{k+1} & 1\leq i\leq k.
\end{array}\right.
\end{equation}

A chain $x\in C_\ast(k)$ is \textit{normalized} if $\sigma_{k-1,i}(x)=0$ for all $0\leq i\leq k-1$. The normalized chains form a subcomplex $C_\ast^\mathrm{nm}\subset C_\ast$. It follows from \cite{kia}, Lemma 2.5 that the natural inclusion $C_\ast^\mathrm{nm}\hookrightarrow C_\ast$ is a quasi-isomorphism. On the other hand, a chain $x\in C_\ast(k)$ is \textit{degenerate} if there exist $1\leq i\leq k$ and $y\in C_\ast(k-1)$ such that $x=\delta_{k,i}(y)$. Degenerate chains also form a subcomplex $D_\ast\subset C_\ast$, and the quotient $C_\ast^\mathrm{nd}:=C_\ast/D_\ast$ is called the complex of non-degenerate de Rham chains. It follows from the Dold-Kan correspondence that the composition
\begin{equation}
C_\ast^\mathrm{nm}\hookrightarrow C_\ast\twoheadrightarrow C_\ast^\mathrm{nd} \nonumber
\end{equation}
is a quasi-isomorphism, see \cite{yla}, Lemma 18. It follows that there is a well-defined BV operator on $C_\ast^\mathrm{nd}$, which we still denote by $\delta_\mathrm{cyc}$ by abuse of notations, such that $(C_\ast^\mathrm{nd},\partial^\mathrm{tot},\delta_\mathrm{cyc})$ is a strict $S^1$-complex (cf. \cite{sg1}, Definition 2.1). In particular, $\delta_\mathrm{cyc}^2=0$. See \cite{yla}, Lemma 21. 

Associated to the strict $S^1$-complex $(C_\ast^\mathrm{nd},\partial^\mathrm{tot},\delta_\mathrm{cyc})$ is an $S^1$-equivariant chain complex
\begin{equation}\label{eq:cpx-S1}
C_\ast^{S^1}:=\left(C_\ast^\mathrm{nd}\otimes_\mathbb{R}\mathbb{R}(\!(h)\!)/h\mathbb{R}[\![h]\!],\partial^{S^1}:=\partial^\mathrm{tot}+h\delta_\mathrm{cyc}\right),
\end{equation}
where the formal variable $h$ has degree $-2$. There is an obvious decomposition $C_\ast^{S^1}=\bigoplus_{a\in H_1(L;\mathbb{Z})}\prod_{k=0}^\infty C_\ast^{S^1}(a,k)$.

\begin{proposition}
There is an isomorphism
\begin{equation}\label{eq:eq-sing}
H_\ast(C_\ast^{S^1})\cong H_\ast^{S^1}(\mathcal{L}L;\mathbb{R}).
\end{equation}
\end{proposition}
\begin{proof}
This is a consequence of \cite{ywa}, Proposition 4.8, combined with \cite{ywt}, Theorem 2.2.1.
\end{proof}

The action filtration (\ref{eq:Xi}) induces a similar action filtration on $C_\ast^{S^1}$, and we denote the corresponding completion by $\widehat{C}_\ast^{S^1}$. The homology $H_\ast(\widehat{C}_\ast^{S^1})$ gives the vector space $\widehat{\mathbb{H}}_\ast^{S^1}$ defined by (\ref{eq:H-c}).

We now define the chain level string bracket, which is expected to be a refinement of the string bracket introduced by Chas-Sullivan \cite{css}. For two $S^1$-equivariant chains $\tilde{x}=\sum_{d=0}^\infty x_d\otimes h^{-d}\in C_\ast^{S^1}$ and $\tilde{y}=\sum_{d=0}^\infty y_d\otimes h^{-d}\in C_\ast^{S^1}$, define
\begin{equation}\label{eq:bracket-pr}
	\begin{split}
	\left\{\tilde{x},\tilde{y}\right\}(a,k)&:=\sum_{\substack{a_1+a_2=a\\k_1+k_2=k+1}}\sum_{d=0}^\infty\sum_{d_1+d_2=d}\sum_{i=1}^{k_1}\sum_{j=1}^{k_2+1}(-1)^{\maltese_{ij}^{d_2}}x_{d_1}(a_1,k_1) \\
	&\circ_i\left(\tau_{k_2+1}^j(y_{d_2}(a_2,k_2+1))\circ_{k_2+2-j}e_L\right)\otimes h^{-d} \\
	&-\sum_{\substack{a_1+a_2=a\\k_1+k_2=k+1}}\sum_{d=0}^\infty\sum_{d_1+d_2=d}\sum_{i=1}^{k_1}\sum_{j=1}^{k_1+1}(-1)^{\maltese_{ij}^{d_1}+(|x_{d_2}|+1)(|y_{d_1}|+1)} \\
	&\left(\tau_{k_1+1}^j(y_{d_1}(a_1,k_1+1))\circ_{k_1+2-j}e_L\right)\circ_ix_{d_2}(a_2,k_2)\otimes h^{-d},
	\end{split}
\end{equation}
where
\begin{equation}
\maltese_{ij}^d=(i-1)(k_2-1)+(k_1-1)(|y_d|+k_2)+|y_d|+k_2(j-1). \nonumber
\end{equation}
This gives a bilinear operation on $C_\ast^{S^1}$, which is in general \textit{not} a Lie bracket. However, $\{\cdot,\cdot\}$ becomes an odd Lie bracket after passing to \textit{Connes' complex}, which is the quotient
\begin{equation}\label{eq:Connes}
C_\ast^\lambda:=C_\ast^\mathrm{nd}/\mathrm{im}(1-t),
\end{equation}
where $t:C_\ast^\mathrm{nd}\rightarrow C_\ast^\mathrm{nd}$ is the cyclic operator defined by $t_k(x):=(-1)^k\tau_k(x)$ for $x\in C_\ast(k)$. Using the cocyclic identities (\ref{eq:td}) and (\ref{eq:cocy2}), one can check that the differential $\partial^\mathrm{tot}$ on $C_\ast^\mathrm{nd}$ descends to one on $C_\ast^\lambda$, which we will still denote by $\partial^\mathrm{tot}$. It follows from \cite{yla}, Lemma 23 that the natural projection
\begin{equation}\label{eq:proj}
\left(C_\ast^\mathrm{nd}\otimes_\mathbb{R}\mathbb{R}(\!(h)\!)/u\mathbb{R}[\![h]\!],\partial^\mathrm{tot}+h\delta_\mathrm{cyc}\right)\rightarrow(C_\ast^\lambda,\partial^\mathrm{tot})
\end{equation}
is a quasi-isomorphism. For a cochain $\tilde{x}\in C_\ast^{S^1}$, denote its image under the projection (\ref{eq:proj}) by $\underline{x}\in C_\ast^\lambda$. Then the string bracket is defined as
\begin{equation}
\begin{split}
\{\cdot,\cdot\}:C_i^\lambda\otimes C_j^\lambda&\rightarrow C_{i+j+1}^\lambda, \\
\underline{x}\otimes\underline{y}&\mapsto\underline{\left\{\tilde{x},\tilde{y}\right\}}. \nonumber
\end{split}
\end{equation}
This is independent of the choices of the lifts $\tilde{x}$ and $\tilde{y}$ of $\underline{x}$ and $\underline{y}$.

\begin{proposition}
$\left(C_\ast^\lambda,\partial^\mathrm{tot},\{\cdot,\cdot\}\right)$ is a dg Lie algebra of degree $1$. In particular, for $\underline{x},\underline{y},\underline{z}\in C_\ast^\lambda$, we have the Jacobi identity
\begin{equation}\label{eq:Jacobi}
\left\{\underline{x},\{\underline{y},\underline{z}\}\right\}=\left\{\{\underline{x},\underline{y}\},\underline{z}\right\}+(-1)^{(|\underline{x}|+1)(|\underline{y}|+1)}\left\{\underline{y},\{\underline{x},\underline{z}\}\right\},
\end{equation}
and the odd Lie bracket is graded anti-symmetric in the sense that
\begin{equation}\label{eq:anti-sym}
\{\underline{x},\underline{y}\}=(-1)^{(|\underline{x}|+1)(|\underline{y}|+1)+1}\{\underline{y},\underline{x}\}.
\end{equation}
\end{proposition}
\begin{proof}
Our definition of $\{\cdot,\cdot\}$ in (\ref{eq:bracket-pr}) above is a slight modification of \cite{yla}, (3.70). It is an easy observation that after passing to the quotient complex $C_\ast^\lambda$, it coincides with the definition of the chain level string bracket in \cite{yla}. Thus the proposition follows from \cite{yla}, Lemma 24.
\end{proof}

Since the odd Lie bracket $\{\cdot,\cdot\}$ is compatible with the decomposition of $C_\ast^\lambda$ according to $(a,k)\in H_1(L;\mathbb{Z})\times\mathbb{Z}_{\geq0}$, it naturally extends to the completion $\widehat{C}_\ast^\lambda$ of $C_\ast^\lambda$ with respect to the action filtration induced from the one on $C_\ast^\mathrm{nd}$ defined by (\ref{eq:Xi}), and equips it with the structure of a dg Lie algebra $\left(\widehat{C}_\ast^\lambda,\partial^\mathrm{tot},\{\cdot,\cdot\}\right)$ of degree $1$.

There is a parallel story in the relative case, where one can construct a dg Lie algebra of degree $1$ on the Connes' complex $\overline{C}_\ast^\lambda=\overline{C}_\ast^\mathrm{nd}/\mathrm{im}(1-\bar{t})$ and its completion. The complexes $C_\ast^\lambda$ and $\overline{C}_\ast^\lambda$ are related by the chain maps
\begin{equation}\label{eq:ie}
\underline{i}:C_\ast^\lambda\rightarrow\overline{C}_\ast^\lambda,\textrm{ }\underline{e}_\pm:\overline{C}_\ast^\lambda\rightarrow C_\ast^\lambda
\end{equation}
induced by (\ref{eq:i}), (\ref{eq:e+}) and (\ref{eq:e-}), respectively, which satisfy $\underline{e}_+\circ\underline{i}=\underline{e}_-\circ\underline{i}=\mathrm{id}_C$, where $\mathrm{id}_C$ is the identity map on $C_\ast^\lambda$, and are compatible with the odd Lie brackets. The chain maps (\ref{eq:ie}) on Connes' complexes can also be lifted to maps
\begin{equation}
\tilde{i}:C_\ast^{S^1}\rightarrow\overline{C}_\ast^{S^1},\textrm{ }\tilde{e}_\pm:\overline{C}_\ast^{S^1}\rightarrow C_\ast^{S^1} \nonumber
\end{equation}
on the corresponding $S^1$-equivariant complexes such that $\tilde{e}_+\circ\tilde{i}=\tilde{e}_-\circ\tilde{i}=\mathrm{id}_C$. Here, $\mathrm{id}_C$ denotes the identity map on $C_\ast^{S^1}$. We refer to \cite{yla}, Section 3.2 for details.

\section{Proof of Theorem \ref{theorem:deform-Viterbo}}\label{section:proof}

This section is devoted to the proof of Theorem \ref{theorem:deform-Viterbo}. Using the chain model introduced in Section \ref{section:de Rham}, we first write down its chain level statement (Theorem \ref{theorem:chain}) in Section \ref{section:chain-s}. Then, we define the required chains by pushing forward the virtual fundamental chains of the moduli spaces introduced in Section \ref{section:CG}. This is done in Section \ref{section:chain-d}. Finally, we analyze the boundary strata of these moduli spaces to prove Theorem \ref{theorem:chain} in Section \ref{section:completion}.

\subsection{Chain level statements}\label{section:chain-s}

Let $L\subset\mathrm{int}(X)$ be a closed Lagrangian submanifold in the interior of a Liouville domain $X$ with $c_1(X)=0$ that is oriented and $\mathit{Spin}$. Recall that given this data, we have constructed a dg Lie algebra $\left(\widehat{C}_\ast^\lambda,\partial^\mathrm{tot},\{\cdot,\cdot\}\right)$ of degree $1$ in Section \ref{section:de Rham}, where $C_\ast^\lambda$ is a quotient complex of $C_\ast^{S^1}$, and the latter complex carries a chain map
\begin{equation}\label{eq:B}
\begin{split}
B:C_\ast^{S^1}(a,k+1)&\rightarrow C_{\ast+1}^\mathrm{nd}(a,k),\\
B(\tilde{x})(a,k)&:=\sum_{j=1}^{k+1}(-1)^{|\tilde{x}|+k(j-1)}\tau_{k+1}^j(x_0(a,k+1))\circ_{k+2-j}e_L,
\end{split}
\end{equation}
which induces the marking map $H_\ast^{S^1}(\mathcal{L}L;\mathbb{R})\rightarrow H_{\ast+1}(\mathcal{L}L;\mathbb{R})$ in string topology \cite{css} after passing to the homology (of the total complex) and composing with the isomorphisms (\ref{eq:sing}) and (\ref{eq:eq-sing}).

The aim of this subsection is to show that Theorem \ref{theorem:deform-Viterbo} is a consequence of the following theorem.

\begin{theorem}\label{theorem:chain}
Under the assumption that $C_d^\mathrm{GH}(X)<\infty$, there exist chains $\underline{x}\in\widehat{C}_{-2}^\lambda$, $\underline{y}\in\widehat{C}_{2d}^\lambda$, $\underline{z}\in\widehat{C}_{2d-1}^\lambda$ and a real number $\varepsilon\in\mathbb{R}_{>0}$ such that the following are satisfied.
\begin{itemize}
	\item[(i)] $\partial^\mathrm{tot}(\underline{x})-\frac{1}{2}\left\{\underline{x},\underline{x}\right\}=0$.
	\item[(ii)] $\partial^\mathrm{tot}(\underline{y})-\left\{\underline{x},\underline{y}\right\}=\underline{z}$.
	\item[(iii)] $\underline{x}(a,k)\neq0$ only if $\lambda(a)\geq2\varepsilon$ or $a=0$, $k\geq3$. Moreover, the chain $\underline{x}(0,3)$ admits a lift $\tilde{x}(0,3)\in C_{-2}^{S^1}(0,3)$, such that under the isomorphism (\ref{eq:sing}), the homology class of $x(0,2):=B(\tilde{x}(0,3))\in C_n^\mathrm{dR}(\mathcal{L}_3(0))$, regarded as a cycle in the total complex $C_n^\mathrm{dR}(0)$, gives $(-1)^{n+1}[L]\in H_n(\mathcal{L}(0)L;\mathbb{R})$, where $[L]$ is the cycle of constant loops.
	\item[(iv)] $\underline{z}(a,k)\neq0$ only if $\lambda(a)\geq2\varepsilon$ or $a=0$. Moreover, the chain $\underline{z}(0,0)$ lifts to a cycle $\tilde{z}(0,0)\in C_{2d-1}^{S^1}(\mathcal{L}_1(0)L)$, whose homology class $[\tilde{z}(0,0)]$ in the total complex $C_{2d-1}^{S^1}(0)$ becomes $(-1)^{n+1}[\![L]\!]\otimes h^{-d+1}$ under the isomorphism $H_\ast\left(C_{2d-1}^{S^1}(0),\partial^{S^1}\right)\cong H_{n+2d-2}^{S^1}\left(\mathcal{L}(0)L;\mathbb{R}\right)$.
	\item[(v)] $\underline{y}(a,k)\neq0$ only if $\lambda(a)\geq-C_d^\mathrm{GH}(X)$.
\end{itemize}
\end{theorem}

To show that Theorem \ref{theorem:chain} implies Theorem \ref{theorem:deform-Viterbo}, first note that Theorem \ref{theorem:chain}, (i) and (iii) imply that
\begin{equation}
\underline{x}(0):=\sum_{k=3}^\infty\underline{x}(0,k)\in\widehat{C}_{-2}^\lambda \nonumber
\end{equation}
is still a Maurer-Cartan element for the odd dg Lie algebra $\left(\widehat{C}_\ast^\lambda,\partial^\mathrm{tot},\{\cdot,\cdot\}\right)$, so we can use it to deform the differential $\partial^\mathrm{tot}$. Define
\begin{equation}
C_\ast^\lambda(a):=\prod_{k=0}^\infty C_\ast^\lambda(a,k), \nonumber
\end{equation}
which carries the deformed differential $\partial^\mathrm{tot}_{\underline{x}(0)}:C_\ast^\lambda(a)\rightarrow C_{\ast-1}^\lambda(a)$ defined as
\begin{equation}
\partial_{\underline{x}(0)}^\mathrm{tot}(\underline{w}):=\partial^\mathrm{tot}(\underline{w})-\left\{\underline{x}(0),\underline{w}\right\} \nonumber
\end{equation}
for any $\underline{w}\in C_\ast^\lambda(a)$. It follows from \cite{yla}, Lemma 30 that the deformation of the homology $H_\ast\left(C_\ast^\lambda(a),\partial^\mathrm{tot}\right)$ induced by the ``low energy" Maurer-Carten element $\underline{x}(0)$ is trivial, i.e. we have an isomorphism
\begin{equation}\label{eq:iso1}
H_\ast\left(C_\ast^\lambda(a),\partial^\mathrm{tot}_{\underline{x}(0)}\right)\cong H_{\ast+n+\mu(a)-1}^{S^1}(\mathcal{L}(a)L;\mathbb{R}).
\end{equation}
On the other hand, consider the ``high energy" part of the Maurer-Cartan element $\underline{x}^+:=\underline{x}-\underline{x}(0)$, direct computations yield the following identities
\begin{equation}\label{eq:+}
\partial^\mathrm{tot}_{\underline{x}(0)}(\underline{x}^+)-\frac{1}{2}\left\{\underline{x}^+,\underline{x}^+\right\}=0,\textrm{ }\partial^\mathrm{tot}_{\underline{x}(0)}(\underline{y})-\left\{\underline{x}^+,\underline{y}\right\}=\underline{z}.
\end{equation}

\begin{proposition}
Theorem \ref{theorem:chain} implies Theorem \ref{theorem:deform-Viterbo}.
\end{proposition}
\begin{proof}
The isomorphism (\ref{eq:iso1}) implies the existence of $\mathbb{R}$-linear maps
\begin{equation}
\iota:\mathbb{H}_\ast^{S^1}\rightarrow C_\ast^\lambda,\textrm{ }\pi:C_\ast^\lambda\rightarrow\mathbb{H}_\ast^{S^1}\textrm{ and }\kappa:C_\ast^\lambda\rightarrow C_{\ast+1}^\lambda, \nonumber
\end{equation}
such that
\begin{equation}
\partial^\mathrm{tot}_{\underline{x}(0)}\circ\iota=0,\textrm{ }\pi\circ\partial^\mathrm{tot}_{\underline{x}(0)}=0,\textrm{ }\pi\circ\iota=\mathrm{id}_\mathbb{H}, \nonumber
\end{equation}
where $\mathrm{id}_\mathbb{H}$ denotes the identity of $\mathbb{H}_\ast^{S^1}$, and
\begin{equation}
\kappa\circ\partial^\mathrm{tot}_{\underline{x}(0)}+\partial^\mathrm{tot}_{\underline{x}(0)}\circ\kappa=\mathrm{id}_C-\iota\circ\pi, \nonumber
\end{equation}
where $\mathrm{id}_C$ is the identity of $C_\ast^\lambda$. It is not hard to see that the maps $\iota$, $\pi$ and $\kappa$ extend to the completions $\widehat{\mathbb{H}}_\ast^{S^1}$ and $\widehat{C}_\ast^\lambda$ with respect to the action filtrations, and we shall use the same notations to denote their extensions. By Theorem \ref{theorem:chain}, (iv), the map $\pi$ can be chosen so that $\sum_{k=0}^\infty\underline{z}(0,k)\in\widehat{C}_{2d-1}^\lambda$ is mapped to $(-1)^{n+1}[\![L]\!]\otimes h^{-d+1}\in H_{n+2d-2}^{S^1}(\mathcal{L}(0)L;\mathbb{R})$.

Now we apply the homotopy transfer lemma for $L_\infty$-structures (cf. \cite{kic}, Theorem 2.4 and \cite{jlf}, Proposition 4.9), which implies the existence of an $L_\infty$-structure $(\ell_m)_{m\geq1}$ on $\mathbb{H}_\ast^{S^1}$ and an $L_\infty$-homomorphism
\begin{equation}\label{eq:p}
p=(p_m)_{m\geq1}:\left(C_\ast^\lambda,\partial^\mathrm{tot}_{\underline{x}(0)},\{\cdot,\cdot\}\right)\rightarrow\left(\mathbb{H}_\ast^{S^1},(\ell_m)_{m\geq1}\right)
\end{equation}
such that $\ell_1=0$ and $p_1=\pi$. The $L_\infty$-structure $(\ell_m)_{m\geq1}$ and the $L_\infty$-homomorphism $p$ can be taken so that they respect the decompositions over $H_1(L;\mathbb{Z})$, therefore also extend to the completions $\widehat{\mathbb{H}}_\ast^{S^1}$ and $\widehat{C}_\ast^\lambda$. It follows from  the identities (\ref{eq:+}) that the elements
\begin{equation}\label{eq:MC}
\underline{X}:=\sum_{m=1}^\infty\frac{1}{m!}p_m(\underline{x}^+,\cdots,\underline{x}^+)\in\widehat{\mathbb{H}}_{-2}^{S^1},
\end{equation}
\begin{equation}\label{eq:prim}
\underline{Y}:=\sum_{m=1}^\infty\frac{1}{(m-1)!}p_m(\underline{y},\underline{x}^+,\cdots,\underline{x}^+)\in\widehat{\mathbb{H}}_{2d}^{S^1},
\end{equation}
\begin{equation}
\underline{Z}:=\sum_{m=1}^\infty\frac{1}{(m-1)!}p_m(\underline{z},\underline{x}^+,\cdots,\underline{x}^+)\in\widehat{\mathbb{H}}_{2d-1}^{S^1} \nonumber
\end{equation}
satisfy
\begin{equation}
\sum_{m=2}^\infty\frac{1}{m!}\ell_m(\underline{X},\cdots,\underline{X})=0 \nonumber
\end{equation}
and
\begin{equation}
\sum_{m=2}^\infty\frac{1}{(m-1)!}\ell_m(\underline{Y},\underline{X},\cdots,\underline{X})=\underline{Z}. \nonumber
\end{equation}
The infinite sums in the definitions of the homology classes $\underline{X}$, $\underline{Y}$, and $\underline{Z}$ above make sense since by Theorem \ref{theorem:chain}, (iii), $\underline{x}^+(a)\neq0$ only when $\lambda(a)\geq2\varepsilon$. Since $\underline{X}(a)\neq0$ only when $\lambda(a)\geq2\varepsilon$ and $\underline{Y}(a)\neq0$ only when $\lambda(a)\geq-C_d^\mathrm{GH}(X)$, Theorem \ref{theorem:deform-Viterbo}, (iii) holds with $\eta=2\varepsilon$. It remains to show that $\underline{Z}(0)=(-1)^{n+1}[\![L]\!]\otimes h^{-d+1}$. Since the $L_\infty$-homomorphism $p$ respects the decompositions over $H_1(L;\mathbb{Z})$, and $\underline{z}(a,k)\neq0$ only if $\lambda(a)\geq2\varepsilon$ or $a=0$, we obtain
\begin{equation}
\underline{Z}(0)=\pi\left(\sum_{m=0}^\infty\underline{z}(0,m)\right)=(-1)^{n+1}[\![L]\!]\otimes h^{-d+1} \nonumber
\end{equation}
by our choice of $\pi$.
\end{proof}

To overcome the technical difficulties of simultaneous perturbations of infinitely many Kuranishi maps (cf. \cite{kic}, Section 6), we will translate Theorem \ref{theorem:chain} into a statement involving finite energy chains, i.e. chains in the uncompleted complex $C_\ast^\lambda$. From now on, we choose $\varepsilon>0$ in the statement of Theorem \ref{theorem:chain} so that $2\varepsilon$ is less than the minimal symplectic area of $J$-holomorphic discs with boundary on $L\subset X$, see Section \ref{section:CG}. Define
\begin{equation}
\mathcal{F}^mC_\ast^{S^1}:=\bigoplus_{\substack{(a,k)\in H_1(L;\mathbb{Z})\times\mathbb{Z}_{\geq0}\\ \lambda(a)\geq\varepsilon(m+1-k)}}C_\ast^{S^1}(a,k), \nonumber
\end{equation}
which induces a similar filtration $\mathcal{F}^\bullet$ on its quotient complex $C_\ast^\lambda$. It is clear from the definition that the differential $\partial^\mathrm{tot}$ and the Lie bracket $\{\cdot,\cdot\}$ preserve the filtration $\mathcal{F}^\bullet$ on $C_\ast^\lambda$, i.e. $\partial^\mathrm{tot}\mathcal{F}^m\subset\mathcal{F}^m$ and $\left\{\mathcal{F}^m,\mathcal{F}^{m'}\right\}\subset\mathcal{F}^{m+m'}$. In the same way, one can define a filtration $\overline{\mathcal{F}}^\bullet$ on the relative chain complexes $\overline{C}_\ast^{S^1}$ and $\overline{C}_\ast^\lambda$. We will define the chains $\underline{x},\underline{y},\underline{z}\in\widehat{C}_\ast^\lambda$ in the statement of Theorem \ref{theorem:chain} as limits of sequences of finite energy chains in $C_\ast^\lambda$ with respect to the filtration $\mathcal{F}^\bullet$ (and similarly for the corresponding relative chains $\underline{\bar{x}},\underline{\bar{y}},\underline{\bar{z}}\in\widehat{\overline{C}}_\ast^\lambda$).

Note that a priori, the completions of $C_\ast^{S^1}$ (resp. $C_\ast^\lambda$) with respect to the filtration $\mathcal{F}^m$ defined above can be different from $\widehat{C}_\ast^{S^1}$ (resp. $\widehat{C}_\ast^\lambda$) defined using the action filtration $F^\Xi$ (cf. (\ref{eq:Xi})), hence the different notations. However, following a strategy of \cite{kie}, we will show that the limits $\underline{x}$, $\underline{y}$ and $\underline{z}$ actually lie in $\widehat{C}_\ast^\lambda$, see Proposition \ref{proposition:filt} below. We now give the finite energy version of the chain level statement of Theorem \ref{theorem:deform-Viterbo}.

\begin{theorem}\label{theorem:finite-chain}
Let $L$ and $X$ be as before, with $L\subset\mathrm{int}(X)$, and assume that $C_d^\mathrm{GH}(X)<\infty$ for some $d\in\mathbb{N}$. There exist integers $I,U\in\mathbb{Z}_{\geq3}$ and a sequence $(\underline{x}_i,\underline{y}_i,\underline{z}_i,\bar{\underline{x}}_i,\bar{\underline{y}}_i,\bar{\underline{z}}_i)_{i\geq I}$ of chains with $\underline{x}_i,\underline{y}_i,\underline{z}_i\in C_\ast^\lambda$, and $\bar{\underline{x}}_i,\bar{\underline{y}}_i,\bar{\underline{z}}_i\in\overline{C}_\ast^\lambda$, such that the following conditions hold.
\begin{itemize}
	\item[(i)]
	$\underline{x}_i\in\mathcal{F}^1C_{-2}^\lambda$, $\bar{\underline{x}}_i\in\overline{\mathcal{F}}^1\overline{C}_{-2}^\lambda$, $\underline{y}_i\in \mathcal{F}^{-U}C_{2d}^\lambda$, $\bar{\underline{y}}_i\in\overline{\mathcal{F}}^{-U}\overline{C}_{2d}^\lambda$, $\underline{z}_i\in \mathcal{F}^{-1}C_{2d-1}^\lambda$, $\bar{\underline{z}}_i\in\overline{\mathcal{F}}^{-1}\overline{C}_{2d-1}^\lambda$.
	\item[(ii)]
	$\underline{x}_i=\underline{e}_-(\bar{\underline{x}}_i)$, $\underline{y}_i=\underline{e}_-(\bar{\underline{y}}_i)$, $\underline{z}_i=\underline{e}_-(\bar{\underline{z}}_i)$.
	\item[(iii)]
	$\bar{\partial}^\mathrm{tot}(\bar{\underline{x}}_i)-\frac{1}{2}\left\{\bar{\underline{x}}_i,\bar{\underline{x}}_i\right\}\in\overline{\mathcal{F}}^i\overline{C}_{-3}^\lambda$, $\bar{\partial}^\mathrm{tot}(\bar{\underline{y}}_i)-\left\{\bar{\underline{x}}_i,\bar{\underline{y}}_i\right\}-\bar{\underline{z}}_i\in \overline{\mathcal{F}}^{i-U-1}\overline{C}_{2d-1}^\lambda$, $\bar{\partial}^\mathrm{tot}(\bar{\underline{z}}_i)-\left\{\bar{\underline{x}}_i,\bar{\underline{z}}_i\right\}\in \overline{\mathcal{F}}^{i-2}\overline{C}_{2d-2}^\lambda$, where $\bar{\partial}^\mathrm{tot}$ is the differential on $\overline{C}_\ast^\lambda$.
	\item[(iv)]
	$\underline{x}_{i+1}-\underline{e}_+(\bar{\underline{x}}_i)\in\mathcal{F}^iC_{-2}^\lambda$, $\underline{y}_{i+1}-\underline{e}_+(\bar{\underline{y}}_i)\in\mathcal{F}^{i-U-1}C_{2d}^\lambda$, $\underline{z}_{i+1}-\underline{e}_+(\bar{\underline{z}}_i)\in\mathcal{F}^{i-2}C_{2d-1}^\lambda$.
	\item[(v)]$\underline{x}_i(a,k)\neq0$ only if $\lambda(a)\geq2\varepsilon$ or $a=0$, $k\geq3$. Moreover, $\underline{x}_i(0,3)$ admits a lift $\tilde{x}_i(0,3)\in C_{-2}^{S^1}(0,3)$ such that $B\left(\tilde{x}_i(0,3)\right)=x_i(0,2)\in C_{-1}^\mathrm{nd}(0,2)$ is a cycle, whose homology class in the total complex $C_{-1}^\mathrm{nd}(0)$ coincides with $(-1)^{n+1}[L]$ under the isomorphism between the de Rham homology and the singular homology of $\mathcal{L}(0)L$. Here, the chain level marking map $B$ is defined by (\ref{eq:B}).
	\item[(vi)] $\underline{z}_i(a,k)\neq0$ only if $\lambda(a)\geq2\varepsilon$ or $a=0$. Moreover, $\underline{z}_i(0,0)$ lifts to a cycle $\tilde{z}_i(0,0)\in C_{2d-1}^{S^1}(0,0)$, whose homology class $\left[\tilde{z}_i(0,0)\right]$ in the total complex $C_{2d-1}^{S^1}(0)$ corresponds to $(-1)^{n+1}[\![L]\!]\otimes h^{-d+1}$ under the isomorphism between the $S^1$-equivariant de Rham homology and the $S^1$-equivariant singular homology of $\mathcal{L}(0)L$.
	\item[(vii)] $\underline{y}_i(a,k)\neq0$ only if $\lambda(a)\geq-C_d^\mathrm{GH}(X)$.
	\item[(viii)] Consider the following subsets of $H_1(L;\mathbb{Z})$:
	\begin{equation}\label{eq:A1}
	\underline{A}_x:=\left\{a\in H_1(L;\mathbb{Z})\vert\exists(i,k)\textrm{ such that }\bar{\underline{x}}_i(a,k)\neq0\right\},
	\end{equation}
	\begin{equation}\label{eq:A2}
	\underline{A}_x^+:=\left\{a_1+\cdots+a_m\vert m\geq1, a_1,\cdots,a_m\in\underline{A}_x\right\},
	\end{equation}
	\begin{equation}\label{eq:A3}
	\underline{A}_{y,z}:=\left\{a\in H_1(L;\mathbb{Z})\left\vert\exists(i,k)\textrm{ such that }\left(\bar{\underline{y}}_i(a,k),\bar{\underline{z}}_i(a,k)\right)\neq(0,0)\right.\right\},
	\end{equation}
	\begin{equation}\label{eq:A4}
	\underline{A}_{y,z}^+:=\left\{a_1+\cdots+a_m\vert m\geq1, a_1\in\underline{A}_{y,z},a_2\cdots,a_m\in\underline{A}_x\right\}.
	\end{equation}
	Then for any $\Xi>0$,
	\begin{equation}
	\underline{A}_x^+(\Xi):=\left\{a\in\underline{A}_x^+\vert\lambda(a)<\Xi\right\}\textrm{ and }\underline{A}_{y,z}^+(\Xi):=\left\{a\in\underline{A}_{y,z}^+\vert\lambda(a)<\Xi\right\} \nonumber
	\end{equation}
	are finite sets.
\end{itemize}
\end{theorem}

We will show that Theorem \ref{theorem:finite-chain} implies Theorem \ref{theorem:chain}. We start with the following generalization of \cite{yla}, Lemma 35.

\begin{lemma}\label{lemma:induction}
Let $I,U\in\mathbb{Z}_{\geq3}$ and $\underline{x}_i,\underline{y}_i,\underline{z}_i,\bar{\underline{x}}_i,\bar{\underline{y}}_i,\bar{\underline{z}}_i$ be as in Theorem \ref{theorem:finite-chain}. Then there exists a sequence
\begin{equation}
(\underline{x}_{i,j},\underline{y}_{i,j},\underline{z}_{i,j},\bar{\underline{x}}_{i,j},\bar{\underline{y}}_{i,j},\bar{\underline{z}}_{i,j})_{i\geq I,j\geq0}
\end{equation}
of (relative) de Rham chains satisfying the following conditions:
	\begin{itemize}
	\item[(i)] $\underline{x}_{i,0}=\underline{x}_i$, $\underline{y}_{i,0}=\underline{y}_i$, $\underline{z}_{i,0}=\underline{z}_i$, $\bar{\underline{x}}_{i,0}=\bar{\underline{x}}_i$,  $\bar{\underline{y}}_{i,0}=\bar{\underline{y}}_i$, $\bar{\underline{z}}_{i,0}=\bar{\underline{z}}_i$.
	\item[(ii)] $\underline{x}_{i,j}\in\mathcal{F}^1C_{-2}^\lambda$, $\bar{\underline{x}}_{i,j}\in\overline{\mathcal{F}}^1\overline{C}_{-2}^\lambda$, $\underline{y}_{i,j}\in \mathcal{F}^{-U}C_{2d}^\lambda$, $\bar{\underline{y}}_{i,j}\in\overline{\mathcal{F}}^{-U}\overline{C}_{2d}^\lambda$, $\underline{z}_{i,j}\in\mathcal{F}^{-1}C_{2d-1}^\lambda$, $\bar{\underline{z}}_{i,j}\in\overline{\mathcal{F}}^{-1}\overline{C}_{2d-1}^\lambda$.
	\item[(iii)] $\underline{x}_{i,j}=\underline{e}_-(\underline{\bar{x}}_{i,j})$, $\underline{y}_{i,j}=\underline{e}_-(\bar{\underline{y}}_{i,j})$, $\underline{z}_{i,j}=\underline{e}_-(\bar{\underline{z}}_{i,j})$.
	\item[(iv)] $\bar{\partial}^\mathrm{tot}(\bar{\underline{x}}_{i,j})-\frac{1}{2}\left\{\bar{\underline{x}}_{i,j},\bar{\underline{x}}_{i,j}\right\}\in\overline{\mathcal{F}}^{i+j}\overline{C}_{-3}^\lambda$, $\bar{\partial}^\mathrm{tot}(\bar{\underline{y}}_{i,j})-\left\{\bar{\underline{x}}_{i,j},\bar{\underline{y}}_{i,j}\right\}-\bar{\underline{z}}_{i,j}\in\overline{\mathcal{F}}^{i+j-U-1}\overline{C}_{2d-1}^\lambda$, $\bar{\partial}^\mathrm{tot}(\bar{\underline{z}}_{i,j})-\left\{\bar{\underline{x}}_{i,j},\bar{\underline{z}}_{i,j}\right\}\in\overline{\mathcal{F}}^{i+j-2}\overline{C}_{2d-2}^\lambda$.
	\item[(v)] $\underline{x}_{i+1,j}-\underline{e}_+(\bar{\underline{x}}_{i,j})\in\mathcal{F}^{i+j}C_{-2}^\lambda$, $\underline{y}_{i+1,j}-\underline{e}_+(\bar{\underline{y}}_{i,j})\in\mathcal{F}^{i+j-U-1}C_{2d}^\lambda$, $\underline{z}_{i+1,j}-\underline{e}_+(\bar{\underline{z}}_{i,j})\in\mathcal{F}^{i+j-2}C_{2d-1}^\lambda$.
	\item[(vi)] $\bar{\underline{x}}_{i,j+1}-\bar{\underline{x}}_{i,j}\in\overline{\mathcal{F}}^{i+j}\overline{C}_{-2}^\lambda$, $\bar{\underline{y}}_{i,j+1}-\bar{\underline{y}}_{i,j}\in\overline{\mathcal{F}}^{i+j-U-1}\overline{C}_{2d}^\lambda$, $\bar{\underline{z}}_{i,j+1}-\bar{\underline{z}}_{i,j}\in\overline{\mathcal{F}}^{i+j-2}\overline{C}_{2d-1}^\lambda$.
	\item[(vii)] $\bar{\underline{x}}_{i,j}(a,k)\neq0$ only if $\lambda(a)\geq2\varepsilon$ or $a=0$, $k\geq3$. Moreover, $\bar{\underline{x}}_{i,j}(0,3)$ lifts to a chain $\bar{\tilde{x}}_{i,j}(0,3)\in\overline{C}^{S^1}_{-2}$ satisfying $\overline{B}\left(\bar{\tilde{x}}_{i,j}(0,3)\right)=\bar{x}_{i,j}(0,2)$, where $\overline{B}:\overline{C}_\ast^{S^1}(a,k+1)\rightarrow\overline{C}_{\ast+1}^{\mathrm{nd}}(a,k)$ is the marking map for relative chains defined in the same way as (\ref{eq:B}), and $\bar{x}_{i,j}(0,2)\in\overline{C}_{-1}^\mathrm{nd}(0,2)$ is a cycle whose homology class $\left[\bar{x}_{i,j}(0,2)\right]$ in the total complex coincides with $(-1)^{n+1}[L]$ under the isomorphism between the relative de Rham homology and the singular homology of $\mathcal{L}(0)L$.
	\item[(viii)] $\bar{\underline{z}}_{i,j}(a,k)\neq0$ only if $\lambda(a)\geq2\varepsilon$ or $a=0$. Moreover, $\bar{\underline{z}}_{i,j}(0,0)\in\overline{C}_{2d-1}^\lambda$ lifts to a cycle $\bar{\tilde{z}}_{i,j}(0,0)\in\overline{C}_{2d-1}^{S^1}(0,0)$ whose homology class $\left[\bar{\tilde{z}}_{i,j}(0,0)\right]$ in the total complex corresponds to $(-1)^{n+1}[\![L]\!]\otimes h^{-d+1}$ under the isomorphism between the relative $S^1$-equivariant de Rham homology and the $S^1$-equivariant singular homology of $\mathcal{L}(0)L$.
	\item[(ix)] $\bar{\underline{y}}_{i,j}(a,k)\neq0$ only if $\lambda(a)\geq-C_d^\mathrm{GH}(X)$.
	\item[(x)] If there exists a triple $(i,j,k)\in\mathbb{Z}_{\geq I}\times\mathbb{Z}_{\geq0}\times\mathbb{Z}_{\geq0}$ such that 
	\begin{itemize}
	\item $\bar{\underline{x}}_{i,j}(a,k)\neq0$, then $a\in\underline{A}_x^+$.
	\item $\left(\bar{\underline{y}}_{i,j}(a,k),\bar{\underline{z}}_{i,j}(a,k)\right)\neq(0,0)$, then $a\in\underline{A}_{y,z}^+$.
	\end{itemize}
	\end{itemize}
\end{lemma}
\begin{proof}
This is a slight modification of the proof of \cite{yla}, Lemma 35, which deals with the $d=1$ case. The only essential difference is that the $y$ and $z$ type chains now have different degrees. We therefore omit the details and leave the proof to the reader. Note that although (\ref{eq:bracket-pr}) differs from \cite{yla}, (3.70), the actual form of the chain level string bracket $\{\cdot,\cdot\}$ on $C_\ast^\lambda$ and $\overline{C}_\ast^\lambda$ does not affect the proof, since we only need to use its property as an odd Lie bracket. Moreover, during the proof, we define the chains $\underline{x}_{i,j+1}$, $\underline{y}_{i,j+1}$ and $\underline{z}_{i,j+1}$ as
\[
\underline{x}_{i,j+1}:=\underline{e}_-(\bar{\underline{x}}_{i,j+1}),\textrm{ }\underline{y}_{i,j+1}:=\underline{e}_-(\bar{\underline{y}}_{i,j+1}),\textrm{ }\underline{z}_{i,j+1}:=\underline{e}_-(\bar{\underline{z}}_{i,j+1}). \qedhere
\]
\end{proof}

\begin{proposition}\label{proposition:filt}
Theorem \ref{theorem:finite-chain} implies Theorem \ref{theorem:chain}.
\end{proposition}
\begin{proof}
Fix an integer $i\geq I\in\mathbb{Z}_{\geq3}$. For every $j\in\mathbb{Z}_{\geq0}$, applying the map $\underline{e}_-$ to the chains in Lemma \ref{lemma:induction}, (vi) we obtain
\begin{equation}
\underline{x}_{i,j+1}-\underline{x}_{i,j}\in\mathcal{F}^{i+j}C_{-2}^\lambda,\textrm{ }\underline{y}_{i,j+1}-\underline{y}_{i,j}\in\mathcal{F}^{i+j-U-1}C_{2d}^\lambda,\textrm{ }\underline{z}_{i,j+1}-\underline{z}_{i,j}\in\mathcal{F}^{i+j-2}C_{2d-1}^\lambda. \nonumber
\end{equation}
Thus the limits
\begin{equation}
\underline{x}:=\lim_{j\rightarrow\infty}\underline{x}_{i,j},\textrm{ }\underline{y}:=\lim_{j\rightarrow\infty}\underline{y}_{i,j},\textrm{ }\underline{z}:=\lim_{j\rightarrow\infty}\underline{z}_{i,j} \nonumber
\end{equation}
exist in the completion of $C_\ast^\lambda$ with respect to the filtration $\mathcal{F}^\bullet$, and they satisfy the equations
\begin{equation}
\partial^\mathrm{tot}(\underline{x})-\frac{1}{2}\left\{\underline{x},\underline{x}\right\}=0,\textrm{ }\partial^\mathrm{tot}(\underline{y})-\left\{\underline{x},\underline{y}\right\}=\underline{z}. \nonumber
\end{equation}
by Theorem \ref{theorem:finite-chain}, (iv) and the compatibility of $\underline{e}_-$ with the odd Lie brackets. To show that we actually have $\underline{x}\in\widehat{C}_{-2}^\lambda$, $\underline{y}\in\widehat{C}_{2d}^\lambda$ and $\underline{z}\in\widehat{C}_{2d-1}^\lambda$, i.e. the completion of $C_\ast^\lambda$ with respect to the filtrations $F^\Xi$ and $\mathcal{F}^m$ coincide, we need to verify that for any $\Xi>0$, there are only finitely many classes $a\in H_1(L;\mathbb{Z})$ with $\lambda(a)<\Xi$ and $\left(x(a,k),y(a,k),z(a,k)\right)\neq0$ for some $k\in\mathbb{Z}_{\geq0}$. By Lemma \ref{lemma:induction} (x), such a class $a$ must satisfy $a\in\underline{A}_x^+\cup\underline{A}_{y,z}^+$. On the other hand, Theorem \ref{theorem:finite-chain}, (vii) implies that $\underline{A}_x^+(\Xi)\cup\underline{A}_{y,z}^+(\Xi)$ is a finite set. Finally, Theorem \ref{theorem:chain}, (iii) and (iv) follow from Lemma \ref{lemma:induction}, (viii) and (x).
\end{proof}

\subsection{Defining the chains}\label{section:chain-d}

In order to define de Rham chains on our models $\{\mathcal{L}_{k+1}\}_{k\geq0}$ of the free loop space of $L$ using the moduli spaces introduced in Section \ref{section:CG}, we need to have strongly smooth maps from these moduli spaces to $\{\mathcal{L}_{k+1}\}_{k\geq0}$.

For $\bar{a}\in\pi_2(X,L)$ and $(T,B)\in\mathcal{T}(k+1,\bar{a})$, the interior evaluation map (\ref{eq:ev-int}) gives rise to a smooth map
\begin{equation}
\mathrm{ev}_\mathrm{int}^P:\prod_{v\in V_{\mathrm{int}}(T)}P\times\mathcal{L}_{k_v+1}(\partial B(v))\rightarrow\prod_{e\in E_{\mathrm{int}}(T)}(P\times L)^2 \nonumber
\end{equation}
on our finite-dimensional models of the free loop space. Using the concatenation map (\ref{eq:con}), which is also smooth, we obtain a smooth map
\begin{equation}\label{eq:ev-con}
\left(\prod_{e\in E_{\mathrm{int}}(T)}P\times L\right)\textrm{ }{{}_\Delta\times_{\mathrm{ev}_\mathrm{int}^P}}\textrm{ }\left(\prod_{v\in V_{\mathrm{int}}(T)}P\times\mathcal{L}_{k_v+1}(\partial B(v))\right)\rightarrow P\times\mathcal{L}_{k+1}(a).
\end{equation}
Similarly, we have a smooth map
\begin{equation}\label{eq:ev-con1}
\left(\prod_{e\in E_\mathrm{int}(T)}P\times L\right){{}_\Delta\times_{\mathrm{ev}_\mathrm{int}^P}}\left(\prod_{v\in V_\mathrm{int}(T)}P\times\mathcal{L}_{k_v+1}\left(\partial \ring{B}(v)\right)\right)\rightarrow P\times\mathcal{L}_{k+1}(a)
\end{equation}
for any $\ring{\bar{a}}\in\pi_2(X,x,L)$ and $(T,\ring{B},v_0)\in\mathcal{T}(k+1,\ring{\bar{a}})$.

Using Wang's model $\{\mathcal{L}_{k+1}\}_{k\geq0}$, it is easy to prove the following.

\begin{lemma}[\cite{yla}, Proposition 51]\label{lemma:smev}
For $k,m\in\mathbb{Z}_{\geq0}$, and $P=\{m\}$ or $[m,m+1]$, there are strongly smooth maps
\begin{equation}\label{eq:sc1}
\mathrm{Ev}^\mathcal{R}:\overline{\mathcal{R}}_{k+1}(L,\bar{a};P)\rightarrow P\times\mathcal{L}_{k+1}(a),\textrm{ where }\lambda(a)<(m+1-k)\varepsilon,
\end{equation}
\begin{equation}\label{eq:sc2}
\mathrm{Ev}^{\mathcal{R}}_\vartheta:\overline{\mathcal{R}}_{k+2,\vartheta}(L,\bar{a};P)\rightarrow P\times\mathcal{L}_{k+1}(a),\textrm{ where }\lambda(a)<(m+1-k)\varepsilon,
\end{equation}
\begin{equation}\label{eq:sc3}
{}_l\mathrm{Ev}^\mathcal{R}:{}_l\overline{\mathcal{R}}_{k+1}^1(x,L,\ring{\bar{a}};P)\rightarrow P\times\mathcal{L}_{k+1}(a),\textrm{ where }\lambda(a)<(m-k-U)\varepsilon,
\end{equation}
\begin{equation}\label{eq:sc4}
{}_{l-1}\mathrm{Ev}^{S^1}:{}_{l-1}\overline{\mathcal{R}}_{k+1}^{S^1}(x,L,\ring{\bar{a}};P)\rightarrow P\times\mathcal{L}_{k+1}(a),\textrm{ where }\lambda(a)<(m-k-U)\varepsilon,
\end{equation}
\begin{equation}\label{eq:sc5}
{}_l^{j,j+1}\mathrm{Ev}^\mathcal{R}:{}_l^{j,j+1}\overline{\mathcal{R}}_{k+1}^1(x,L,\ring{\bar{a}};P)\rightarrow P\times\mathcal{L}_{k+1}(a),\textrm{ where }\lambda(a)<(m-k-U)\varepsilon,
\end{equation}
\begin{equation}\label{eq:sc6}
{}_{l-1}\mathrm{Ev}_i:{}_{l-1}\overline{\mathcal{R}}_{k+1,\tau_i}^1(x,L,\ring{\bar{a}};P)\rightarrow P\times\mathcal{L}_{k+1}(a),\textrm{ where }\lambda(a)<(m-k-U)\varepsilon,
\end{equation}
such that the following diagram commutes for every $(T,B)\in\mathcal{T}(k+2,\bar{a})$:
\begin{equation}
	\begin{tikzcd}[font=\small]
	(\prod_eP\times L){{}_\Delta\times_{\mathrm{ev}_\mathrm{int}}}\left(\prod_{v\neq v_0}\mathcal{R}_{k_v+1}(B(v);P)\times\mathcal{R}_{k_{v_0}+1,\vartheta}(B(v_0);P)\right) \arrow[r] \arrow[d] &\overline{\mathcal{R}}_{k+2,\vartheta}(\bar{a};P) \arrow[d] \\
	(\prod_eP\times L){{}_\Delta\times_{\mathrm{ev}_\mathrm{int}^P}}\left(\prod_vP\times\mathcal{L}_{k_v+1}(\partial B(v))\right) \arrow[r,"(\ref{eq:ev-con})"] &P\times\mathcal{L}_{k+2}(a) \nonumber
	\end{tikzcd}
\end{equation}
where the first horizontal map is defined from (\ref{eq:corner1}) by setting $d=0$, and the vertical maps are given by the $\mathrm{Ev}^\mathcal{R}$ and $\mathrm{Ev}_\vartheta^\mathcal{R}$ above; and the diagram
\begin{equation}\label{eq:cd2}
\begin{tikzcd}[font=\small]
(\prod_eP\times L){{}_\Delta\times_{\mathrm{ev}_\mathrm{int}}}\left(\prod_{v\neq v_0}\mathcal{R}_{k_v+1}\left(\ring{B}(v);P\right)\times{}_l\mathcal{R}^1_{k_{v_0}+1}\left(\ring{B}(v_0);P\right)\right) \arrow[r] \arrow[d] &{}_l\overline{\mathcal{R}}_{k+1}^1(\ring{\bar{a}};P) \arrow[d] \\
(\prod_eP\times L){{}_\Delta\times_{\mathrm{ev}_\mathrm{int}^P}}\left(\prod_vP\times\mathcal{L}_{k_v+1}\left(\partial \ring{B}(v)\right)\right) \arrow[r,"(\ref{eq:ev-con1})"] &P\times\mathcal{L}_{k+1}(a)
\end{tikzcd}
\end{equation}
commutes for every $(T,\ring{B},v_0)\in\mathcal{T}(k+1,\ring{\bar{a}})$, where the first horizontal map is defined from (\ref{eq:corner2}) by setting $d=0$, and the vertical maps are given by ${}_l\mathrm{Ev}^\mathcal{R}$. 
	
In the above, we have abbreviated the notations of the moduli spaces so that the boundary conditions specified by the Lagrangian submanifold $L$ and the asymptotic conditions specified by a Hamiltonian orbit $x$ of $X_{H_t}$ are omitted. In the commutative diagram (\ref{eq:cd2}) above, one can also include cylinder bubbles in $\prod_{i=1}^{r_2}{}_{j_i}\mathcal{M}(y_{j_i},x)$. There are similar compatibility diagrams for the strongly smooth maps $\mathrm{Ev}^\mathcal{R}$, ${}_{l-1}\mathrm{Ev}^{S^1}$, ${}_l^{j,j+1}\mathrm{Ev}^\mathcal{R}$ and ${}_{l-1}\mathrm{Ev}_i$, which we have omitted.
\end{lemma}

The following theorem is a variant of \cite{kic}, Theorem 7.2, see also \cite{yla}, Appendix A.2 for details.

\begin{theorem}\label{theorem:pushforward}
Let $(\mathbb{X},\widehat{\mathcal{U}})$ be a compact, oriented admissible K-space of dimension $d$, equipped with an admissible map $\hat{f}:(\mathbb{X},\widehat{\mathcal{U}})\rightarrow\mathcal{L}_{k+1}$, a differential form $\hat{\omega}$, and a CF-perturbation $\widehat{\mathcal{S}}=(\widehat{\mathcal{S}}^\theta)_{0<\theta\leq1}$. We assume that $\widehat{\mathcal{S}}$ is transversal to $0$, and $\mathrm{ev}_0\circ\hat{f}:(\mathbb{X},\widehat{\mathcal{U}})\rightarrow L$ is a corner-stratified strong submersion with respect to $\widehat{\mathcal{S}}$. Then one can define a de Rham chain
\begin{equation}
\hat{f}_\ast(\mathbb{X},\widehat{\mathcal{U}},\hat{\omega},\widehat{\mathcal{S}}^\theta)\in C_{d-|\hat{\omega}|}^\mathrm{dR}(\mathcal{L}_{k+1}) \nonumber
\end{equation}
for sufficiently small $\theta>0$, so that the following are true.
\begin{itemize}
\item[(i)] (Stokes' formula)
\begin{equation}
\begin{split}
\partial\left(\hat{f}_\ast(\mathbb{X},\widehat{\mathcal{U}},\hat{\omega},\widehat{\mathcal{S}}^\theta)\right)&=(-1)^{|\hat{\omega}|}(\hat{f}|_{\partial\mathbb{X}})_\ast(\partial\mathbb{X},\widehat{\mathcal{U}}|_{\partial\mathbb{X}},\hat{\omega}|_{\partial\mathbb{X}},\widehat{\mathcal{S}}^\theta|_{\partial\mathbb{X}}) \\
&+(-1)^{|\hat{\omega}|+1}\hat{f}_\ast(\mathbb{X},\widehat{\mathcal{U}},d\hat{\omega},\widehat{\mathcal{S}}^\theta), \nonumber
\end{split}
\end{equation}
where $\partial\mathbb{X}$ is the normalized boundary of $\mathbb{X}$.
\item[(ii)] (Fiber product formula) Suppose we are given the following data
\begin{itemize}
	\item a compact oriented admissible K-space $(\mathbb{X}_i,\widehat{\mathcal{U}}_i)$ of dimension $d_i$;
	\item an admissible map $\hat{f}_i:(\mathbb{X}_i,\widehat{\mathcal{U}}_i)\rightarrow\mathcal{L}_{k_i+1}$;
	\item a differential form $\hat{\omega}_i$ on $(\mathbb{X}_i,\widehat{\mathcal{U}}_i)$;
	\item a CF-perturbation $\widehat{\mathcal{S}}_i$ on $(\mathbb{X}_i,\widehat{\mathcal{U}}_i)$ such that $\mathrm{ev}_0\circ\hat{f}_i:(\mathbb{X}_i,\widehat{\mathcal{U}}_i)\rightarrow L$ is corner-stratified strongly submersive with respect to $\widehat{\mathcal{S}}_i$
\end{itemize}
for $i=1,2$ such that $d_1+d_2=d$. For each $1\leq j\leq k_1$, consider the fiber product of admissible K-spaces
\begin{equation}
(\mathbb{X}_{12},\widehat{\mathcal{U}}_{12}):=(\mathbb{X}_1,\widehat{\mathcal{U}}_1)\textrm{ }{{}_{\mathrm{ev}_j\circ\hat{f}_1}\times_{\mathrm{ev}_0\circ\hat{f}_2}}\textrm{ }(\mathbb{X}_2,\widehat{\mathcal{U}}_2). \nonumber
\end{equation}
equipped with the CF-perturbation $\widehat{\mathcal{S}}_{12}:=\widehat{\mathcal{S}}_1\times\widehat{\mathcal{S}}_2$ (cf. \cite{fooo4}, \S10.2), the differential form $\hat{\omega}_{12}:=(-1)^{(d-|\hat{\omega}|_1-n)|\hat{\omega}_2|}\hat{\omega}_1\times\hat{\omega}_2$, and an admissible map
\begin{equation}
\begin{split}
\hat{f}_{12}:(\mathbb{X}_{12},\widehat{\mathcal{U}}_{12})&\rightarrow\mathcal{L}_{k_1+k_2},\\
(f_{12})_{p_1,p_2}(x_1,x_2)&:=\mathrm{con}_j\left((f_1)_{p_1}(x_1),(f_2)_{p_2}(x_2)\right), \nonumber
\end{split}
\end{equation}
where $x_1\in U_{p_1}$, $x_2\in U_{p_2}$, and $\mathrm{ev}_j\circ f_{p_1}(x_1)=\mathrm{ev}_0\circ f_{p_2}(x_2)$. Then we have
\begin{equation}
(\hat{f}_{12})_\ast\left(\mathbb{X}_{12},\widehat{\mathcal{U}}_{12},(-1)^{|\hat{\omega}_{12}|+n}\hat{\omega}_{12},\widehat{\mathcal{S}}_{12}^\theta\right)=(\hat{f}_1)_\ast(\mathbb{X}_1,\widehat{\mathcal{U}}_1,\hat{\omega}_1,\widehat{\mathcal{S}}_1^\theta)\circ_j(\hat{f}_2)_\ast(\mathbb{X}_2,\widehat{\mathcal{U}}_2,\hat{\omega}_2,\widehat{\mathcal{S}}_2^\theta). \nonumber
\end{equation}
\end{itemize}

Similarly, if we replace the admissible map $\hat{f}$ above with an admissible map $\hat{f}:(\mathbb{X},\widehat{\mathcal{U}})\rightarrow[a,b]\times\mathcal{L}_{k+1}$, where $a<b$, and
\begin{equation}
\left(\mathrm{pr}_{[a,b]}\circ\hat{f},\mathrm{ev}_0\circ\mathrm{pr}_{\mathcal{L}_{k+1}}\circ\hat{f}\right):(\mathbb{X},\widehat{\mathcal{U}})\rightarrow[a,b]\times L \nonumber
\end{equation}
is a corner-stratified strong submersion, where $\mathrm{pr}_{[a,b]}$ and $\mathrm{pr}_{\mathcal{L}_{k+1}}$ are projections to $[a,b]$ and $\mathcal{L}_{k+1}$, respectively, then there is a well-defined relative de Rham chain
\begin{equation}
\hat{f}_\ast(\mathbb{X},\widehat{\mathcal{U}},\hat{\omega},\widehat{\mathcal{S}}^\theta)\in \overline{C}_{d-|\hat{\omega}|-1}^\mathrm{dR}(\mathcal{L}_{k+1}) \nonumber
\end{equation}
for sufficiently small $\theta>0$, which satisfies
\begin{equation}\label{eq:epm}
\begin{split}
e_\pm\left(\hat{f}_\ast(\mathbb{X},\widehat{\mathcal{U}},\hat{\omega},\widehat{\mathcal{S}}^\theta)\right)&=(-1)^{d-1}(\hat{f}|_{\partial_\pm\mathbb{X}})_\ast\left(\partial_\pm\mathbb{X},\widehat{\mathcal{U}}|_{\partial_\pm\mathbb{X}},\hat{\omega}|_{\partial_\pm\mathbb{X}},\widehat{\mathcal{S}}^\theta|_{\partial_\pm\mathbb{X}}\right) \\
&\in C_{d-|\hat{\omega}|-1}^\mathrm{dR}(\mathcal{L}_{k+1}).
\end{split}
\end{equation}
The Stokes' formula in this case is
\begin{equation}
\begin{split}
\partial\left(\hat{f}_\ast(\mathbb{X},\widehat{\mathcal{U}},\hat{\omega},\widehat{\mathcal{S}}^\theta)\right)&=(-1)^{|\hat{\omega}|}(\hat{f}|_{\partial_h\mathbb{X}})_\ast\left(\partial\mathbb{X},\widehat{\mathcal{U}}|_{\partial_h\mathbb{X}},\hat{\omega}|_{\partial_h\mathbb{X}},\widehat{\mathcal{S}}^\theta|_{\partial_h\mathbb{X}}\right)\\
&+(-1)^{|\hat{\omega}|+1}\hat{f}_\ast(\mathbb{X},\widehat{\mathcal{U}},d\hat{\omega},\widehat{\mathcal{S}}^\theta), \nonumber
\end{split}
\end{equation}
where $\partial_h\mathbb{X}=\hat{f}^{-1}\left(\{a,b\}\times\mathcal{L}_{k+1}\right)$ is the horizontal boundary, $\partial_-\mathbb{X}=\hat{f}^{-1}(\{a\}\times\mathcal{L}_{k+1})$ and $\partial_+\mathbb{X}=\hat{f}^{-1}(\{b\}\times\mathcal{L}_{k+1})$. The fiber product formula takes the same form as above, with the admissible maps $\hat{f}_i$ taking the form $(\mathbb{X}_i,\widehat{\mathcal{U}}_i)\rightarrow[a,b]\times\mathcal{L}_{k_i+1}$.
\end{theorem}

Let $X$ be a Liouville domain with $c_1(X)=0$ and $C_d^\mathrm{GH}(X)<\infty$ for some $d\in\mathbb{N}$. It follows that there exists a cochain $\tilde{\beta}=\sum_{l=0}^\infty\beta_{l}\otimes u^{-d-l+1}\in\mathit{SC}_{S^1}^{-2d+1}(X)$, where $\beta_{l}\in\mathit{SC}^{2l-1}(X)$ and only finitely many cochains $\beta_l$ are non-vanishing, such that
\begin{equation}\label{eq:cl}
\partial^{S^1}(\tilde{\beta})=\left(\partial(\beta_0)+\sum_{l=1}^\infty\delta_l(\beta_l)\right)\otimes u^{-d+1}=e_X\otimes u^{-d+1}.
\end{equation}
Note that the terms in $\partial^{S^1}(\tilde{\beta})$ with different powers of $u$ necessarily vanish due to the form of the right-hand side of (\ref{eq:cl}), which gives the first equality above.

Fix $a\in H_1(L;\mathbb{Z})$. For a closed Lagrangian submanifold $L\subset\mathrm{int}(X)$ that is oriented and $\mathit{Spin}$, choose $a_l\in H_1(L;\mathbb{Z})$ for each $l\in\mathbb{Z}_{\geq0}$ such that $a_l=0$ when $\beta_l=0$ and $\sum_{l=0}^\infty a_l=a$. For $\ring{\bar{a}}_l\in\pi_2(X,\beta_l,L)$ with $\partial\ring{\bar{a}}_l=a_l$, and $P\in\left\{\{m\},[m,m+1]\right\}$ for some $m\in\mathbb{Z}_{\geq0}$, consider the moduli spaces
\begin{equation}
{}_l\mathcal{R}_{k+1}^1(\beta_l,L,\ring{\bar{a}}_l;P),\textrm{ }l\in\mathbb{Z}_{\geq0}.
\end{equation}
Let $\ring{\bar{a}}=\sum_{l=0}^\infty\ring{\bar{a}}_{l-1}$, it follows that $\partial\ring{\bar{a}}=a$. 

Recall that we have the following data for every $k,m,l\in\mathbb{Z}_{\geq0}$ and $P\in\{\{m\},[m,m+1]\}$.
\begin{itemize}
	\item[(i)] Compact admissible K-spaces, where $l\geq1$ in (\ref{eq:S1}) and (\ref{eq:tau}), $l\geq2$ in (\ref{eq:jj}):
	\begin{equation}\label{eq:d}
	\overline{\mathcal{R}}_{k+1}(L,\bar{a};P),\textrm{ }\lambda(a)<(m+1-k)\varepsilon,
	\end{equation}
	\begin{equation}\label{eq:theta}
	\overline{\mathcal{R}}_{k+2,\vartheta}(L,\bar{a};P),\textrm{ }\lambda(a)<(m+1-k)\varepsilon,
	\end{equation}
	\begin{equation}\label{eq:CG}
	{}_l\overline{\mathcal{R}}_{k+1}^1(\beta_{l},L,\ring{\bar{a}}_{l};P),\textrm{ }\lambda(a)<(m-k-U)\varepsilon,
	\end{equation}
	\begin{equation}\label{eq:S1}
	{}_{l-1}\overline{\mathcal{R}}^{S^1}_{k+1}(\beta_{l},L,\ring{\bar{a}}_{l};P),\textrm{ }\lambda(a)<(m-k-U)\varepsilon,
	\end{equation}
	\begin{equation}\label{eq:tau}
	{}_{l-1}\overline{\mathcal{R}}_{k+1,\tau_i}^1(\beta_{l},L,\ring{\bar{a}}_{l};P),\textrm{ }0\leq i\leq k,\textrm{ }\lambda(a)<(m-k-U)\varepsilon,
	\end{equation}
	\begin{equation}\label{eq:jj}
	{}_l^{j,j+1}\overline{\mathcal{R}}_{k+1}^1(\beta_{l},L,\ring{\bar{a}}_{l};P),\textrm{ }1\leq j\leq l-1,\textrm{ }\lambda(a)<(m-k-U)\varepsilon,
	\end{equation}
	and admissible CF-perturbations on these moduli spaces.
	
	Moreover, Kuranishi structures and CF-perturbations on (\ref{eq:d}), (\ref{eq:theta}), (\ref{eq:CG}), (\ref{eq:S1}) and (\ref{eq:jj}) are $\mathbb{Z}_{k+1}$-invariant with respect to the cyclic permutation of the boundary marked points $z_0\cdots,z_k$. For the moduli spaces (\ref{eq:tau}), we require that $\mathbb{Z}_{k+1}$ acts transitively on the set $\left\{\left({}_{l-1}\overline{\mathcal{R}}_{k+1,\tau_i}^1(\beta_l,L,\ring{\bar{a}}_l;P),\widehat{\mathcal{S}}_i^\theta\right)\right\}_{i=0}^k$ of K-spaces together with their admissible CF-perturbations.
	\item[(ii)] Admissible maps
	\begin{equation}\label{eq:ev1}
	\mathrm{Ev}^\mathcal{R}:\overline{\mathcal{R}}_{k+1}(L,\bar{a};P)\rightarrow P\times\mathcal{L}_{k+1}(a),
	\end{equation}
	\begin{equation}\label{eq:ev2}
	\mathrm{Ev}_\vartheta^\mathcal{R}:\overline{\mathcal{R}}_{k+2,\vartheta}(L,\bar{a};P)\rightarrow P\times\mathcal{L}_{k+2}(a),
	\end{equation}
	\begin{equation}\label{eq:ev3}
	{}_l\mathrm{Ev}^\mathcal{R}:{}_l\overline{\mathcal{R}}_{k+1}^1(\beta_{l},L,\ring{\bar{a}}_{l};P)\rightarrow P\times\mathcal{L}_{k+1}(a_{l}),
	\end{equation}
	\begin{equation}\label{eq:ev4}
	{}_{l-1}\mathrm{Ev}^{S^1}: {}_{l-1}\overline{\mathcal{R}}^{S^1}_{k+1}(\beta_{l},L,\ring{\bar{a}}_{l};P)\rightarrow P\times\mathcal{L}_{k+1}(a_{l}),
	\end{equation}
	\begin{equation}\label{eq:ev5}
	{}_{l-1}\mathrm{Ev}_i:{}_{l-1}\overline{\mathcal{R}}_{k+1,\tau_i}^1(\beta_{l},L,\ring{\bar{a}}_l;P)\rightarrow P\times\mathcal{L}_{k+1}(a_{l}),
	\end{equation}
	\begin{equation}\label{eq:ev6}
	{}_l^{j,j+1}\mathrm{Ev}^\mathcal{R}:{}_l^{j,j+1}\overline{\mathcal{R}}_{k+1}^1(\beta_{l},L,\ring{\bar{a}}_{l};P)\rightarrow P\times\mathcal{L}_{k+1}(a_{l}),
	\end{equation}
	such that their compositions with $\mathrm{id}_P\times\mathrm{ev}_0^\mathcal{L}$ are corner stratified strong submersions with respect to the CF-perturbations fixed in (i).
	\item[(iii)] Isomorphisms of admissible K-spaces  (\ref{eq:bd2}), (\ref{eq:bd5}), (\ref{eq:bd1}), (\ref{eq:bd4}) that are compatible with CF-perturbations fixed in (i) and the evaluation maps in (ii).
\end{itemize}

By Theorem \ref{theorem:pushforward}, we obtain the following (relative) de Rham chains
\begin{equation}\label{eq:c1}
x_m(k):=\sum_{\substack{a\in H_1(L;\mathbb{Z})\\ \lambda(a)<(m+1-k)\varepsilon}}(-1)^{n+1}\mathrm{Ev}_\ast\left(\overline{\mathcal{R}}_{k+1}\left(L,\bar{a};\{m\}\right)\right)\in C_{-1},
\end{equation}
\begin{equation}\label{eq:c2}
\bar{x}_m(k):=\sum_{\substack{a\in H_1(L;\mathbb{Z})\\ \lambda(a)<(m+1-k)\varepsilon}}(-1)^{k+1}\mathrm{Ev}_\ast\left(\overline{\mathcal{R}}_{k+1}\left(L,\bar{a};[m,m+1]\right)\right)\in\overline{C}_{-1},
\end{equation}
\begin{equation}\label{eq:c3}
x_{m,0}(k+1):=\sum_{\substack{a\in H_1(L;\mathbb{Z})\\ \lambda(a)<(m+1-k)\varepsilon}}(-1)^{n+1}\mathrm{Ev}_\ast\left(\overline{\mathcal{R}}_{k+2,\vartheta}\left(L,\bar{a};\{m\}\right)\right)\in C_{-2},
\end{equation}
\begin{equation}\label{eq:c4}
\bar{x}_{m,0}(k+1):=\sum_{\substack{a\in H_1(L;\mathbb{Z})\\ \lambda(a)<(m+1-k)\varepsilon}}(-1)^k\mathrm{Ev}_\ast\left(\overline{\mathcal{R}}_{k+2,\vartheta}\left(L,\bar{a};[m,m+1]\right)\right)\in\overline{C}_{-2},
\end{equation}
\begin{equation}\label{eq:c5}
y_{m,0}(k):=\sum_{\substack{a\in H_1(L;\mathbb{Z})\\ \lambda(a)<(m-U-k)\varepsilon}}\sum_{l=0}^\infty(-1)^{n+k+1}\mathrm{Ev}_\ast\left({}_{l}\overline{\mathcal{R}}_{k+1}^1\left(\beta_{l},L,\ring{\bar{a}}_{l};\{m\}\right)\right)\in C_2,
\end{equation}
\begin{equation}\label{eq:c6}
y_{m,1}(k+1):=\sum_{\substack{a\in H_1(L;\mathbb{Z})\\ \lambda(a)<(m-U-k-1)\varepsilon}}\sum_{l=1}^\infty(-1)^{n+k+1}\mathrm{Ev}_\ast\left({}_{l-1}\overline{\mathcal{R}}_{k+2}^1\left(\beta_{l},L,\ring{\bar{a}}_{l};\{m\}\right)\right)\in C_0,
\end{equation}
\begin{equation}\label{eq:c7}
\bar{y}_{m,0}(k):=\sum_{\substack{a\in H_1(L;\mathbb{Z})\\ \lambda(a)<(m-U-k)\varepsilon}}\sum_{l=0}^\infty\mathrm{Ev}_\ast\left({}_l\overline{\mathcal{R}}_{k+1}^1\left(\beta_{l},L,\ring{\bar{a}}_{l};[m,m+1]\right)\right)\in\overline{C}_2,
\end{equation}
\begin{equation}\label{eq:c8}
\bar{y}_{m,1}(k+1):=\sum_{\substack{a\in H_1(L;\mathbb{Z})\\ \lambda(a)<(m-U-k-1)\varepsilon}}\sum_{l=1}^\infty\mathrm{Ev}_\ast\left({}_{l-1}\overline{\mathcal{R}}_{k+2}^1\left(\beta_{l},L,\ring{\bar{a}}_{l};[m,m+1]\right)\right)\in\overline{C}_0,
\end{equation}
\begin{equation}\label{eq:c9}
z_{m}(k):=\sum_{\substack{a\in H_1(L;\mathbb{Z})\\ \lambda(a)<(m-1-k)\varepsilon}}(-1)^{n+k+1}\mathrm{Ev}_\ast\left(\overline{\mathcal{R}}_{k+1}^1\left(e_X,L,\bar{a};\{m\}\right)\right)\in C_1,
\end{equation}
\begin{equation}\label{eq:c10}
\bar{z}_{m}(k):=-\sum_{\substack{a\in H_1(L;\mathbb{Z})\\ \lambda(a)<(m-1-k)\varepsilon}}\mathrm{Ev}_\ast\left(\overline{\mathcal{R}}_{k+1}^1\left(e_X,L,\bar{a};[m,m+1]\right)\right)\in\overline{C}_1.
\end{equation}
Under the natural projections $C_\ast\rightarrow C_\ast^\mathrm{nd}$ and $\overline{C}_\ast\rightarrow\overline{C}_\ast^\mathrm{nd}$, we get non-degenerate chains which will still denoted by $x_m,\cdots,\bar{z}_m$ by abuse of notations. Note that in the definitions of all the de Rham chains above, we have omitted the differential form $\hat{\omega}$, which is always taken to be $1$, and the CF-perturbations, which are fixed earlier. The superscripts and subscripts for evaluation maps on various different moduli spaces are also omitted to simplify the notations.

Using the chains defined above, we can then form the following $S^1$-equivariant de Rham chains:
\begin{equation}\label{eq:t1}
\tilde{x}_m(k):=x_{m,0}(k)\otimes1\in C_{-2}^{S^1},
\end{equation}
\begin{equation}\label{eq:t2}
\bar{\tilde{x}}_m(k):=\bar{x}_{m,0}(k)\otimes1\in\overline{C}_{-2}^{S^1},
\end{equation}
\begin{equation}\label{eq:t3}
\tilde{y}_m(k,k+1):=y_{m,0}(k)\otimes h^{-d+1}+y_{m,1}(k+1)\otimes h^{-d}\in C_{2d}^{S^1},
\end{equation}
\begin{equation}\label{eq:t4}
\bar{\tilde{y}}_m(k,k+1):=\bar{y}_{m,0}(k)\otimes h^{-d+1}+\bar{y}_{m,1}(k+1)\otimes h^{-d}\in \overline{C}_{2d}^{S^1},
\end{equation}
\begin{equation}\label{eq:t5}
\tilde{z}_m(k):=z_m(k)\otimes h^{-d+1}\in C_{2d-1}^{S^1},
\end{equation}
\begin{equation}\label{eq:t6}
\bar{\tilde{z}}_m(k):=\bar{z}_m(k)\otimes h^{-d+1}\in\overline{C}_{2d-1}^{S^1}.
\end{equation}
Note that for the $y$-type chains defined above, the $k$ and $(k+1)$-components are mixed. This will become natural once the $S^1$-equivariant differentials on $C_\ast^{S^1}$ and $\overline{C}_\ast^{S^1}$ are applied to these chains, as the outcome will eventually lie in the $k$-component.

Besides these chains that will appear in the statement of Theorem \ref{theorem:finite-chain}, there are also some auxiliary chains that we need to deal with, which are defined as follows.
\begin{equation}\label{eq:c11}	\bar{y}^{j,j+1}_m(k):=\sum_{\substack{a\in H_1(L;\mathbb{Z})\\ \lambda(a)<(m-k-U)\varepsilon}}\sum_{l=2}^\infty\mathrm{Ev}_\ast\left({}_l^{j,j+1}\overline{\mathcal{R}}_{k+1}^1\left(\beta_{l},L,\ring{\bar{a}}_{l};[m,m+1]\right)\right)\in\overline{C}_1,
\end{equation}
\begin{equation}\label{eq:c12}
\bar{y}_m^{S^1}(k):=\sum_{\substack{a\in H_1(L;\mathbb{Z})\\ \lambda(a)<(m-k-U)\varepsilon}}\sum_{l=1}^\infty\mathrm{Ev}_\ast\left({}_{l-1}\overline{\mathcal{R}}_{k+1}^{S^1}\left(\beta_{l},L,\ring{\bar{a}}_{l};[m,m+1]\right)\right)\in\overline{C}_1.
\end{equation}
\begin{equation}\label{eq:c13}
\bar{y}_{m,1}^i(k):=\sum_{\substack{a\in H_1(L;\mathbb{Z})\\ \lambda(a)<(m-k-U)\varepsilon}}\sum_{l=1}^\infty\mathrm{Ev}_\ast\left({}_{l-1}\overline{\mathcal{R}}_{k+1,\tau_i}^1\left(\beta_{l},L,\ring{\bar{a}}_{l};[m,m+1]\right)\right)\in\overline{C}_1.
\end{equation}

\subsection{Completion of proof}\label{section:completion}

By our discussions in Section \ref{section:chain-s}, in order to prove Theorem \ref{theorem:deform-Viterbo}, it suffices to prove Theorem \ref{theorem:finite-chain}. 

Recall from Theorem \ref{theorem:moduli} that the codimension one boundary of the admissible K-space ${}_{l}\overline{\mathcal{R}}_{k+1}^1\left(\beta_{l},L,\ring{\bar{a}}_{l};[m,m+1]\right)$ contain the following strata:
\begin{equation}\label{eq:st1}
\bigsqcup_{0\leq j\leq l}{}_j\overline{\mathcal{M}}\left(\beta_{l},y_{j,l};[m,m+1]\right)\times{}_{l-j}\overline{\mathcal{R}}_{k+1}^1\left(y_{j,l},L,\ring{\bar{a}}_{l};[m,m+1]\right)
\end{equation}
where the $y_{j,l}$'s are 1-periodic orbits of $X_{H_t}$,
\begin{equation}\label{eq:st2}
\bigsqcup_{1\leq j\leq l-1}{}_l^{j,j+1}\overline{\mathcal{R}}_{k+1}^1\left(\beta_{l},L,\ring{\bar{a}}_{l};[m,m+1]\right),
\end{equation}
and
\begin{equation}\label{eq:st3}
{}_{l-1}\overline{\mathcal{R}}_{k+1}^{S^1}\left(\beta_{l},L,\ring{\bar{a}}_{l};[m,m+1]\right).
\end{equation}
By analyzing these strata we obtain the following.

\begin{lemma}\label{lemma:ids}
We have the following identities for (relative) de Rham chains in $\overline{C}_\ast$ for any $k,l\in\mathbb{Z}_{\geq0}$:
\begin{equation}\label{eq:id0}
\bar{x}_m(k)=\sum_{i=1}^{k+1}(-1)^{|\bar{x}_{m,0}|+k(i-1)}(\bar{\tau}_{k+1})^i_\ast\bar{x}_{m,0}(k+1)\circ_{k+2-i}\bar{e}_L,
\end{equation}
where $\bar{\tau}_{k+1}$ is the cyclic structure on the relative complex $\overline{C}_\ast(k+1)$ (cf. (\ref{eq:cyc})), and
\begin{equation}\label{eq:id1}
\sum_{l=0}^\infty\sum_{j=0}^ly_{j,l}=\sum_{l=0}^\infty\sum_{j=0}^l\delta_j(\beta_{l})=e_X,
\end{equation}
\begin{equation}\label{eq:id2}
\bar{y}_m^{j,j+1}(k)=0,\textrm{ }1\leq j\leq l-1,
\end{equation}
\begin{equation}\label{eq:id3}
\bar{y}_m^{S^1}(k)=\bar{\delta}_\mathrm{cyc}\left(\bar{y}_{m,1}(k+1)\right),
\end{equation}
where the $\delta_j$'s are structure maps on the $S^1$-complex $\mathit{SC}^\ast(X)$, with $\delta_0=\partial$ the usual Floer differential, and $\bar{\delta}_\mathrm{cyc}$ is the BV operator on $\overline{C}_\ast$.
\end{lemma}
\begin{proof}
This follows essentially from \cite{yla}, Lemmas 55 and 56. Roughly speaking, the identity (\ref{eq:id0}) holds since the forgetful map $\pi_{\vartheta,i}$ embeds $\mathcal{R}_{k+2,\vartheta}$ as an open sector in $\mathcal{R}_{k+1}$, and by varying $i$, their images cover $\mathcal{R}_{k+1}$ up to codimension $1$ strata. (\ref{eq:id2}) holds because of dimension reasons, while (\ref{eq:id3}) is true since up to codimension $1$ strata, the moduli space ${}_{l-1}\mathcal{R}_{k+1}^{S^1}$ can be decomposed as $\bigsqcup_{i=0}^k{}_{l-1}\mathcal{R}_{k+1,\tau_i}^1$ and our choice of the Floer data in Section \ref{section:CG} ensures that similar decomposition holds for the corresponding moduli space of Floer solutions. The reason for the validity of (\ref{eq:id1}) is slightly different in our case, since $\tilde{\beta}=\sum_{l=0}^\infty\beta_l\otimes u^{-d-l+1}$ is no longer a cyclic dilation, therefore not necessarily an $S^1$-equivariant cocycle. However, we can apply our observation made in (\ref{eq:cl}), which gives $\sum_{l=1}^\infty\sum_{j=0}^{l-1}\delta_j(\beta_l)=0$.
\end{proof}

\begin{proof}[Proof of Theorem \ref{theorem:finite-chain}]
Recall that we have defined the $S^1$-equivariant chains $\tilde{x}_m$, $\bar{\tilde{x}}_m$, $\tilde{y}_m$, $\bar{\tilde{y}}_m$, $\tilde{z}_m$ and $\bar{\tilde{z}}_m$ in (\ref{eq:t1})---(\ref{eq:t6}). We need to show that under the natural projections $C_\ast^{S^1}\twoheadrightarrow C_\ast^\lambda$ and $\overline{C}_\ast^{S^1}\twoheadrightarrow\overline{C}_\ast^\lambda$, the images of these chains satisfy the conditions (i)--(vi) of Theorem \ref{theorem:finite-chain}.

By definitions and (\ref{eq:epm}), we have the relations $\tilde{x}_m=\tilde{e}_-(\bar{\tilde{x}}_m)$, $\tilde{y}_m=\tilde{e}_-(\bar{\tilde{y}}_m)$, and $\tilde{z}_m=\tilde{e}_-(\bar{\tilde{z}}_m)$. After projecting to $C_\ast^\lambda$ and $\overline{C}_\ast^\lambda$, we have verified (ii).

The fact that $\tilde{x}_{m+1}-\tilde{e}_+(\bar{\tilde{x}}_m)\in\mathcal{F}^mC_{-2}^{S^1}$ follows from
\begin{equation}
\left(\tilde{x}_{m+1}-\tilde{e}_+(\bar{\tilde{x}}_m)\right)(a,k)\neq0\Rightarrow\lambda(a)\geq(m-k)\varepsilon, \nonumber
\end{equation}
which is a direct consequence of the definition of the chains $\tilde{x}_{m+1}$ and $\bar{\tilde{x}}_m$. After taking the projections, we obtain $\underline{x}_{m+1}-\underline{e}_+(\bar{\underline{x}}_m)\in\mathcal{F}^mC_{-2}^\lambda$. The verifications of $\underline{y}_{m+1}-\underline{e}_+(\bar{\underline{y}}_m)\in\mathcal{F}^{m-U-1}C_{2d}^\lambda$ and $\underline{z}_{m+1}-\underline{e}_+(\bar{\underline{z}}_m)\in\mathcal{F}^{m-2}C_{2d-1}^\lambda$ are similar. This confirms (iv).

It follows from the isomorphism (\ref{eq:bd4}) and the fiber product formula stated in Theorem \ref{theorem:pushforward} that
\begin{equation}\label{eq:dx1}
\bar{\partial}\bar{x}_{m,0}(a,k+1)=\sum_{\substack{k_1+k_2=k+1\\a_1+a_2=a\\1\leq i\leq k_1}}(-1)^{(k_1-i)(k_2-1)+k_1}\bar{x}_{m,0}(a_1,k_1+1)\circ_i\bar{x}_m(a_2,k_2).
\end{equation}
Combining with (\ref{eq:id0}) of Lemma \ref{lemma:ids}, we further deduce
\begin{equation}\label{eq:dx2}
\begin{split}
&\;\;\;\;\;\;\bar{\partial}\bar{x}_{m,0}(a,k+1) \\
&=\sum_{\substack{k_1+k_2=k+1\\a_1+a_2=a\\1\leq i\leq k_1}}\sum_{j=1}^{k_2+1}(-1)^{\maltese_{ij}^1}\bar{x}_{m,0}(a_1,k_1+1)\circ_i\left(\overline{\tau}_{k_2+1}^j\left(\bar{x}_{m,0}(a_2,k_2+1)\right)\circ_{k_2+2-j}e_L\right),
\end{split}
\end{equation}
where
\begin{equation}
\maltese_{ij}^1=(i-1)(k_2-1)+(k_1-1)k_2+k_2(j-1) \nonumber
\end{equation}
for every $(a,k)$ with $\lambda(a)<(m+k-1)\varepsilon$. This verifies the requirement $\bar{\partial}^\mathrm{tot}(\bar{\underline{x}}_m)-\frac{1}{2}\left\{\bar{\underline{x}}_m,\bar{\underline{x}}_m\right\}\in\overline{\mathcal{F}}^m\overline{C}_{-3}^\lambda$ in (iii). Note that the right-hand side of (\ref{eq:dx2}) is different from $\frac{1}{2}\left\{\bar{\tilde{x}}_m,\bar{\tilde{x}}_m\right\}(a,k+1)$ in $\overline{C}^{S^1}_\ast$. However, after passing to the quotient $\overline{C}_\ast^\lambda$, it becomes $\frac{1}{2}\left\{\bar{\underline{x}}_m,\bar{\underline{x}}_m\right\}(a,k+1)$. Similarly one can show that $\bar{\partial}^\mathrm{tot}(\bar{\underline{z}}_m)-\left\{\bar{\underline{x}}_m,\bar{\underline{z}}_m\right\}\in \overline{\mathcal{F}}^{m-2}\overline{C}_{2d-2}^\lambda$. The verification of $\bar{\partial}^\mathrm{tot}(\bar{\underline{y}}_m)-\left\{\bar{\underline{x}}_m,\bar{\underline{y}}_m\right\}-\bar{\underline{z}}_m\in \overline{\mathcal{F}}^{m-U-1}\overline{C}_{2d-1}^\lambda$ is slightly more complicated. First note that the fiber product formula combined with the isomorphism (\ref{eq:bd5}) gives
\begin{equation}\label{eq:ana}
\begin{split}
\bar{\partial}{\bar{y}}_{m,0}(a,k)&=\sum_{\substack{k_1+k_2=k+1\\a_1+a_2=a\\1\leq i\leq k_1}}(-1)^{(k_1-i)(k_2-1)+k_1-1}\bar{y}_{m,0}(a_1,k_1)\circ_i\bar{x}_m(a_2,k_2) \\
&+\sum_{\substack{k_1+k_2=k+1\\a_1+a_2=a\\1\leq i\leq k_1}}(-1)^{(k_1-i)(k_2-1)+k_1}\bar{x}_m(a_1,k_1)\circ_i\bar{y}_{m,0}(a_2,k_2) \\
&-\sum_{l=0}^\infty\sum_{j=0}^l\mathrm{Ev}_\ast\left({}_{l-j}\overline{\mathcal{R}}_{k+1}^1\left(\delta_j(\beta_{l}),L,\ring{\bar{a}}_{l};[m,m+1]\right)\right) \\
&-\bar{y}_m^{j,j+1}(a,k)-\bar{y}_m^{S^1}(a,k).
\end{split}
\end{equation}
Applying the identity (\ref{eq:id1}), we obtain
\begin{equation}\label{eq:fun}
\begin{split}
&\textrm{ }\textrm{ }\textrm{ }\sum_{l=0}^\infty\sum_{j=0}^l\mathrm{Ev}_\ast\left({}_{l-j}\overline{\mathcal{R}}_{k+1}^1\left(\delta_j(\beta_{l}),L,\ring{\bar{a}}_{l};[m,m+1]\right)\right) \\
&=\mathrm{Ev}_\ast\left(\overline{\mathcal{R}}_{k+1}^1(e_X,L,\bar{a};[m,m+1])\right).
\end{split}
\end{equation}
Since the right-hand side of (\ref{eq:fun}) is by definition the de Rham chain $\bar{z}_m(a,k)\in\overline{C}_1$, using the identities (\ref{eq:id2}) and (\ref{eq:id3}) we can write
\begin{equation}\label{eq:ana1}
\begin{split}
&\textrm{ }\textrm{ }\textrm{ }\textrm{ }\textrm{ }\textrm{ }\bar{\partial}{\bar{y}}_{m,0}(a,k)\otimes h^{-d+1} \\
&=\sum_{\substack{k_1+k_2=k+1\\a_1+a_2=a\\1\leq i\leq k_1}}(-1)^{(k_1-i)(k_2-1)+k_1-1}\bar{y}_{m,0}(a_1,k_1)\otimes h^{-d+1}\circ_i\bar{x}_m(a_2,k_2) \\
&+\sum_{\substack{k_1+k_2=k+1\\a_1+a_2=a\\1\leq i\leq k_1}}(-1)^{(k_1-i)(k_2-1)+k_1}\bar{x}_m(a_1,k_1)\circ_i\bar{y}_{m,0}(a_2,k_2)\otimes h^{-d+1} \\
&+\bar{z}_m(a,k)\otimes h^{-d+1}-\bar{\delta}_\mathrm{cyc}\left(\bar{y}_{m,1}(a,k+1)\right)\otimes h^{-d+1}.
\end{split}
\end{equation}
Note that $\bar{\partial}{\bar{y}}_{m,0}(a,k)\otimes h^{-d+1}+\bar{\delta}_\mathrm{cyc}\left(\bar{y}_{m,1}(a,k+1)\right)\otimes h^{-d+1}$ gives the $(a,k)$-part of $\bar{\partial}^{S^1}(\bar{\tilde{y}}_m)(a,k)$. By (\ref{eq:bracket-pr}) (or more precisely, its relative version) and (\ref{eq:id0}), we can write (\ref{eq:ana1}) as
\begin{equation}
\bar{\partial}^{S^1}(\bar{\tilde{y}}_m)(a,k)-\left\{\bar{\tilde{x}}_m,\bar{\tilde{y}}_m\right\}(a,k)-\bar{\tilde{z}}_m(a,k)=0. \nonumber
\end{equation}
In other words, $\bar{\partial}^{S^1}(\bar{\tilde{y}}_m)-\left\{\bar{\tilde{x}}_m,\bar{\tilde{y}}_m\right\}-\bar{\tilde{z}}_m\in\overline{\mathcal{F}}^{m-U-1}\overline{C}_{2d-1}^\lambda$. This completes the verification of (iii).

To see that (v) and (vi) are true, note that $\tilde{x}_{m,0}(a,k+1)\neq0$ implies $\mathcal{R}_{k+2,\vartheta}(L,\bar{a};\{m\})\neq\emptyset$, thus $\lambda(a)\geq2\varepsilon$ or $a=0$, $k\geq2$. Moreover, it follows from (\ref{eq:id0}) that
\begin{equation}
B\left(\tilde{x}_{m,0}(0,3)\right)=\sum_{j=1}^{k+1}(-1)^{j-1}\tau_3^j\left(x_{m,0}(0,3)\right)\circ_{4-j}e_L=x_m(0,2). \nonumber
\end{equation}
Since the chain $x_m(0,2)$ is defined using the (oriented) moduli space $(-1)^{n+1}\mathcal{R}_3(L,0;\{m\})$, which counts constant maps to $L$, we obtain $[x_m(0,2)]=(-1)^{n+1}[L]$. Similarly, $\tilde{z}_m(a,k)\neq0$ implies that $\mathcal{R}_{k+1}^1(e_X,L,\bar{a};\{m\})\neq\emptyset$, thus $\lambda(a)\geq2\varepsilon$ or $a=0$, and
\begin{equation}
\left[\tilde{z}_m(0,0)\right]=(-1)^{n+1}\left[\overline{\mathcal{R}}_1^1(e_X,L,0;\{m\})\right]\otimes h^{-d+1}=(-1)^{n+1}[\![L]\!]\otimes h^{-d+1}. \nonumber
\end{equation}
(vii) follows directly from Lemma \ref{lemma:action}, and the fact that the $L_\infty$-homomorphism (\ref{eq:p}) preserves the action.

Finally, we confirm (viii). Note that $\left(\bar{\tilde{y}}_m(a,k),\bar{\tilde{z}}_m(a,k)\right)\neq(0,0)$ implies that there exists some $l\in\mathbb{Z}_{\geq0}$ such that one of the two moduli spaces
\begin{equation}
{}_l\overline{\mathcal{R}}_{k+1}^1(\beta_{l},L,\ring{\bar{a}}_{l})\textrm{ and }{}_{l-1}\overline{\mathcal{R}}_{k+2}^1(\beta_{l},L,\ring{\bar{a}}_{l}) \nonumber
\end{equation}
is non-empty. Note that when $l=0$, $\overline{\mathcal{R}}_{k+1}^1(\beta_0,L,\ring{\bar{a}}_0)\neq\emptyset$. By Gromov compactness,
\begin{equation}
{}_l\overline{\mathcal{R}}_{k+1}^1(x,L,\ring{\bar{a}})=\emptyset\Leftrightarrow{}_l\overline{\mathcal{R}}_1^1(x,L,\ring{\bar{a}})=\emptyset. \nonumber
\end{equation}
It follows that for such an $l$, we have
\begin{equation}\label{eq:l}
\left({}_{l}\overline{\mathcal{R}}_1^1(\beta_{l},L,\ring{\bar{a}}_{l}),{}_{l-1}\overline{\mathcal{R}}_1^1(\beta_{l},L,\ring{\bar{a}}_{l})\right)\neq(\emptyset,\emptyset).
\end{equation}
Hence, $a\in A_x$ implies that $\overline{\mathcal{R}}_{2,\vartheta}(L,\bar{a};\{m\})\neq\emptyset$, while $a\in A_{y,z}$ implies there is an $l$ such that (\ref{eq:l}) holds. From now on, fix the choice of such an $l$. We claim that the set
\begin{equation}\label{eq:s}
\left\{a\in H_1(L;\mathbb{Z})\left\vert\lambda(a)<\Xi,\left(\overline{\mathcal{R}}_{2,\vartheta}(L,\bar{a};\{m\}),{}_{l}\overline{\mathcal{R}}_1^1(\beta_{l},L,\ring{\bar{a}}),{}_{l-1}\overline{\mathcal{R}}_1^1(\beta_{l},L,\ring{\bar{a}})\right)\neq(\emptyset,\emptyset,\emptyset)\right.\right\}
\end{equation}
is finite for any $\Xi>0$. In fact, if this is not the case, then there exists a sequence $(a_\ell)_{\ell\in\mathbb{N}}$ of distinct elements $a_\ell\in H_1(L;\mathbb{Z})$ such that $\lambda(a_\ell)<\Xi+|A_{H_t}(\beta_l)|$ for every $\ell\in\mathbb{N}$, and at least one of the following three conditions holds:
\begin{itemize}
	\item $\overline{\mathcal{R}}_{2,\vartheta}(L,\bar{a}_\ell;\{m\})\neq\emptyset$ for every $\ell\in\mathbb{N}$,
	\item ${}_{l}\overline{\mathcal{R}}_1^1(\beta_{l},L,\ring{\bar{a}}_\ell)\neq\emptyset$ for every $\ell\in\mathbb{N}$,
	\item ${}_{l-1}\overline{\mathcal{R}}_1^1(\beta_{l},L,\ring{\bar{a}}_\ell)\neq\emptyset$ for every $\ell\in\mathbb{N}$.
\end{itemize}
Consider the first case. Pick an element $u_\ell\in\overline{\mathcal{R}}_{2,\vartheta}(L,\bar{a}_\ell;\{m\})$ for each $\ell$. By possibly passing to a subsequence one may assume that $(u_\ell)_{\ell\in\mathbb{N}}$ is Gromov convergent, so $(a_\ell)_{\ell\in\mathbb{N}}$ is constant for $\ell\gg0$, which contradicts the assumption. In the second and the third cases, similar arguments will lead to the same contradiction. Thus we conclude that (\ref{eq:s}) is a finite set, which shows that both $A_x^+(\Xi)$ and $A_{y,z}^+(\Xi)$ are finite sets for every $\Xi>0$. 
\end{proof}

\section{Extremal Lagrangians in spectrally convex Liouville domains}\label{section:extremal}

In this section, we prove Theorem \ref{theorem:boundaryrigidity}, which asserts that extremal aspherical Lagrangians in the sense of Definition \ref{definition:Asp-extremal} lie in the boundary $\partial X$ provided that the Gutt--Hutchings capacities of $X$ decay fast enough in the sense of Definition \ref{definition:asymp-sta}.

\begin{proof}[Proof of Theorem \ref{theorem:boundaryrigidity}]
Let $(X,\lambda)$ be a $2n$-dimensional Liouville domain with $c_1(X)=0$. By our assumption, there exists a sequence $\{d_i\}_{i\in\mathbb{N}}$ of positive integers such that
\[
\lim_{d_i\to\infty}\frac{C^{\operatorname{GH}}_{d_i}(X)}{d_i}=\inf_{d\in\mathbb{N}}\frac{C^{\operatorname{GH}}_d(X)}{d},
\]
and that the convergence occurs at a rate faster than $\frac1{d_i}$. Suppose, for contradiction, that there exists a closed oriented aspherical Lagrangian submanifold $L\subset X$ which is $\mathit{Spin}$, with
\[
A_{\min}(L)=\inf_{d\in\mathbb{N}}\frac{C^{\operatorname{GH}}_d(X)}{d},
\]
and has non-trivial intersections with the interior $\mathrm{int}(X)$ of $X$. We will derive a contradiction using the monotonicity property for pseudoholomorphic curves with Lagrangian boundary conditions. 
	
Pick a decreasing sequence of positive numbers $\{\eta_i\}_{i\in\mathbb{N}}$ and consider the extensions of $X$ by attaching the collars
\[
\left([0,\eta_i]\times \partial X,d(e^r\lambda)\right),
\]
which we denote by
\[
X_{+{\eta_i}}:=X\cup_{\partial X}\left([0,\eta_i]\times \partial X\right).
\]
We can regard $L$ as a Lagrangian submanifold in $\mathrm{int}(X_{+\eta_i})$. By the assumption on the symplectic area spectrum of $L$ in $(X,\lambda)$, every smooth disc
\[
u:(D,\partial D)\to (X_{+\eta_i},L)
\]
satisfies
\[
\int_{\partial D}u^*\lambda>0\quad\Longrightarrow\quad\int_{\partial D}u^*\lambda\ge\inf_{d\in\mathbb{N}}\frac{C^{\operatorname{GH}}_d(X)}{d}.
\]
Here, we emphasize that the symplectic area of a disc $u:(D,\partial D)\to (X_{+\eta_i},L)$ depends only on its relative homotopy class $[u]\in \pi_2(X_{+\eta_i},L)$. Moreover, since every such disc $u:(D,\partial D)\to (X_{+\eta_i},L)$ can be smoothly isotoped to a smooth disc $u:(D,\partial D)\to (X,L)$ with boundary on $L$, we have an identification $\pi_2(X_{+\eta_i},L)\cong\pi_2(X,L)$ between relative homotopy groups. By assumption, the Lagrangian $L$ intersects $\mathrm{int}(X)$ (hence also $\mathrm{int}(X_{\eta_i})$ for each $i\in\mathbb{N}$) non-trivially. Fix a point $p\in L\cap \operatorname{int}(X)$. Let  $B^{2n}(\pi\delta^2)$ be the closed ball of radius $\delta>0$\footnote{Recall our convention in (\ref{eq:ball}).} centered at the origin in $(\mathbb{R}^{2n},\omega_\mathrm{std})$. For sufficiently small $\delta>0$, choose a symplectic embedding $\phi:B^{2n}(\pi\delta^2)\to \operatorname{int}(X)$ such that
\[\phi(0)=p,\ \phi^{-1}(L)=B^{2n}(\pi\delta^2)\cap \mathbb{R}^n.\]
	
From now on, we choose $d_i\in\mathbb{N}$ sufficiently large and $\eta_i<\frac{1}{d_i^2}$ sufficiently small so that
\begin{equation}\label{eq:choice}
d_i\left(\frac{C^{\operatorname{GH}}_{d_i}(X)}{d_i}-\inf_{d\in\mathbb{N}}\frac{C^{\operatorname{GH}}_d(X)}{d}\right)+o(\eta_i) C^{\operatorname{GH}}_{d_i}(X)<\min\left\{\frac{1}{3}\pi \delta^2,\inf_{d\in\mathbb{N}}\frac{C^{\operatorname{GH}}_d(X)}{d}\right\}.
\end{equation}
The following monotonicity lemma is crucial for our argument.
	
\begin{lemma}[\cite{cel}, Lemma 3.6]\label{monotoncity-with-lag-boundary}
Let $(M,\omega)$ be a symplectic manifold and let $K\subset M$ be a closed Lagrangian submanifold. Fix a point $x\in M$, which possibly lies in $K$, and an open neighbourhood $U_x$ of $x$ in $M$. Let $J_0$ be a fixed $\omega$-compatible almost complex structure on $U_x$. Let $J:\mathit{TM}\rightarrow\mathit{TM}$ be an $\omega$-compatible almost complex structure on $M$, such that $J|_{\textit{TU}_x}=J_0$. Then for any compact connected Riemann surface with boundary $(\Sigma,j)$, there exists a constant $\hbar>0$, depending on $\left(M,\omega,J_0,x,U_x\right)$, with the following property. If $u:(\Sigma,\partial \Sigma)\rightarrow (M,K)$ is a non-constant continuous map which is $(j,J)$-holomorphic, passes through $x$, and such that $u|_{u^{-1}(U_x)}:u^{-1}(U_x)\longrightarrow U_x$ is proper, we have
\[
\int_S u^\ast\omega\geq \hbar>0.
\]
\end{lemma}
	
We will apply Lemma \ref{monotoncity-with-lag-boundary} to the tuple
\begin{equation}
\left(M,K,\omega,J_0,x,U_x\right)=
\left(X_{+\eta_i},L,d\lambda,\phi_\ast J_{\mathrm{std}},p,\phi\left(B^{2n}(\pi\delta^2)\right)\right), \nonumber
\end{equation}
where $J_\mathrm{std}$ is the standard complex structure on $\mathbb{C}^n$. Note that in this case, because $U_x$ is a standard symplectic ball at $p\in L\cap\mathrm{int}(X)$, we can actually take $\hbar=\frac{\pi\delta^2}{2}$.
	
Consider the domain-dependent contact type $d\lambda$-compatible almost complex structure $J_S:S\times\mathit{TX}_{+\eta_i}\rightarrow\mathit{TX}_{+\eta_i}$ as part of the Floer datum fixed in Definition \ref{definition:data}, where in our case $S=[0,\infty)\times S^1$ and $J_S|_{[0,\kappa)\times S^1}=J$ for some fixed almost complex structure $J:\mathit{TX}_{+\eta_i}\rightarrow\mathit{TX}_{+\eta_i}$ and $\kappa>0$ (possibly varies over the Cohen--Ganatra moduli space) so that $J|_{\phi\left(B^{2n}(\pi\delta^2)\right)}=\phi_\ast J_{\mathrm{std}}$. We repeat the argument used in the proof of Theorem \ref{theorem:GH-ALbound}, under the assumption that $C^{\operatorname{GH}}_{d_i}(X)<\infty$, until we obtain non-trivial classes $a_1,\dots,a_m\in H_1(L;\mathbb{Z})$, $m\geq d_i$, such that $x(a_j)\neq 0$ for all $j=1,\dots,m$. Moreover,
\[
y(-a)\neq 0,\textrm{ where }a:=\sum_{j=1}^m a_j,
\]
and we have
\begin{equation}\label{grading-y}
2m\geq \sum_{j=1}^m \mu(a_j)=\mu(a)\geq 2d_i. \nonumber
\end{equation}
	
By the construction of the class $x\in\widehat{\mathbb{H}}_{-2}^{S^1}$ in the proof of Theorem \ref{theorem:deform-Viterbo}, which corresponds to the chain $\underline{x}\in\widehat{C}_{-2}^\lambda$ in Theorem \ref{theorem:chain}, together with the fact that $x(a_j)\neq 0$ for each $j=1,\dots,m$, there exist non-constant $J$-holomorphic discs $u_1,\dots,u_m:(D,\partial D)\rightarrow(X_{+\eta_i},L)$
such that
\[
\left[u_j(\partial D)\right]=a_j\in H_1(L;\mathbb{Z})
\]
for $j=1,\dots,m$ and the almost complex structure $J$ on $X_{+\eta_i}$ fixed above. Similarly, by the construction of $y\in\widehat{\mathbb{H}}_{2d}^{S^1}$, which corresponds to the chain $\underline{y}\in\widehat{C}_{2d}^\lambda$ in Theorem \ref{theorem:chain}, together with the fact that $y(-a)\neq 0$, there exists a smooth map
\[
v:(S,\partial S)\rightarrow (X_{+\eta_i},L)
\]
satisfying
\begin{equation}
\left\{
\begin{array}{l}
(dv-X_{H_S}\otimes\gamma_S)^{0,1}=0,\\
v\left(\partial S\right)\subset L, \\
\left[v\left(\partial S\right)\right]=-a\in H_1(L;\mathbb{Z}), \\
\displaystyle \lim_{s\to\infty} v(s,\cdot)=\beta,
\end{array}
\right.
\end{equation}
where $\gamma_S\in\Omega^1(S)$ is a sub-closed $1$-form which equals $dt$ on $[R,\infty)\times S^1$ for $R\gg1$ and vanishes near $\{0\}\times S^1$, $\beta$ is a $1$-periodic orbit of the Hamiltonian vector field $X_{H_S}$, where $H_S:S\times X_{+\eta_i}\rightarrow\mathbb{R}$ is a domain-dependent Hamiltonian which vanishes on $[0,\kappa)\times S^1$ and equals $X_{H_t}$ on $[R,\infty)\times S^1$ for some time-dependent Hamiltonian $H_t:S^1\times X_{+\eta_i}\rightarrow\mathbb{R}$. Here the $(0,1)$-part in the Floer equation is taken with respect to the domain-dependent almost complex structure $J_S$, which agrees on $[0,\kappa)\times S^1$ with the fixed almost complex structure $J$ on $X_{+\eta_i}$. In particular, it follows that $v$ is genuinely $J$-holomorphic on $[0,\kappa)\times S^1\subset S$ for some fixed $\kappa>0$.

It follows from \eqref{eq:lm} and Theorem \ref{theorem:deform-Viterbo}, (ii) and Lemma \ref{lemma:action} that
\[
E_\mathrm{top}(v)+\sum_{j=1}^m \int_{\partial D}u_j^\ast\lambda\leq C^{\operatorname{GH}}_{d_i}(X_{+\eta_i}),
\]
where
\[
E_\mathrm{top}(v):=\int_Sv^\ast d\lambda-d(v^\ast H_S\cdot\gamma_S)
\]
is the topological energy of the map $v$.

We claim that the number of $J$-holomorphic discs with boundary on $L$ satisfies $m=d_i$. Suppose, to the contrary, that $m\geq d_i+1$. By our assumption on the minimal symplectic area $A_{\min}(L)$, we have
\[
\int_{\partial D}u_j^\ast\lambda
\geq
\inf_{d\in\mathbb{N}}\frac{C^{\operatorname{GH}}_d(X)}{d}
\]
for each $j=1,\dots,m$. Combining this estimate with the preceding energy estimate and \eqref{eq:choice}, we obtain
\[
\begin{aligned}
0&\leq E_\mathrm{top}(v) \\
&\leq d_i\left(\frac{C^{\operatorname{GH}}_{d_i}(X)}{d_i}-\inf_{d\in\mathbb{N}}\frac{C^{\operatorname{GH}}_d(X)}{d}\right)+o(\eta_i) C^{\operatorname{GH}}_{d_i}(X)-\inf_{d\in\mathbb{N}}\frac{C^{\operatorname{GH}}_d(X)}{d} \\
&<0,
\end{aligned}
\]
which is impossible. Hence $m=d_i$ and we have
\begin{equation}\label{eq:areaest}
0\leq E_\mathrm{top}(v)<\min\left\{\frac{1}{3}\pi\delta^2,\inf_{d\in\mathbb{N}}\frac{C^{\operatorname{GH}}_d(X)}{d}\right\}.
\end{equation}

By (\ref{eq:c5}) and (\ref{eq:t3}), the map $v:(S,\partial S)\rightarrow(X_{+\eta_i},L)$ (together with parametrizations of the domains) gives rise to an element of the Cohen--Ganatra moduli space ${}_l\mathcal{R}_{k+1}^1(\beta,L,-\ring{\bar{a}})$ for some $k,l\in\mathbb{Z}_{\geq0}$ and $\beta\in\mathit{SC}^{2l-1}(X_{+\eta_i})$, where $\ring{\bar{a}}\in\pi_2(X_{+\eta_i},\beta,L)$ satisfies $\partial\ring{\bar{a}}=a\in H_1(L;\mathbb{Z})$. Note that the area bound (\ref{eq:areaest}) rules out the disc bubbling for elements in ${}_l\mathcal{R}_{k+1}^1(\beta,L,-\ring{\bar{a}})$ since any non-constant $J$-holomorphic disc bubble with boundary on $L$ has symplectic area at least $A_{\min}(L)=\inf_{d\in\mathbb{N}}\frac{C^{\operatorname{GH}}_d(X)}{d}$, whereas the total symplectic area of the curves in this moduli space is strictly smaller than this quantity by (\ref{eq:areaest}). By (\ref{eq:dim}) we have
\begin{equation}\label{eq:dim1}
\dim{}_l\mathcal{R}_{k+1}^1(\beta,L,-\ring{\bar{a}})=n+1+k-\mu(a).
\end{equation}
Here we pick $k\geq\mu(a)$ so that $\dim{}_l\mathcal{R}_{k+1}^1(\beta,L,-\ring{\bar{a}})\geq n+1$.

Let ${}_l\overline{\mathcal{R}}_{k+1}^1(\beta,L,-\ring{\bar{a}})$ be the compactification of ${}_l\mathcal{R}_{k+1}^1(\beta,L,-\ring{\bar{a}})$ obtained by adding broken configurations of Floer cylinders at the unique puncture $\zeta$. With the absence of disc and sphere bubbles, ${}_l\overline{\mathcal{R}}_{k+1}^1(\beta,L,-\ring{\bar{a}})$ is an oriented compact smooth manifold with corners for generic choices of Floer data, specifically the time-dependent Hamiltonian $H_t$ and the almost complex structure $J_t$ on $[R,\infty)\times S^1\subset S$ with $R\gg1$. Consider the evaluation map
\[
\operatorname{ev}_0:{}_l\overline{\mathcal{R}}_{k+1}^1(\beta,L,-\ring{\bar{a}})\rightarrow L
\]
at the boundary marked point $z_0\in\partial S$, which is given by (\ref{eq:eva2}) composed with the projection to the factor corresponding to $z_0$. This map is a corner-stratified smooth submersion with appropriate choices of Floer data. It follows that $\mathrm{ev}_0^{-1}(p)\neq\emptyset$ for any $p\in L$, and $\mathrm{ev}_0^{-1}(p)$ is in fact a compact manifold with corners of dimension $\geq1$ by (\ref{eq:dim1}), with the boundary strata corresponding to semi-stable breakings of (parametrized) Floer cylinders at $\zeta$.
	
In particular, let $p\in L\cap \operatorname{int}(X)$ be the point chosen at the beginning of the proof. There exists a smooth map
\[
v':(S,\partial S)\rightarrow(X_{+\eta_i},L)
\]
arising as an element of $\mathrm{ev}_0^{-1}(p)\cap{}_l\mathcal{R}_{k+1}^1(\beta,L,-\ring{\bar{a}})$, which is asymptotic to the Hamiltonian orbit $\beta$ at its unique puncture and satisfies $v'(z_0)=p$. Moreover, $v'$ is $J$-holomorphic on the neighborhood $[0,\kappa)\times S^1$ for some fixed $\kappa>0$, with $J|_{\phi\left(B^{2n}(\pi\delta^2)\right)}=\phi_\ast J_{\mathrm{std}}$. It follows that if we choose $\delta>0$ to be sufficiently small, then
\[
\phi^{-1}\circ v'\big|_{(v')^{-1}\left(\phi\left(B^{2n}(\pi\delta^2)\right)\right)}:(v')^{-1}\left(\phi\left(B^{2n}(\pi\delta^2)\right)\right)\longrightarrow B^{2n}(\pi\delta^2)
\]
is a proper $J_{\mathrm{std}}$-holomorphic map passing through the origin. Since $v'$ lies in the same moduli space as $v$, our previous estimate holds for $v'$ as well. By Lemma \ref{monotoncity-with-lag-boundary} and \eqref{eq:areaest}, we obtain
\begin{equation}
\begin{split}
\frac{1}{2}\pi\delta^2&\leq\int_{(v')^{-1}\left(\phi\left(B^{2n}(\pi\delta^2)\right)\right)}(\phi^{-1}\circ v')^\ast\omega_{\mathrm{std}} \\
&=\int_{(v')^{-1}\left(\phi\left(B^{2n}(\pi\delta^2)\right)\right)}(v')^\ast d\lambda \\
&\leq E_\mathrm{top}(v)<\frac{1}{3}\pi\delta^2, \nonumber
\end{split}
\end{equation}
which is a contradiction. This completes the proof.
\end{proof}

\section{Examples}\label{section:example}

In this section we discuss the examples (beyond star-shaped domains in $\mathbb{C}^n$) where the results in this paper are applicable. In Section \ref{section:dilation}, we explore the relations between the finiteness of Gutt--Hutchings capacities and the existence of higher dilations introduced by Zhao \cite{jz}, which provides examples of Liouville domains for which Corollary \ref{Lagrangian-width} and Theorem \ref{theorem:deform-Viterbo} hold. In Section \ref{section:capacity}, we explicitly compute the Lagrangian capacities for some Liouville domains admitting higher dilations.

\subsection{Higher dilations and Gutt--Hutchings capacities}\label{section:dilation}

Let $X$ be a Liouville domain with $c_1(X)=0$. Its \textit{completed periodic symplectic symplectic cohomology}, denoted $\widehat{\mathit{PSH}}^\ast(X)$, is the cohomology of the complex
\begin{equation}\label{eq:PSC}
\left(\mathit{SC}^\ast(X)\otimes_\mathbb{R}\mathbb{R}(\!(u)\!),\partial^{S^1}:=\partial+u\delta_1+u^2\delta_2+\cdots\right), \nonumber
\end{equation}
where $u$ is a formal variable of degree $2$. Just like in the case of ordinary symplectic cohomology, there is a version of the PSS map
\begin{equation}\label{eq:PSS}
\phi_X:H^\ast(\widehat{X};\mathbb{R})_{\otimes_\mathbb{R}}\mathbb{R}(\!(u)\!)\rightarrow\widehat{\mathit{PSH}}^\ast(X),
\end{equation}
where $\widehat{X}$ is the Liouville manifold obtained by completing $X$ with the cylindrical end $[1,\infty)\times\partial X$.

The following notion is introduced by J. Zhao in her thesis.

\begin{definition}[\cite{jz}, Definition 4.2.1]
We say that $X$ admits a higher dilation if the unit $1\in H^\ast(\widehat{X};\mathbb{R})$ lies in the kernel of the PSS map (\ref{eq:PSS}).
\end{definition}

The existence of a higher dilation is related to the finiteness of the Gutt--Hutchings capacities $C_d^\mathrm{GH}(X)$, $d\in\mathbb{N}$ in the following sense.

\begin{proposition}\label{proposition:finite}
The Liouville domain $X$ admits a higher dilation if and only if $C_d^\mathrm{GH}(X)<\infty$ for all $d\in\mathbb{N}$.
\end{proposition}
\begin{proof}
It is proved in \cite{jz}, Lemma 4.2.5 that $X$ admits a higher dilation if and only if the classes $1\otimes u^{-d+1}\in H_{S^1}^\ast(\widehat{X};\mathbb{R})=H^\ast(X;\mathbb{R})\otimes_\mathbb{R}\mathbb{R}(\!(u)\!)/u\mathbb{R}[\![u]\!]$, $d\in\mathbb{N}$ lie in the kernel of the $S^1$-equivariant PSS map
\begin{equation}
H_{S^1}^\ast(\widehat{X};\mathbb{R})\rightarrow\mathit{SH}_{S^1}^\ast(X). \nonumber
\end{equation}
In other words, there exists a finite sequence of chains $y_d,y_{d+1},\cdots,y_N\in\mathit{SC}^\ast(X)$ such that
\begin{equation}
\partial^{S^1}\left(\sum_{i=d}^Ny_i\otimes u^{-i}\right)=e_X\otimes u^{-d+1} \nonumber
\end{equation}
for any $d\in\mathbb{N}$. This is equivalent to $C_d^\mathrm{GH}(X)<\infty$.
\end{proof}

\begin{remark}
There are many variations of the notion of a ``dilation" introduced by Seidel-Solomon \cite{sss}. For example, the existence of a cyclic dilation in the sense of \cite{yle} (with the marking map hitting the identity) is equivalent to $C_1^\mathrm{GH}(X)<\infty$, which is in turn equivalent to the existence of a $k$-dilation for some $k\in\mathbb{N}$ in the sense of Zhou \cite{zzs}.
\end{remark}

It is clear that every Liouville domain $X$ with $\mathit{SH}^\ast(X)=0$ admits a higher dilation. This is the case of star-shaped domains in $\mathbb{C}^n$, where most of the studies of symplectic capacities occur. Nevertheless, there are more interesting examples of Liouville domains admitting higher dilations. Let $X=D^\ast Q$ be the unit disc cotangent bundle of some closed manifold $Q$, it is proved in \cite{jz}, Corollary 5.2.5 that $X$ admits a higher dilation if and only if $Q$ is \textit{rationally inessential}, i.e. the fundamental class $[Q]\in H_n(Q;\mathbb{R})$ vanishes under the map $H_n(Q;\mathbb{R})\rightarrow H_n\left(B\pi_1(Q);\mathbb{R}\right)$ induced by the classifying map. In particular, any disc cotangent bundle of a simply-connected closed manifold admits a higher-dilation, therefore having $C_d^\mathrm{GH}(X)<\infty$ for any $d\in\mathbb{N}$ by Proposition \ref{proposition:finite}. 

In another direction, Zhao proved in \cite{jz}, Proposition 5.3.2 that if the Liouville manifold $\widehat{X}$ admits a symplectic Lefschetz fibration $\pi:\widehat{X}\rightarrow\mathbb{C}$ whose fibers admit higher dilations, then the total space $\widehat{X}$ also admits a higher-dilation. This implies, for example, any $(A_m)$ Milnor fiber of dimension $2n\geq6$ admits a higher dilation, since it admits a Lefschetz fibration with fiber symplectomorphic to $T^\ast S^{n-1}$. Zhao's argument can be extended to prove the following.

\begin{proposition}\label{proposition:Lefschetz}
Let $\pi:\widehat{X}\rightarrow\mathbb{C}^\ast$ be a Lefschetz fibration on some Liouville manifold of dimension $2n\geq4$, such that the monodromy around the origin is trivial. If the smooth fiber of $\pi$ admits a higher dilation, then so is the total space $\widehat{X}$.
\end{proposition}
\begin{proof}
Denote by $\widehat{F}$ a smooth fiber of $\pi$, which is a Liouville manifold. The same argument as in \cite{jz}, Section 5.3 produces a commutative diagram
\begin{equation}\label{eq:cd}
\begin{tikzcd}[font=\small]
\cdots \arrow[r] &\mathbb{R}^{\mathrm{Crit}(\pi)}[-n](\!(u)\!) \arrow[d,"\cong"] \arrow[r] &H^\ast(\widehat{X};\mathbb{R})(\!(u)\!) \arrow[d,"\phi_X^\mathrm{vert}"] \arrow[r] &H^\ast(\widehat{F};\mathbb{R})\otimes H^\ast(S^1;\mathbb{R})(\!(u)\!)\arrow[d,"\phi_F\otimes\mathrm{id}"] \arrow[r] &\cdots \\
\cdots \arrow[r] &\mathbb{R}^{\mathrm{Crit}(\pi)}[-n](\!(u)\!) \arrow[r] &\widehat{\mathit{PSH}}_\mathrm{vert}^\ast(X) \arrow[r] &\widehat{\mathit{PSH}}^\ast(F)\otimes H^\ast(S^1;\mathbb{R}) \arrow[r] &\cdots
\end{tikzcd}
\end{equation}
where $\widehat{\mathit{PSH}}_\mathrm{vert}^\ast(X)$ is a vertical version of the completed periodic symplectic cohomology, defined as the direct limit of the Floer cohomology of a sequence of Hamiltonians with increasing slopes in the fiber direction of $\pi$, and $\phi_X^\mathrm{vert}$ is the corresponding PSS map. There is a continuation map
\begin{equation}
\widehat{\mathit{PSH}}_\mathrm{vert}^\ast(X)\rightarrow\widehat{\mathit{PSH}}^\ast(X), \nonumber
\end{equation}
whose composition with $\phi_X^\mathrm{vert}$ gives the PSS map $\phi_X$ in (\ref{eq:PSS}). It follows from the assumption $n>1$ and the commutativity of (\ref{eq:cd}) that if the identity $1_F\in H^0(\widehat{F};\mathbb{R})$ vanishes under the PSS map $\phi_F$, then the identity $1_X\in H^0(\widehat{X};\mathbb{R})$ also vanishes under $\phi_X^\mathrm{vert}$. Thus $X$ also admits a higher dilation.
\end{proof}

As concrete examples, consider the affine $(A_m)$ Milnor fibers
\begin{equation}
\widehat{X}=\left\{(z_1,\cdots,z_n,z_{n+1})\in\mathbb{C}^n\times\mathbb{C}^\ast\left\vert z_1^2+\cdots+z_n^2=z_{n+1}^{m+1}-1\right.\right\}, \nonumber
\end{equation}
and let $X$ be any Liouville domain whose completion gives $\widehat{X}$. It follows from Proposition \ref{proposition:Lefschetz} that $C_d^\mathrm{GH}(X)<\infty$ for all $d\in\mathbb{N}$. These examples are interesting since $\pi_1(X)\neq0$.

\subsection{Lagrangian capacities of disc cotangent bundles}\label{section:capacity}

As examples of our discussions in Section \ref{section:dilation}, we explicitly compute the Lagrangian capacities $C^\mathrm{CM}(X)$ and $C^\mathrm{AL}(X)$ for the unit disc cotangent bundles over $S^2$, $S^3$, $\mathbb{RP}^2$ and Zoll spheres. As we will point out, the computational methods presented here actually extend to many other Liouville domains in dimensions $4$ and $6$.

\subsubsection{$S^2$ and $\mathbb{RP}^2$}

We first consider the case of $(D^\ast S^2,\lambda_\mathrm{can})$, where $\lambda_\mathrm{can}$ is the canonical Liouville form on the cotangent bundle. The contact boundary $S^\ast S^2:=\partial D^\ast S^2$ is $\mathbb{RP}^3$ equipped with the contact form $4\alpha_\mathrm{std}$, where $\alpha_\mathrm{std}$ denotes the standard contact form. Since the Reeb flow on $(\mathbb{RP}^3,4\alpha_\mathrm{std})$ is periodic, we can explicitly write down the orbits, see for example \cite{kvb}, Appendix A. In particular, all simple orbits $\gamma$ have symplectic action
\begin{equation}
\mathcal{A}(\gamma)=4\int_\gamma\alpha_\mathrm{std}=2\pi.
\end{equation}

By identifying $T^\ast S^2$ with the affine quadric surface
\begin{equation}
\widehat{X}=\left\{(z_1,z_2,z_3)|z_1^2+z_2^2+z_3^2=1\right\}
\end{equation}
equipped with the restriction of the standard Liouville form $\lambda_\mathrm{std}$ on $\mathbb{C}^3$, we get a standard Lefschetz fibration $p_X:\widehat{X}\rightarrow\mathbb{C}$ by projecting to the $z_1$ coordinate plane. Given any embedded closed curve $\sigma\subset\mathbb{C}$ away from the critical values $\pm1$ of $p_X$, we get a Lagrangian torus $T_\sigma\subset T^\ast S^2$ by parallel transporting the vanishing cycle. It is not hard to show that if we choose $\sigma$ to be a curve enclosing $\pm1$, then $T_\sigma\subset T^\ast S^2$ is a monotone Lagrangian torus. See for example \cite{lmt}, Section 2.2. 

To find the monotonicity constant of $T_\sigma$ (or equivalently, the symplectic area of a Maslov $2$ disc in $T^\ast S^2$ with boundary on $T_\sigma$), consider a parametrized curve
\begin{equation}
c(t)=\left(\sigma(t),n(t)e^{2\pi ia(t)},n(t)e^{2\pi ia(t)}\right),
\end{equation}
where $t\in[0,1]$, $n(t)>0$, $a(t)$ is a real-valued function satisfying $2n(t)^2e^{4\pi ia(t)}=1-\sigma(t)^2$. Then $c(t)\subset T_\sigma$ is the lift of $\sigma\subset\mathbb{C}$, and $p_X|_{c(t)}$ is a degree $1$ map. The symplectic area of any disc $u:(D,\partial D)\rightarrow(T^\ast S^2,T_\sigma)$ with boundary on $c(t)$ can be expressed as
\begin{equation}\label{eq:area1}
\int_Du^\ast d\lambda_\mathrm{can}=\frac{i}{4}\int_\sigma zd\bar{z}-\bar{z}dz+2\pi\int_0^1a'(t)n(t)^2dt,
\end{equation}
where the first term is the area enclosed by $\sigma\subset\mathbb{C}$, and the second term records the sum of areas in the other two coordinate projections.

Here we take $\sigma$ to be the ellipse with foci at $\pm1$ and eccentricity $\sqrt{\frac{\sqrt{5}-1}{2}}$, see Figure \ref{fig:base}. In this case, $T_\sigma\subset T^\ast S^2$ coincides with the Polterovich torus $T_\mathrm{Pol}$ (cf. \cite{afa}), which is constructed as the geodesic flow of unit covectors over the point $(1,0,0)\in S^2$. See \cite{lmt}, Proposition 2.10. By construction, $T_\mathrm{Pol}$ is contained in $D^\ast S^2$. In fact, the restriction of the Lefschetz fibration $p_X$ to $D^\ast S^2$ maps it to the filled ellipse, and $T_\mathrm{Pol}$ lies in the contact boundary $S^\ast S^2$. Let $u:(D,\partial D)\rightarrow(D^\ast S^2,T_\mathrm{Pol})$ be a $J$-holomorphic disc with boundary on the curve $c(t)\subset T_\mathrm{Pol}$ for some compatible almost complex structure $J$. Applying the formula (\ref{eq:area1}) we obtain
\begin{equation}\label{eq:area2}
\int_Du^\ast d\lambda_\mathrm{can}=\pi+\frac{\pi}{2}+\frac{\pi}{2}=2\pi,
\end{equation}
because the projections of $c(t)$ in the $z_2$ and $z_3$ coordinate planes are the same ellipse $\sigma(t)$ scaled by $\frac{1}{\sqrt{2}}$.

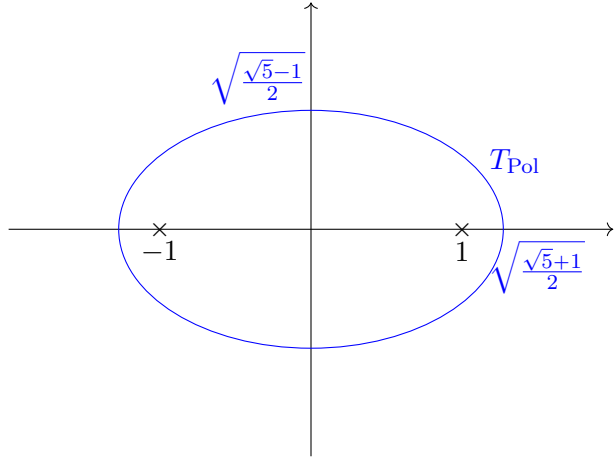
\begin{figure}
\centering
\begin{tikzpicture}
\draw[->] (-4, 0)--(4, 0);
\draw[->] (0,-3)--(0,3);
\draw [blue] (0,0) ellipse (2.544 and 1.572);
\draw (2,0) node {$\times$};
\draw (-2,0) node {$\times$};
\node at (2,-0.3) {$1$};
\node at (-2,-0.3) {$-1$};
\draw [blue] (2.7,0.9) node {$T_\mathrm{Pol}$};
\draw [blue] (3,-0.5) node {$\sqrt{\frac{\sqrt{5}+1}{2}}$};
\draw [blue] (-0.7,2) node {$\sqrt{\frac{\sqrt{5}-1}{2}}$};
\end{tikzpicture}
\caption{The image of the Polterovich torus $T_\mathrm{Pol}\subset D^\ast S^2$ under the Lefschetz fibration $p_X:T^\ast S^2\rightarrow\mathbb{C}$}\label{fig:base}
\end{figure}

\begin{proposition}\label{Lag-cap-S2}
For the unit disc cotangent bundle $D^\ast S^2$, we have
\begin{equation}
C^\mathrm{CM}(D^\ast S^2)=C^\mathrm{AL}(D^\ast S^2)=2\pi.
\end{equation}
\end{proposition}
\begin{proof}
By \cite{psd}, Example 2.5, a dilation in $\mathit{SH}^1(D^\ast S^2)$, whose action gives the first Gutt--Hutchings capacity $C_1^\mathrm{GH}(D^\ast S^2)$, corresponds to the shortest Reeb orbit on the contact boundary $(\mathbb{RP}^3,4\alpha_\mathrm{std})$. We can perturb the contact form $4\alpha_\mathrm{std}$ on $\mathbb{RP}^3$ to a non-degenerate one and keep the symplectic action $\mathcal{A}(\gamma)$ for the orbit $\gamma$ contributing to the first Gutt--Hutchings capacity $C_1^\mathrm{GH}(D^\ast S^2)$ arbitrarily close to $2\pi$. It follows from \cite{ghs}, Theorem 1.1 and our Theorem \ref{theorem:GH-ALbound} that
\begin{equation}\label{eq:1}
C^\mathrm{CM}(D^\ast S^2)\leq C^\mathrm{AL}(D^\ast S^2)\leq C_1^\mathrm{GH}(D^\ast S^2)=2\pi.
\end{equation}
On the other hand, it follows from (\ref{eq:area2}) that for the Polterovich torus $T_\mathrm{Pol}\subset D^\ast S^2$, $A_{\min}(T_\mathrm{Pol})=2\pi$, which gives
\begin{equation}\label{eq:2}
C^\mathrm{CM}(D^\ast S^2)\geq2\pi.
\end{equation}
Combining (\ref{eq:1}) and (\ref{eq:2}) finishes the proof.
\end{proof}

\begin{remark}
In fact, the equality $C^\mathrm{CM}(D^\ast S^2)=C^\mathrm{AL}(D^\ast S^2)$ also follows from the fact that any aspherical Lagrangian surface in $D^\ast S$ must be a torus. Since $D^\ast S^2$ is $(1,\Lambda)$-uniruled for some $\Lambda>0$ in the sense of \cite{mms}, it does not admit hyperbolic Lagrangians by the Viterbo--Eliashberg SFT argument. See \cite{egh}, Theorem 1.7.5. The same holds for $D^\ast\mathbb{RP}^2$ considered below, since $T^\ast\mathbb{RP}^2$ is symplectically equivalent to the complement of a smooth conic in $\mathbb{CP}^2$, which is known to be $\mathbb{A}^1$-uniruled.
\end{remark}

Our analysis above shows that the Polterovich torus $T_\mathrm{Pol}\subset D^\ast S^2$ is an extremal Lagrangian torus. In fact, we have the following.

\begin{proposition}
Any extremal Lagrangian torus $L\subset D^\ast S^2$ lies entirely on the boundary $S^\ast S^2$ and is non-displaceable. Furthermore, up to Hamiltonian isotopies preserving $S^\ast S^2$, the Polterovich torus $T_\mathrm{Pol}$ is the unique extremal Lagrangian torus.
\end{proposition}
\begin{proof}
There is the symplectic reduction
\[
R:S^*S^2\longrightarrow S^*S^2/S^1=(S^2,\omega),
\]
where
\[
\int_{S^2}\omega=4\pi .
\]
Since the Lagrangian capacity $C^\mathrm{CM}(D^\ast S^2)$ is already achieved by $C_1^\mathrm{GH}(D^\ast S^2)$, it follows from Theorem \ref{theorem:boundaryrigidity} that for any extremal torus $L\subset D^*S^2$, we have $L\subset S^*S^2$. Under the reduction map, $R(L)\subset S^2$ is an embedded loop which bounds a region of area $2\pi$. Any such loop can be mapped to the equator by a Hamiltonian isotopy $\phi_h^t:S^2\to S^2$ generated by $h:S^1\times S^2\to \mathbb{R}$.

Pulling back $h$ by the reduction map $R$, we obtain a Hamiltonian $h\circ(\operatorname{Id}\times R):S^1\times S^*S\to \mathbb{R}$. We extend this Hamiltonian to $T^*S\setminus S=\mathbb{R}\times S^*S $
by translations in the radial direction. After a smooth extension across the zero section, this gives a Hamiltonian $
\hat h:S^1\times T^*S\to \mathbb{R}$ whose flow preserves the hypersurfaces $S_r^*S$ for any $r\geq 1$. The time-$1$ flow of $\hat h$ sends $L$ to the lift of the equator under $R$.

Since $L$ is Hamiltonian isotopic to the Polterovich torus $T_\mathrm{Pol}\subset D^\ast S^2$, and it is shown in \cite{afa} that the Floer cohomology $\mathit{HF}^\ast(T_\mathrm{Pol},T_\mathrm{Pol})=H^\ast(T^2;\mathbb{Z}_2)$, it follows that $\mathit{HF}^\ast(L,L)\neq0$ (with $\mathbb{Z}_2$ coefficient) and $L$ is non-displaceable.
\end{proof}

Similar arguments can be used to compute the Lagrangian capacities for all $4$-dimensional $(A_m)$ Milnor fibers
\begin{equation}
A_{m,r}:=\left\{\left.(z_1,z_2,z_3)\in\mathbb{C}^3\right\vert z_1^{m+1}+z_2^2+z_3^2=\delta,|z_1|^2+|z_2|^2+|z_3|^2\leq r\right\},
\end{equation}
where $\delta>0$ is sufficiently small and $r>\sqrt{\delta}$. However, in this dimension it is unclear whether they admit higher dilations (or even cyclic dilations) if $m>1$, so there may not be a Reeb orbit on the contact boundary $\partial A_{m,r}\cong L(m+1,m)$ (equipped with $r(n+1)$ times the standard contact form) whose symplectic action coincides with the first Gutt--Hutchings capacity $C_1^\mathrm{GH}(A_{m,r})$.

We then compute the Lagrangian capacities of $(D^\ast\mathbb{RP}^2,\bar{\lambda}_\mathrm{can})$, where $\bar{\lambda}_\mathrm{can}$ is the induced Liouville form under the quotient map
\begin{equation}\label{eq:quo}
q:D^\ast S^2\rightarrow D^\ast\mathbb{RP}^2.
\end{equation}
First note that since $\mathbb{RP}^2$ is rationally inessential, $D^\ast\mathbb{RP}^2$ admits a higher dilation, therefore having finite Gutt-Hutchings capacities $C_d^\mathrm{GH}(D^\ast\mathbb{RP}^2)$ for $d\in\mathbb{N}$. The simple Reeb orbits on the contact boundary $S^\ast\mathbb{RP}^2$, which is diffeomorphic to the lens space $L(4,1)$ equipped with the contact form $4\bar{\alpha}_\mathrm{std}$, have period $\pi$, where $\bar{\alpha}_\mathrm{std}$ is the standard contact form on $L(4,1)$ coming from the quotient of $S^3$. However, these Reeb orbits are not contractible, therefore cannot be the cocycle in the definition of $C_1^\mathrm{GH}(D^\ast\mathbb{RP}^2)$ that kills the unit. Under the quotient map $q$, the the Floer cylinder $u:\mathbb{R}\times S^1\rightarrow D^\ast S^2$ with asymptotics $\beta_1$ and the constant orbit $e_X$ contributing to the equation $\delta(\beta_1)=e_X-\partial(\beta_0)$ in the symplectic cochain complex $\mathit{SC}^\ast(D^\ast\mathbb{RP}^2)$ now becomes a Floer cylinder in $D^\ast\mathbb{RP}^2$ that is asymptotic to the double cover of a simple Reeb orbit on $\left(L(4,1),4\bar{\alpha}_\mathrm{std}\right)$ at its positive puncture. It follows that
\begin{equation}\label{eq:C1}
C_1^\mathrm{GH}(D^\ast\mathbb{RP}^2)=2\pi.
\end{equation}

On the other hand, the Polterovich torus $T_\mathrm{Pol}\subset D^\ast S^2$ descends to a Lagrangian torus $T'_\mathrm{Pol}=q(T_\mathrm{Pol})\subset D^\ast\mathbb{RP}^2$ under the $\mathbb{Z}_2$-quotient.

\begin{lemma}\label{lemma:area}
Let $\bar{u}:(D,\partial D)\rightarrow(D^\ast\mathbb{RP}^2,T'_\mathrm{Pol})$ be a $J'$-holomorphic disc of Maslov index $2$ for some compatible almost complex structure $J'$ on $D^\ast\mathbb{RP}^2$ inherited from $D^\ast S^2$, then
\begin{equation}
\int_D\bar{u}^\ast d\bar{\lambda}_\mathrm{can}=2\pi.
\end{equation}
\end{lemma}
\begin{proof}
Any such $J'$-holomorphic disc admits a lift to a $J$-holomorphic disc $u:(D,\partial D)\rightarrow(D^\ast S^2,T_\mathrm{Pol})$ of Maslov index $2$ (cf. \cite{lmt}, Proposition 3.6), which gives $\bar{u}$ under the $\mathbb{Z}_2$-quotient. According to the previous discussion, we have $\int_Du^\ast d\lambda_\mathrm{can}=2\pi$. Since $q^\ast\bar{\lambda}_\mathrm{can}=\lambda_\mathrm{can}$, the quotient map (\ref{eq:quo}) preserves the symplectic area.
\end{proof}

Combining (\ref{eq:C1}), Theorem \ref{theorem:GH-ALbound} and Lemma \ref{lemma:area}, we obtain the following.

\begin{proposition}
For the unit disc cotangent bundle $D^\ast\mathbb{RP}^2$, we have
\begin{equation}
C^\mathrm{CM}(D^\ast\mathbb{RP}^2)=C^\mathrm{AL}(D^\ast\mathbb{RP}^2)=2\pi.
\end{equation}
\end{proposition}

Note that the Lagrangian capacities for $D^\ast S^2$ and $D^\ast\mathbb{RP}^2$ coincide with their Gromov width, see \cite{fr}, Theorem 1.1. Generalizing the example of $D^\ast\mathbb{RP}^2$, we have the rational homology balls $B_{p,q}$, where $p>q>0$ and $(p,q)=1$, which arise as the quotients of the $(A_{p-1})$ Milnor fiber by the cyclic group $\mathbb{Z}_p$. Assuming the $B_{p,q}$'s admit higher dilations\footnote{Note however that $2c_1(B_{p,q})$ is in general not $0$, therefore $\widehat{\mathit{PSH}}^\ast(B_{p,q})$ is not $\mathbb{Z}$-graded.}, then their Lagrangian capacities can be computed similarly as in the case of $D^\ast\mathbb{RP}^2$.

\subsubsection{Zoll spheres}

Proposition \ref{Lag-cap-S2} has an alternative proof based on the works \cite{fr,fr1} on the symplectic embedding into disc cotangent bundles, which can be extended to compute the Lagrangian capacities of a more general classes of Liouville domains, namely unit disc cotangent bundles of Zoll spheres of resolution.

Let $S\subset \mathbb{R}^3$ be a compact smooth surface of genus zero which is invariant under rotations about a fixed coordinate axis. We may write $S$ in the form
\[
S=\left\{(\rho(x)\cos\theta,\rho(x)\sin\theta,x)\vert x\in [a,b],\
\theta\in \mathbb{R}/2\pi\mathbb{Z}\right\},
\]
where $\rho:[a,b]\to[0,\infty)$ is a smooth function satisfying $\rho(a)=\rho(b)=0$. We say that $S$ is a Zoll sphere of revolution if all geodesics of the induced Riemannian metric are closed and have the same length.

\begin{proposition}\label{Lag-cap-S}
Let $S\subset \mathbb{R}^3$ be a Zoll sphere of revolution, and let $l$ denote
the length of any simple closed geodesic on $S$. Then
\[
C^\mathrm{CM}(D^\ast S)=C^\mathrm{AL}(D^\ast S)=l.
\]
In particular, for the unit round sphere $S=S^2$ this recovers Proposition \ref{Lag-cap-S2}.
\end{proposition}
\begin{proof}
Since $\pi_1(D^\ast S)=0$, for every Lagrangian torus $L\subset D^\ast S$, the long exact sequence for homotopy groups associated to the pair $(D^\ast S,L)$ gives rise to a short exact sequence
\[
0 \longrightarrow \pi_2(D^\ast S)\longrightarrow \pi_2(D^\ast S,L)\longrightarrow \pi_1(L)\longrightarrow 0.
\]
Let $\lambda_\mathrm{can}$ be the canonical Liouville form on $D^\ast S$. Since the symplectic form is exact, its integrals over the spherical classes
in $\pi_2(D^\ast S)$ vanish. It follows that
\[
A_{\min}(L)=\inf\left\{\left.\int_{\gamma}\lambda_{\mathrm{can}}>0\right\vert\gamma\in \pi_1(L)\right\}.
\]
Consequently, if the Liouville domain $(X,\omega)$ admits a symplectic embedding into $(D^\ast S,d\lambda_{\mathrm{can}})$, then
\begin{equation}\label{eq:cal}
C^{\mathrm{AL}}(X,\omega)\leq C^{\mathrm{AL}}(D^\ast S,d\lambda_{\mathrm{can}}).
\end{equation}
We emphasize that the capacity $C^{\mathrm{AL}}$ (or $C^\mathrm{CM}$) is in general not monotone under symplectic embeddings, but monotonicity holds, for instance, for symplectic embeddings $X\hookrightarrow Y$ satisfying $\pi_2(Y,X)=0$. Although this condition does not hold in the present situation, the exactness of the symplectic form on $D^\ast S$ allows the preceding argument to go through. By \cite{fr1}, Proposition 1.6 (see also \cite{fr}, Theorem 1.1, which deals with the case of a round sphere) there exists a symplectic embedding
\begin{equation}\label{eq:emb}
\left(B^2(l)\times B^2(l),\omega_{\mathrm{std}}\right)\hookrightarrow(D^\ast S,d\lambda_{\mathrm{can}}).
\end{equation}
On the other hand, by Corollary \ref{CM=AL} ($B^2(l)\times B^2(l)$ is an exact symplectic manifold with corners, but we can approximate it using convex toric domains),
\[
C^{\mathrm{AL}}\left(B^2(l)\times B^2(l)\right)=l.
\]
Therefore, it follows from (\ref{eq:cal}) and (\ref{eq:emb}) that
\[
C^{\mathrm{AL}}(D^\ast S,d\lambda_{\mathrm{can}})\geq l.
\]
Finally, the restriction of $\lambda_{\mathrm{can}}$ to the unit cosphere bundle $S^*S$ is Zoll, with all simple Reeb orbits having period $l$. By the classification of Zoll contact forms on $\mathbb{RP}^3$ in \cite{abhs}, Theorem B.2 we have
\[C_1^\mathrm{GH}(D^\ast S)=\frac{l}{2\pi}C_1^\mathrm{GH}(D^\ast S^2)=l.\]
Hence, by our Theorem \ref{theorem:GH-ALbound}, we have
\[
C^{\mathrm{AL}}(D^*S,d\lambda_{\mathrm{can}})\leq l.
\]
Combining the two inequalities, we conclude that
\[
C^{\mathrm{AL}}(D^*S,d\lambda_{\mathrm{can}})=l. \qedhere
\]
\end{proof}

\subsubsection{$3$-dimensional lens spaces}

Finally, we consider the higher-dimensional case and compute the Lagrangian capacity of $D^\ast L(p,q)$, the unit disc cotangent bundle over the $3$-dimensional lens space $L(p,q)$ equipped with the canonical Liouville form $\lambda_\mathrm{can}$. 

We start with the case of $D^\ast S^3$. First observe that the minimal period of simple Reeb orbits on the contact boundary $S^\ast S^3\cong S^2\times S^3$ is $2\pi$. Moreover, it is known that $D^\ast S^3$ admits a dilation and the dilation comes from a simple Reeb orbit, see \cite{psd}, Example 2.5, therefore $C_1^\mathrm{GH}(D^\ast S^3)=2\pi$. However, unlike the $2$-dimensional case, the higher Gutt--Hutchings capacities play a role here.

\begin{lemma}
We have $C_d^\mathrm{GH}(D^\ast S^3)=2\pi\lceil\frac{d}{2}\rceil$ for any $d\in\mathbb{N}$.
\end{lemma}
\begin{proof}
Let $\mathcal{L}^{<a}S^3\subset\mathcal{L}S^3$ be the subspace of loops $\gamma\in\mathcal{L}S^3$ with
\begin{equation}
\left(\int_0^1\left\vert\frac{\partial\gamma(t)}{\partial t}\right\vert^2dt\right)^{1/2}<a. \nonumber
\end{equation}
It follows from the definition that
\begin{equation}
C_d^\mathrm{GH}(D^\ast S^3)=\inf\left\{a\left\vert(i_a)_\ast\left([S^3]\otimes\beta_{d-1}\right)=0\in H_\ast^{S^1}(\mathcal{L}^{<a}S^3;\mathbb{R})\right\}\right., \nonumber
\end{equation}
where $i_a$ denotes the inclusion of constant loops and $\beta_d$ is the generator of $H_{2d}(\mathit{BS}^1;\mathbb{R})$. For $a\in(2\pi m,2\pi(m+1))$, we need to show that
\begin{equation}\label{eq:ker}
\ker(i_a)_\ast=\bigoplus_{j=0}^{2m-1}\textrm{span}_\mathbb{R}\left\langle[S^3]\otimes\beta_j\right\rangle.
\end{equation}

We first prove that $\ker(i_a)_\ast$ contains the right-hand side of (\ref{eq:ker}). Consider the Borel spectral sequence converging to the $S^1$-equivariant homology of $\mathcal{L}^{<a}S^3$, whose $E_2$ page is
\begin{equation}
E_{p,q}^2=H_p(\mathit{BS}^1;\mathbb{R})\otimes H_q(\mathcal{L}^{<a}S^3;\mathbb{R}), \nonumber
\end{equation}
with the differential given by
\begin{equation}
d_2(\beta_j\otimes y)=\beta_{j-1}\otimes\Delta y, \textrm{ }j\geq1, \nonumber
\end{equation}
where $\Delta$ is the BV operator. It follows from the computations of \cite{lms} that the BV operator hits all classes in $H_\ast(\mathcal{L}^{<a}S^3;\mathbb{R})$ with odd degrees, except for the point class $[\textrm{pt}]$ and the highest degree class, whose degree with the assumption that $a\in(2\pi m,2\pi(m+1))$ is given by $4m+3$. It follows that the remaining classes under degree $4m+3$ in the $E_3$ page of the spectral sequence must have even degrees, which in particular implies
\begin{equation}
(i_a)_\ast\left([S^3]\otimes\beta_j\right)=0\textrm{ for }j<2m. \nonumber
\end{equation}

Conversely, the relative $S^1$-equivariant homology group $H_\ast^{S^1}(\mathcal{L}^{<a}S^3,S^3)$ vanishes for $\ast>4m+2$, which one can easily see from Morse--Bott theory. An application of the $S^1$-equivariant long exact sequence associated to the pair $(\mathcal{L}^{<a}S^3,S^3)$ then shows that $(i_a)_\ast$ is injective for degrees $\geq 4m+3$, which completes the proof.
\end{proof}

\begin{corollary}\label{corollary:lens}
For any $3$-dimensional lens space $L(p,q)$, we have $C_d^\mathrm{GH}\left(D^\ast L(p,q)\right)=2\pi\lceil\frac{d}{2}\rceil$ for any $d\in\mathbb{N}$.
\end{corollary}
\begin{proof}
This follows from the compatibility of the $p$-fold covering map $S^3\rightarrow L(p,q)$ and its induced map $\mathcal{L}^{<a} S^3\rightarrow\mathcal{L}^{<a}L(p,q)$ on the free loop space with the inclusion of constant loops.
\end{proof}

In particular, we obtain the upper bound
\begin{equation}\label{eq:uppb}
C^\mathrm{CM}\left(D^\ast L(p,q)\right)\leq C^\mathrm{AL}\left(D^\ast L(p,q)\right)\leq\frac{C_2^\mathrm{GH}\left(D^\ast L(p,q)\right)}{2}\leq\pi.
\end{equation}

The cotangent bundle
\begin{equation}
T^\ast S^3=\left\{\left.(u,v)\in\mathbb{R}^4\times\mathbb{R}^4\right\vert |u|=1,\langle u,v\rangle=0\right\}
\end{equation}
with $\lambda_\mathrm{can}=\sum_{i=1}^4v_idu_i$, where $u=(u_1,u_2,u_3,u_4)$ and $v=(v_1,v_2,v_3,v_4)$ can in fact be realized as the smooth affine $3$-fold
\begin{equation}
\widehat{Y}=\left\{\left.(x_1,x_2,y_1,y_2,z)\in\mathbb{C}^5\right\vert x_1y_1=z+1,x_2y_2=z-1\right\}
\end{equation}
equipped with the restriction of the standard Liouville form $\lambda_\mathrm{std}$ on $\mathbb{C}^5$ up to a constant factor. To see that this is indeed the case, consider the change of variables
\begin{equation}\label{eq:wxy}
w_1=\frac{x_1+y_1}{2\sqrt{2}},\textrm{ }w_2=\frac{x_1-y_1}{2\sqrt{2}i},\textrm{ }w_3=\frac{x_2+y_2}{2\sqrt{2}i},\textrm{ }w_4=\frac{y_2-x_2}{2\sqrt{2}},
\end{equation}
under which  $\widehat{Y}\subset\mathbb{C}^5$ is identified with the affine quadric
\begin{equation}
\widehat{Z}=\left\{\left.(w_1,w_2,w_1,w_4)\in\mathbb{C}^4\right\vert w_1^2+w_2^2+w_3^2+w_4^2=1\right\}.
\end{equation}
An exact symplectomorphism $\widehat{Z}\rightarrow T^\ast S^3$ is given by the map
\begin{equation}
\begin{split}
&(w_1,w_2,w_3,w_4)\mapsto \\
&\left(\frac{\left(\mathrm{Re}(w_1),\mathrm{Re}(w_2),\mathrm{Re}(w_3),\mathrm{Re}(w_4)\right)}{\sqrt{\sum_{k=1}^4\mathrm{Re}(w_k)^2}},\sqrt{\sum_{k=1}^4\mathrm{Re}(w_k)^2}\cdot\left(\mathrm{Im}(w_1),\mathrm{Im}(w_2),\mathrm{Im}(w_3),\mathrm{Im}(w_4)\right)\right). \nonumber
\end{split}
\end{equation}
For the unit disc cotangent bundle $D^\ast S^3\subset T^\ast S^3$, we have $|v|=1$, so it follows that
\begin{equation}\label{eq:wuv}
w_k=\sqrt{\frac{\sqrt{5}+1}{2}}u_k+i\sqrt{\frac{\sqrt{5}-1}{2}}v_k,\textrm{ }k=1,2,3,4.
\end{equation}
Combining with (\ref{eq:wxy}), we can write down an explicit symplectomorphism $\widehat{Y}\rightarrow T^\ast S^3$, with
\begin{equation}
x_1=\sqrt{\sqrt{5}+1}(u_1+iu_2)+i\sqrt{\sqrt{5}-1}(v_1+iv_2),
\end{equation}
\begin{equation}
y_1=\sqrt{\sqrt{5}+1}(u_1-iu_2)+i\sqrt{\sqrt{5}-1}(v_1-iv_2),
\end{equation}
\begin{equation}
x_2=\sqrt{\sqrt{5}+1}(iu_3-u_4)+\sqrt{\sqrt{5}-1}(iv_4-v_3),
\end{equation}
\begin{equation}
y_2=\sqrt{\sqrt{5}+1}(iu_3+u_4)+\sqrt{\sqrt{5}-1}(v_3+iv_4).
\end{equation}
In particular,
\begin{equation}
\begin{split}
z=\frac{x_1y_1+x_2y_2}{2}&=\frac{\sqrt{5}+1}{2}(u_1^2+u_2^2-u_3^2-u_4^2)-\frac{\sqrt{5}-1}{2}(v_1^2+v_2^2-v_3^2-v_4^2) \\
&+2i(u_1v_1+u_2v_2-u_3v_3-u_4v_4).
\end{split}
\end{equation}
Using the facts $|u|=1$, $|v|=1$ and $\langle u,v\rangle=0$, direct computation yields
\begin{equation}\label{eq:ell}
\frac{\mathrm{Re}(z)^2}{5}+\frac{\mathrm{Im}(z)^2}{4}=1.
\end{equation}
In other words, under the projection $p_Y:\widehat{Y}\rightarrow\mathbb{C}$ to the $z$-coordinate plane, the unit disc cotangent bundle $D^\ast S^3\subset T^\ast S^3$ maps to the filled ellipse with foci $\pm1$, major semiaxis $\sqrt{5}$ and minor semiaxis $2$, see Figure \ref{fig:MBf}.

\begin{figure}
	\centering
	\begin{tikzpicture}
	\draw [blue] (0,0) ellipse (2.236 and 2);
	\draw [->] (-3.5,0)--(3.5,0);
	\draw [->] (0,-3)--(0,3);
	\draw (1,0) node {$\times$};
	\draw (-1,0) node {$\times$};
	\draw (-1,-0.3) node {$-1$};
	\draw (1,-0.3) node {$1$};
	\draw [blue] (2.5,-0.3) node {$\sqrt{5}$};
	\draw [blue] (0.2,2.2) node {$2$};
	\draw [blue] (2.3,1.2) node {$T_\mathrm{Pol}$};
	\end{tikzpicture}
	\caption{The image of the generalized Polterovich torus $T_\mathrm{Pol}\subset D^\ast S^3$ under the Morse-Bott fibration $p_Y:T^\ast S^3\rightarrow\mathbb{C}$}\label{fig:MBf}
\end{figure}
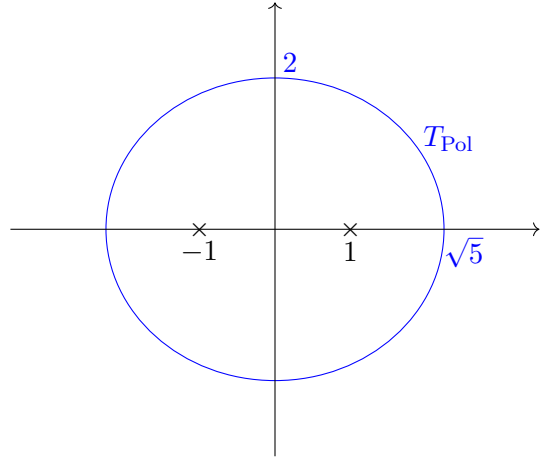

The projection $p_Y:\widehat{Y}\rightarrow\mathbb{C}$ to the $z$-coordinate plane is in fact a Morse-Bott fibration, with smooth fiber symplectomorphic to $T^\ast T^2$, and two singular fibers at $\pm1$ isomorphic to $(\mathbb{C}\vee\mathbb{C})\times\mathbb{C}^\ast$. If we fix a global basis $\alpha,\beta\in H_1(T^2;\mathbb{Z})$, the vanishing cycles at $-1$ and $+1$ correspond to $\alpha$ and $\beta$, respectively. It follows from \cite{swd}, Lemma 4.1 that the symplectic parallel transport is well-defined for $p_Y$ away from the critical values. Take $\sigma\subset\mathbb{C}$ to be a closed curve encircling $\pm1$. Parallel transporting the zero section $T^2$ in the smooth fibers of $p_Y$ along $\sigma$ we obtain a Lagrangian $3$-torus $T_\sigma\subset T^\ast S^3$, which is actually monotone by \cite{adm}, Lemma 2.12. We denote by $T_\mathrm{Pol}\subset D^\ast S^3$ the monotone Lagrangian torus corresponding to the particular choice when $\sigma$ is the ellipse (\ref{eq:ell}), and call it the \textit{generalized Polterovich torus}.

\begin{lemma}\label{lemma:area1}
Let $u:(D,\partial D)\rightarrow(D^\ast S^3,T_\mathrm{Pol})$ be a $J$-holomorphic disc of Maslov index $2$ for some compatible almost complex structure $J$, then
\begin{equation}
\int_{D}u^\ast d\lambda_\mathrm{can}=\pi. \nonumber
\end{equation}
\end{lemma}
\begin{proof}
Under the Euclidean metric on the $z$-coordinate plane $\mathbb{C}_z$, the area bounded by the ellipse (\ref{eq:ell}) is $2\sqrt{5}\pi$. However, unlike the situation of $D^\ast S^2$, the standard metric on $\mathbb{C}_z$ rescales the restriction of the canonical symplectic form $d\lambda_\mathrm{can}$ on $D^\ast S^3$ by a constant factor. To find this factor, we compute
\begin{equation}
\begin{split}
d\mathrm{Re}(z)&=(\sqrt{5}+1)(u_1du_1+u_2du_2-u_3du_3-u_4du_4) \\
&-(\sqrt{5}-1)(v_1dv_1+v_2dv_2-v_3dv_3-v_4dv_4), \nonumber
\end{split}
\end{equation}
\begin{equation}
\begin{split}
d\mathrm{Im}(z)&=2(u_1dv_1+v_1du_1)+2(u_2dv_2+v_2du_2)-2(u_3dv_3+v_3du_3) \\
&-2(u_4dv_4+v_4du_4), \nonumber
\end{split}
\end{equation}
which imply that
\begin{equation}\label{eq:coeff1}
d\mathrm{Re}(z)\wedge d\mathrm{Im}(z)=-4\sqrt{5}\sum_{k=1}^4dv_k\wedge du_k=-4\sqrt{5}d\lambda_\mathrm{can}|_{\mathbb{C}_z},
\end{equation}
where the expression on the right-hand side is simplified using $|u|=|v|=1$ and $\langle u,v\rangle=0$. On the other hand, from (\ref{eq:wxy}) and (\ref{eq:wuv}), it is easy to find that
\begin{equation}\label{eq:coeff2}
d\mathrm{Re}(\zeta)\wedge d\mathrm{Im}(\zeta)=-4d\lambda_\mathrm{can}|_{\mathbb{C}_\zeta}
\end{equation}
on the $\zeta$-coordinate plane, where $\zeta$ stands for any one of the variables $x_1,y_1,x_2,y_2$.

Given a parametrized curve
\begin{equation}
c(t)=\left(n_1(t)e^{2\pi ia(t)},n_1(t),n_2(t)e^{2\pi ib(t)},n_2(t),\sigma(t)\right),
\end{equation}
where $t\in[0,1]$, $n_1(t)>0$, $n_2(t)>0$, $a(t)$ and $b(t)$ are real-valued functions satisfying
\begin{equation}\label{eq:cons}
n_1(t)^2e^{2\pi i a(t)}=\sigma(t)+1,\textrm{ }n_2(t)^2e^{2\pi ib(t)}=\sigma(t)-1.
\end{equation} 
It is straightforward to verify that $c(t)$ lies on the generalized Polterovich torus $T_\mathrm{Pol}$ that projects to the ellipse (\ref{eq:ell}), which is denoted by $\sigma(t)$ here, there is a formula similar to (\ref{eq:area1}) computing the symplectic area of a smooth disc $u:(D,\partial D)\rightarrow(D^\ast S^3,T_\mathrm{Pol})$ with boundary on $c(t)\subset T_\mathrm{Pol}$, which reads
\begin{equation}\label{eq:area3}
\int_{D}u^\ast d\lambda_\mathrm{can}=\frac{i}{16\sqrt{5}}\int_\sigma(zd\bar{z}-\bar{z}dz)+\frac{\pi}{4}\int_0^1a'(t)n_1(t)^2dt+\frac{\pi}{4}\int_0^1b'(t)n_2(t)^2dt,
\end{equation}
where the first term on the right-hand side is the area bounded by the ellipse $\sigma(t)$, the second term is the area bounded by the projection of $c(t)$ to the $x_1$ coordinate plane, and the third term is the area contributions from the projection of $c(t)$ to the $x_2$ coordinate projections. The coefficients before these terms are rescaled according to our computations in (\ref{eq:coeff1}) and (\ref{eq:coeff2}). 

To complete the computation, we need to find the area bounded by the curve $p_{x_1}(c(t))\subset\mathbb{C}$, where $p_{x_1}:\widehat{Y}\rightarrow\mathbb{C}$ is the projection to the $x_1$ coordinate plane. The case of the projected curve $p_{x_2}(c(t))\subset\mathbb{C}$ is similar. From (\ref{eq:cons}) one finds that the curve $p_{x_1}(c(t))$ is parametrized by the equation
\begin{equation}
w(\theta)=\frac{2e^{i\theta}}{\sqrt{\sqrt{5}+\cos\theta}},\textrm{ }\theta\in[0,2\pi), \nonumber
\end{equation}
which is a smooth oval but not an ellipse. The area it bounds is given by
\begin{equation}
\frac{1}{2}\int_0^{2\pi}r^2d\theta=2\int_0^{2\pi}\frac{d\theta}{\sqrt{5}+\cos\theta}=2\pi.
\end{equation}
After rescaling, we get from (\ref{eq:area3}) that
\[
\int_{D}u^\ast d\lambda_\mathrm{can}=\frac{\pi}{2}+\frac{\pi}{4}+\frac{\pi}{4}=\pi. \nonumber \qedhere
\]
\end{proof}

Lemma \ref{lemma:area1} shows that $A_{\min}(T_\mathrm{Pol})=\pi$. Combining with (\ref{eq:uppb}) we obtain the following.

\begin{proposition}\label{Lag-cap-S3}
For the unit disc cotangent bundle $D^\ast S^3$, we have
\begin{equation}
C^\mathrm{CM}(D^\ast S^3)=C^\mathrm{AL}(D^\ast S^3)=\pi. \nonumber
\end{equation}
\end{proposition}

It is clear that the generalized Polterovich torus $T_\mathrm{Pol}\subset D^\ast S^3$ descends to a monotone Lagrangian torus $T_\mathrm{Pol}'\subset D^\ast L(p,q)$ under the quotient of the $\mathbb{Z}_p$-action
\begin{equation}
(x_1,y_1,x_2,y_2,z)\mapsto\left(e^{\frac{2\pi i}{p}}x_1,e^{-\frac{2\pi i}{p}}y_1,e^{\frac{2\pi iq}{p}}x_2,e^{-\frac{2\pi iq}{p}}y_2,z\right) \nonumber
\end{equation}
on the affine $3$-fold $\widehat{Y}\cong T^\ast S^3$, where $(p,q)=1$. In fact, when $q=1$, $T_\mathrm{Pol}'$ can also be obtained by parallel transporting the zero section $T^2\subset T^\ast T^2$ over the same ellipse (\ref{eq:ell}) of a similar Morse-Bott fibration $p_Y':T^\ast L(p,1)\rightarrow\mathbb{C}$. The critical values of $p_Y'$ are still $\pm1$, but the vanishing cycles are now given by $\alpha$ and $\alpha\pm p\beta$, respectively. By the same argument as in Lemma \ref{lemma:area}, one can show that $A_{\min}(T_\mathrm{Pol}')=\pi$.

Combining with Corollary \ref{corollary:lens}, we have proved the following.

\begin{proposition}
For the unit disc cotangent bundle $D^\ast L(p,q)$, we have
\begin{equation}
C^\mathrm{CM}\left(D^\ast L(p,q)\right)=C^\mathrm{AL}\left(D^\ast L(p,q)\right)=\pi. \nonumber
\end{equation}
\end{proposition}

More generally, one can try to compute the Lagrangian capacities of the double-bubble or even multi-bubble plumbings studied in \cite{swd} and \cite{xlp}. These are Weinstein $6$-manifolds carrying similar Morse-Bott fibrations as $T^\ast S^3$, therefore monotone Lagrangian tori can be constructed in the corresponding Weinstein domains by parallel transporting the $T^2$ vanishing cycles. However, the existence of a higher-dilation is in general unclear for these Weinstein $6$-manifolds, despite the partial results obtained in \cite{alp,gpa,swd}.

\Addresses

\end{document}